\documentclass[12pt,a4paper]{amsart}

\usepackage{longtable}
\usepackage{tikz}

\usepackage{latexsym}
\usepackage{amssymb}
\usepackage{amsmath}
\usepackage{amsthm}
\usepackage{mathrsfs}
\usepackage{color}

\makeatletter
\def\@settitle{\begin{flushleft}%
  \baselineskip14\p@\relax
    \normalfont\Large\bf%<- NEW
  \@title
  \end{flushleft}%
}
\def\section{\@startsection{section}{1}%
  \z@{.7\linespacing\@plus\linespacing}{.5\linespacing}%
  {\normalfont\bf}}

\def\@setauthors{%
  \begingroup
  \def\thanks{\protect\thanks@warning}%
  \trivlist
  \footnotesize \@topsep30\p@\relax
  \advance\@topsep by -\baselineskip
  \item\relax
  \author@andify\authors
  \def\\{\protect\linebreak}%
  \larger\sc\authors%
  \ifx\@empty\contribs
  \else
    ,\penalty-3 \space \@setcontribs
    \@closetoccontribs
  \fi
  \endtrivlist
  \endgroup
}

\makeatother

\usepackage{exscale}
\usepackage[english]{babel}
\usepackage{verbatim}
\usepackage{fullpage}
\usepackage{fix-cm}
\usepackage{enumitem}
\usepackage[colorlinks, breaklinks, allcolors=blue]{hyperref} 

\swapnumbers
\theoremstyle{plain}
\newtheorem{theorem}{Theorem}[section]
\newtheorem{lemma}[theorem]{Lemma}
\newtheorem{proposition}[theorem]{Proposition}
\newtheorem{corollary}[theorem]{Corollary}

\theoremstyle{definition}
\newtheorem{definition}[theorem]{Definition}
\newtheorem{remark}[theorem]{Remark}
\newtheorem{example}[theorem]{Example}

\newcommand{\CC}{\mathbb{C}}
\newcommand{\ZZ}{\mathbb{Z}}
\newcommand{\QQ}{\mathbb{Q}}

\newcommand{\NN}{\mathbb{N}}

\newcommand{\fg}{\mathfrak{g}}

\DeclareMathOperator{\Hom}{Hom}

\DeclareMathOperator{\End}{End}
\DeclareMathOperator{\Span}{span}

\newcommand{\Chi}{{\mathcal X}}

\newcommand{\A}{{\mathbf A}}
\newcommand{\V}{{\mathcal V}}

\DeclareMathOperator{\rk}{rk}
\newcommand{\SL}{\mathrm{SL}}
\newcommand{\GL}{\mathrm{GL}}
\newcommand{\Sp}{\mathrm{Sp}}
\newcommand{\SO}{\mathrm{SO}}
\newcommand{\Spin}{\mathrm{Spin}}
\newcommand{\PGL}{\mathrm{PGL}}
\newcommand{\PSO}{\mathrm{PSO}}
\newcommand{\inv}{^{-1}}
\newcommand{\<}{\langle}
\renewcommand{\>}{\rangle}

\newcommand{\laso}{\mathfrak{so}}

\newcommand{\bbC}{\mathbb{C}}

\newcommand{\bbZ}{\mathbb{Z}}

\newcommand{\lag}{\mathfrak{g}}
\newcommand{\lah}{\mathfrak{h}}

\newcommand{\lap}{\mathfrak{p}}

\newcommand{\calB}{\mathcal{B}}

\newcommand{\calX}{\mathcal{X}}

\newcommand{\Ind}{\mathrm{Ind}}

\newcommand{\eps}{\epsilon}
\newcommand{\diag}{\mathrm{diag}}

\usepackage{parskip}
\usepackage{microtype}

\AtBeginDocument{%
}

\begin{document}
\title{On the extended weight monoid and its applications to orthogonal polynomials}
\author[GP]{Guido Pezzini}
\author[MvP]{Maarten van Pruijssen}
\email[GP]{pezzini@mat.uniroma1.it}
\email[MvP]{m.vanpruijssen@math.ru.nl}

\begin{abstract}
Given a connected simply connected semisimple group $G$ and a connected spherical subgroup $K\subseteq G$ we determine the generators of the extended weight monoid of $G/K$, based on the homogeneous spherical datum of $G/K$.

Let $H\subseteq G$ be a reductive subgroup and let $P\subseteq H$ be a parabolic subgroup for which $G/P$ is spherical. A triple $(G,H,P)$ with this property is called multiplicity free system and we determine the generators of the extended weight monoid of $G/P$ explicitly in the cases where $(G,H)$ is strictly indecomposable.

The extended weight monoid of $G/P$ describes the induction from $H$ to $G$ of an irreducible $H$-representation $\pi:H\to\GL(V)$ whose lowest weight is a character of $P$.
The space of regular $\End(V)$-valued functions on $G$ that satisfy $F(h_{1}gh_{2})=\pi(h_{1})F(g)\pi(h_{2})$ for all $h_{1},h_{2}\in H$ and all $g\in G$, is a module over the algebra of $H$-biinvariant regular functions on $G$. We show that under a mild assumption this module is freely and finitely generated. As a consequence the spherical functions of such a type $\pi$ can be described as a family of matrix-valued orthogonal polynomials.
\end{abstract}

\maketitle

\tableofcontents

%%%%%%%%%%%%%%%%%%%%%%%%%%%%%%%%
%%%%%                                            %%%%%%%%%%%%%%%
%%%%%          NEW SECTION          %%%%%%%%%%%%%%%
%%%%%                                           %%%%%%%%%%%%%%%
%%%%%%%%%%%%%%%%%%%%%%%%%%%%%%%%

\section{Introduction}\label{section:Introduction}

\subsection{Context}

Let $G$ be a connected simply connected semisimple group over the field of complex numbers $\bbC$. The interpretation of the zonal spherical functions on a symmetric space $G/H$ as Heckman-Opdam polynomials with geometric parameters establishes an important connection between the representation theory of complex reductive groups on the one hand and orthogonal polynomials on the other. This interpretation is roughly as follows. The zonal spherical functions are matrix coefficients of spherical representations, i.e.~of irreducible representations of $G$ with an $H$-fixed vector. The polynomial nature of these functions can be explained by the decomposition of tensor products of appropriate $G$-representations. The orthogonality comes from Schur-orthogonality and the algebra of differential operators that determines the Heckman-Opdam polynomials up to scaling as simultaneous eigenfunctions is a subquotient of the universal enveloping algebra of $\lag=\mathrm{Lie}(G)$ that is understood via the Harish-Chandra isomorphism as described by Heckman in \cite[Ch.5]{MR1313912}.

The heuristic argument behind this nice description of the zonal spherical functions is the multiplicity freeness of the induction of the trivial $H$-representation to $G$. It accounts for two important facts, namely that the space of $H$-biinvariant regular functions on $G$ is spanned by the zonal spherical functions, and that the algebra of differential operators mentioned above is commutative.

In the context of multiplicity free systems, a notion we recall below, there are several examples in which more general spherical functions are described as polynomials, see e.g.~\cite{MR1883412,MR783984,MR3550926,MR3006173,MR3085105,MR4053617,MR3801483}.
In these cases, instead of the trivial $H$-representation, another irreducible $H$-representation with suitable properties is considered. The polynomials are now vector-valued orthogonal polynomials and in the examples they can be grouped into matrix-valued orthogonal polynomials. As such, they enjoy properties similar to those of the Heckman-Opdam polynomials. 
For example, the polynomials are simultaneous eigenfunctions of a commutative algebra of differential operators. In the case that $(G,H)$ is symmetric, the polynomials are even determined by this property and we obtain a system of hypergeometric differential equations in the sense of Heckman \cite[Definition 4.2]{MR1313912}.

One of the earliest examples of these vector-valued orthogonal polynomials was found by Koornwinder \cite{MR783984}, related to the pair $(\SL(2)\times\SL(2),\diag(\SL(2)))$. Gr\"unbaum, Pacharoni and Tirao found examples related to the pair $(\SL(n+1),\GL(n))$ in \cite{MR1883412} and subsequent papers, essentially by investigating differential operators. Koelink, van Pruijssen and Rom\'an \cite{MR3550926,MR3006173} added, among other things, the differential properties to Koornwinder's original class of examples. Heckman and van Pruijssen generalized both classes of examples to a general construction of vector-valued orthogonal polynomials for compact Gelfand pairs of rank one.
The existence of multivariable vector-valued orthogonal polynomials with properties similar to the Heckman-Opdam polynomials, was shown by van Pruijssen \cite{MR3801483}, after which a concrete example was worked out by Koelink, van Pruijssen and Rom\'an \cite{MR4053617}.

The main difficulty for the vector-valued polynomials is the determination of all irreducible $G$-representations that contain an appropriate irreducible $H$-representation upon restriction. This problem can also be thought of as inverting the branching law from $G$ to $H$ for appropriate irreducible $H$-representations. 

In this paper we solve this problem in the generality of multiplicity free systems and for each strictly indecomposable multiplicity free system we provide an explicit solution for the inverse branching law for the appropriate irreducible $H$-representations. Moreover, we show that under mild conditions the spherical functions can be described by families of vector-valued orthogonal polynomials, possibly in several variables, and that they can be grouped into families of matrix-valued orthogonal polynomials. To establish these results we use spherical varieties and in particular the combinatorics of homogeneous spherical data.

\subsection{Results}
We fix a maximal torus and a Borel subgroup $T_{G}\subseteq B_{G}\subseteq G$, and let $K$ be a closed subgroup of $G$. We assume that $K$ is connected and {\em spherical}, i.e., a Borel subgroup of $G$ has an open orbit on $G/K$. The definition of a spherical variety along with some facts, notations and conventions will be recalled in Sections \ref{Section: Notations} and \ref{Section: Spherical varieties}. We will also recall the notion of {\em homogeneous spherical data}, which are combinatorial objects used in the classification of such varieties.

The {\em extended weight monoid} $\widetilde{\Gamma}(G/K)$ of the homogeneous space $G/K$ is the set of all pairs $(\omega,\chi)$, where $\omega$ is a character of $B_G$ and $\chi$ is a character of $K$, for which there exists a non-zero function $f\in\bbC[G]$ satisfying
$$f(b\inv gp)=\omega(b)f(g)\chi(p),\quad\mbox{for all $(b,g,p)\in B_{G}\times G\times K$},$$
see Definition \ref{def: EWS}. The assumption that $G$ be simply connected implies that $\widetilde{\Gamma}(G/K)$ is a freely and finitely generated monoid (see Proposition~\ref{prop:generators} below).

Our \textbf{first main result} is Theorem \ref{thm:howto}, where we give a general combinatorial procedure to derive the generators of $\widetilde{\Gamma}(G/K)$ from the homogeneous spherical datum of $G/K$. Computations of extended weight monoids are found in the literature for particular spherical subgroups, and carried out with ad-hoc methods. Reductive ones are dealt with by Kr\"amer~\cite{MR528837} (for $G$ simple) and by Avdeev~\cite{MR2779106} (for $G$ not simple); the approach is case-by-case using representation theory and explicit computations. The case of $K$ solvable is found in Avdeev~\cite{MR3366923}.

While this paper was being finalized, Avdeev's preprint~\cite{avdeev-preprint} appeared, where another way to compute the extended weight monoid of spherical homogeneous spaces $G/K$ in general is proposed. It is based on a description of $K$ as a subgroup of a suitably chosen parabolic subgroup $P$ of $G$. There is some overlap with our techniques: in particular, we compute in the same way some generators coming easily from the natural morphism $G/K\to G/P$. For the other generators our approach is different, and uses only the homogeneous spherical datum of $G/K$.

Let now $H\subseteq G$ be a reductive subgroup and $P\subseteq H$ a parabolic subgroup. The triple $(G,H,P)$ is called a {\em multiplicity free system} if $P$ is a spherical subgroup of $G$. This notion is related to the representation theory of $G$ and $H$ in the following way. We choose a maximal torus and a Borel subgroup $T_{H}\subseteq B_{H}\subseteq H$ in a compatible way, meaning that $B_{H}\subseteq B_{G}$ and $T_{H}\subseteq T_{G}$. The opposite Borel subgroups with respect to $T_{G}$ and $T_{H}$ are denoted by $B_{G}^{-}$ and $B_{H}^{-}$ respectively. Furthermore we arrange $B_{H}^{-}\subseteq P$. Let $\calX(T_{G})$ denote the character group of $T_{G}$ and denote by $\calX_{+}(T_{G})$ the monoid of dominant characters with respect to $B_{G}$. We recall that they are also called dominant weights and they parametrize the equivalence classes of the irreducible $G$-representations, by taking the highest weight appearing in the representation. In the same way the dominant characters in $\calX_{+}(T_{H})\subseteq\calX(T_{H})$ parametrize the equivalence classes of irreducible $H$-representations. Let $\calX(P)\subseteq\calX(T_{H})$ denote the group of characters of $T_{H}$ that can be extended to a character of $P$. If $\chi\in\calX_{+}(T_{H})$ then we denote the corresponding irreducible $H$-representation of highest weight $\chi$ by $\pi_{H,\chi}:H\to\GL(V_{H,\chi})$.

Let $\chi\in\calX_{+}(T_{H})\cap\calX(P)$. Following the Borel-Weil Theorem we see that $\Ind^{H}_{P}(\chi)$ is an irreducible $H$-representation of lowest weight $w_{0}^{H}(\chi)$, where $w_{0}^{H}\in W_{H}=N_{H}(T_{H})/T_{H}$ is the longest Weyl group element, and hence $\Ind^{H}_{P}(\chi)=\pi_{H,\chi}$. Induction in stages yields $\Ind^{G}_{H}(\pi_{H,\chi})=\Ind^{G}_{P}(\chi)$ and the latter space is a $G$-representation whose decomposition into irreducible $G$-representations is multiplicity free. Indeed, $\Ind^{G}_{P}(\chi)$ is the space of regular sections of a $G$-line bundle over the space $G/P$. Since $G/P$ is a spherical variety, the decomposition is multiplicity free, see e.g.~\cite[Thm.25.1]{MR2797018}. This means in particular that $\dim\Hom_{H}(V_{H,\chi},V_{G,\lambda})\le1$ for all $\lambda\in\calX_{+}(T_{G})$. In this context it is natural to study the set $\Chi(T_{G};T_{H},\chi):=\{\lambda\in\calX_{+}(T_{G}):\dim\Hom_{H}(V_{H,\chi},V_{G,\lambda})=1\}$, which we call the $\chi$-well. It is directly related to the extended weight monoid, indeed $\lambda\in\Chi(T_{G};T_{H},\chi)$ if and only if $(\lambda,-\chi)\in\widetilde{\Gamma}(G/P)$. 

If $(G,H,P)$ is a multiplicity free system with $G$ not simply connected and $\chi\in\calX_{+}(T_{H})\cap\calX(P)$, then $\Chi(T_{G};T_{H},\chi)$ can be obtained from $\Chi(T_{\widetilde{G}};T_{\widetilde{H}},\widetilde{\chi})$, where $\psi:\widetilde{G}\to G$ is the simply connected cover, $\widetilde{H}$ is the connected component of $\psi^{-1}(H)$ containing the identity and $\widetilde{\chi}$ is the highest weight of the irreducible representation $\pi^{H}_{\chi}\circ(\psi|_{\widetilde{H}})$. Indeed, $\Chi(T_{G};T_{H},\chi)$ consists of weights of $\Chi(T_{\widetilde{G}};T_{\widetilde{H}},\widetilde{\chi})$ for which the corresponding irreducible representation descends to an irreducible representation of $G$. Hence it is not a restriction to assume $G$ simply connected.

A multiplicity free system $(G,H,P)$ is called {\em strictly indecomposable} if $(G,H)$ is strictly indecomposable, meaning that pair $(G,[H,H])$ is not isogenous to a product $(G_{1}\times G_{2},H_{1}\times H_{2})$ with $H_{i}\subsetneq G_{i}$ non-trivial. The strictly indecomposable multiplicity free systems $(G,H,P)$ with $G$ semisimple are classified by He, Nishiyama, Ochiai, and Oshima in \cite{MR3127988} for $(G,H)$ symmetric and by van Pruijssen in \cite{MR3801483} for $(G,H)$ spherical but not symmetric. Our \textbf{second main result} is the calculation of the generators for each strictly indecomposable multiplicity free system $(G,H,P)$ with $G$ semisimple and simply connected, in Appendices~\ref{s:sym} and~\ref{s:sph}.

In Section \ref{s:ops} we recall the notion of a zonal spherical function and of a spherical function of type $\chi\in\calX_{+}(T_{H})\cap\calX(P)$ for a multiplicity free system $(G,H,P)$. The zonal spherical functions constitute a basis of the vector space $E^{0}:=\bbC[G]^{H\times H}$ of $H$-biinvariant regular functions on $G$. Similarly, the spherical functions of type $\chi$ constitute a basis of the vector space
$$E^{\chi}:=\left(\bbC[G]\otimes\End(V_{H,\chi})\right)^{H\times H}.$$
Let $(G,H,P)$ be a strictly indecomposable multiplicity free system with $G$ simply connected and assume that $E^{0}$ is a polynomial algebra. This is a mild assumption, because it is satisfied in all cases $(G,H,P)$ where $P\subset H$ is proper. Under this assumption, our \textbf{third main result} is that $E^{\chi}$, with $\chi\in\calX_{+}(T_{H})\cap\calX(P)$, is freely and finitely generated as module over $E^{0}$. As a consequence the spherical functions of type $\chi$ can be described as a family of vector-valued orthogonal polynomials. Moreover, these polynomials can be grouped together in a family of matrix-valued orthogonal polynomials. As such, the families of polynomials fit into the framework of \cite{MR4053617}, where also the differential properties are discussed and where the example $(\SL(n)\times\SL(n),\diag(\SL(n)),\diag(P))$, with $P\subseteq\SL(n)$ a parabolic with Levi subgroup isomorphic to $\GL(n-1)$, is discussed in great detail. 

\subsection{Acknowledgements}
We thank Erik Koelink and Friedrich Knop for useful remarks and corrections on a previous version of the paper.

%%%%%%%%%%%%%%%%%%%%%%%%%%%%%%%%
%%%%%                                            %%%%%%%%%%%%%%%
%%%%%          NEW SECTION          %%%%%%%%%%%%%%%
%%%%%                                           %%%%%%%%%%%%%%%
%%%%%%%%%%%%%%%%%%%%%%%%%%%%%%%%  

\section{Notations}\label{Section: Notations}

We fix a simply connected semisimple group $G$ over $\CC$. As above, we fix a choice of a Borel subgroup $B_G\subseteq G$ and a maximal torus $T_G\subseteq B_G$, we denote by $S_G$ the corresponding set of simple roots, by $B^-_G$ the opposite Borel subgroup of $B$, and by $W_G$ the Weyl group of $G$. When no confusion arises, for simplicity they will be also denoted by $B$, $T$, $S$, $B^-$, $W$ respectively.

In general, if $K$ is an algebraic group, the group of its characters is denoted by $\Chi(K)$, and the connected component of $K$ containing the neutral element by $K^\circ$. The commutator of $K$ is denoted by $(K,K)$, the center by $Z(K)$, and all subgroups of $K$ are assumed to be closed. The group $\CC\smallsetminus\{0\}$ will be denoted by $\CC^\times$.

We recall the following definition.

\begin{definition}\label{def:vr}
A subgroup $L$ of a connected reductive group $M$ is called {\em very reductive} if $L$ is not contained in any proper parabolic subgroup of $M$.
\end{definition}

%%%%%%%%%%%%%%%%%%%%%%%%%%%%%%%%
%%%%%                                            %%%%%%%%%%%%%%%
%%%%%          NEW SECTION          %%%%%%%%%%%%%%%
%%%%%                                           %%%%%%%%%%%%%%%
%%%%%%%%%%%%%%%%%%%%%%%%%%%%%%%%  

\section{Spherical varieties}\label{Section: Spherical varieties}

In this section we recall some basic notations and facts on spherical varieties.

\begin{definition}
A {\em spherical $G$-variety} (or simply a {\em spherical variety}) is a normal irreducible $G$-variety having an open $B$-orbit. If $K\subseteq G$ is a closed subgroup such that $G/K$ is a spherical variety, then $K$ is called a {\em spherical subgroup} of $G$.
\end{definition}

Let $X$ be a spherical variety. We denote by $\CC(X)^{(B)}$ the multiplicative subgroup of non-zero rational functions on $X$ that are eigenvectors for the action of $B$-translation. The eigenvalues of such eigenvectors form a lattice of characters of $B$, denoted by $\Xi(X)$. Its rank is by definition the {\em rank} of $X$, and we will denote the vector space $\Xi(X)\otimes_\ZZ\QQ$ simply by $\Xi(X)_\QQ$. We also associate with $X$ the vector space
\[
N(X) =  \Hom_\ZZ(\Xi(X),\QQ).
\]
The natural pairing between $\Xi(X)$ (or $\Xi(X)_\QQ$) and $N(X)$ is denoted by $\langle-,-\rangle$.

By a {\em discrete valuation} on $\CC(X)$ we mean a map $v\colon\CC(X)\smallsetminus\{0\}\to \QQ$ such that $v(fg) = v(f)+v(g)$ for all $f,g\in \CC(X)\smallsetminus\{0\}$, such that $v(f+g)\geq \min\{v(f),v(g)\}$ whenever $f,g,f+g\in \CC(X)\smallsetminus\{0\}$, such that $v(f)=0$ if $f$ is constant, and such that the image of $v$ is a discrete (additive) subgroup of $\QQ$.

Any discrete valuation $v$ of $\CC(X)$ can be restricted to $\CC(X)^{(B)}$. Since $X$ is spherical, this yields a well-defined element $\rho(v)$ of $N(X)$, by identifying $\Xi(X)$ with the multiplicative group $\CC(X)^{(B)}$ modulo the constant functions. This applies in particular to the valuation associated with any $B$-stable prime divisor $D$ of $X$, and in this case we denote simply by $\rho(D)$ (or $\rho_X(D)$) the corresponding element of $N(X)$.

The $B$-stable but not $G$-stable prime divisors of $X$ are called {\em colors}, and their set is denoted by $\Delta(X)$. Given a color $D\in \Delta(X)$ and a simple root $\alpha\in S$, we say that {\em $\alpha$ moves $D$} if $D$ is not stable under the minimal parabolic subgroup of $G$ strictly containing $B$ corresponding to $\alpha$.

There exists a minimal set $\Sigma(X)$ of primitive elements of $\Xi(X)$ such that the set
\[
\V(X)=\{ \eta\in N(X)\;|\; \langle \eta,\sigma\rangle \leq 0 \;\forall\sigma\in\Sigma(X)\}
\]
is equal to the set of the elements $\rho(\nu)$ for $\nu$ any $G$-stable discrete valuation of $\CC(X)$. The elements of $\Sigma(X)$ are called the {\em spherical roots} of $X$.

A simple root $\alpha\in S$ can move up to two colors of $X$, and it moves two colors if and only if $\alpha\in S\cap \Sigma(X)$. The set of simple roots moving no color is denoted by $S^p(X)$, and the set of colors moved by some simple roots in $S\cap \Sigma(X)$ is denoted by $\A(X)$.

The {\em homogeneous spherical datum} of $X$ is defined as the quintuple
\[
(S^p(X),\Sigma(X),\A(X),\Xi(X),\rho_X\colon \Delta(X)\to N(X)).
\]
The pairing $\ZZ\Delta(X)\times\Xi(X)\to \ZZ$ induced by $(D,\xi) \mapsto \langle\rho_X(D),\xi\rangle$ for $D\in \Delta(X)$ and $\xi\in\Xi(X)$ is also denoted by $c(-,-)$ (or $c_X(-,-)$), and it is called the {\em Cartan pairing} of $X$.

We recall that spherical homogeneous spaces are classified up to $G$-equivariant isomorphisms by their homogeneous spherical data, which can be defined as purely combinatorial objects satisfying the axioms given by Luna in~\cite[Section~2]{MR1896179}.

Many of the spherical subgroups encountered in this paper are {\em wonderful}, so we recall the definition of this notion.

\begin{definition}
A spherical subgroup $K$ of $G$ is {\em wonderful} if $\Sigma(G/K)$ is a basis of the lattice $\Xi(G/K)$.
\end{definition}

\section{Spherical closure and morphisms}

In this section we recall some combinatorial constructions and results on equivariant morphisms between spherical homogeneous spaces. Our goal is to prove Proposition~\ref{prop:morphisms} and Corollary~\ref{cor:morphisms} below, which are reformulations of well-known facts. They are formulated with the specific goal of being easy to apply in examples, when computing extended weight monoids of specific spherical homogeneous spaces, so to require as few verifications as possible. For this reason, in this form they are not found in the literature, and we need to recall several facts in order to show how our statements derive from known results.

We first need the following standard definition.

\begin{definition}
Let $K$ be a spherical subgroup of $G$. The {\em spherical closure} $\overline K$ is defined as the kernel of the action of the normalizer $N_GK$ on $\Delta(G/K)$, induced by the natural action of $N_GK/K$ on $G/K$ by $G$-equivariant automorphisms. The {\em spherically closed spherical roots} of $G/K$, or of any spherical variety $X$ with open $G$-orbit $G/K$, are the spherical roots of $G/\overline K$. Their set is denoted by $\Sigma^{sc}(X)$.
\end{definition}

\begin{remark}\label{rem:sphericalclosure}
\begin{enumerate}
\item From its definition it is obvious that $\overline K$ contains $K$, hence $\overline K$ is a spherical subgroup of $G$. Also, for any subgroup $J$ such that $K\subseteq J\subseteq \overline K$, the spherical homogeneous spaces $G/K$ and $G/J$ ``have the same colors'', in the sense that the natural morphism $G/K\to G/J$ induces a bijection $\Delta(G/J)\to\Delta(G/K)$ compatible with the Cartan pairing, and respecting the property of a color of being moved by any given simple root.
\item\label{rem:sphericalclosure:doubling} The spherically closed spherical roots of a spherical variety $X$ are easily deduced from $\Sigma(X)$. Indeed, the elements of $\Sigma^{sc}(X)$ are equal to the elements of $\Sigma(X)$, except for the fact that any $\sigma\in\Sigma(X)$ is replaced by $2\sigma$, if $\sigma\notin S$ and $2\sigma$ is a spherical root of some spherical $G$-variety. For a proof of this fact see Pezzini and Van Steirteghem \cite[Proposition~2.7]{PVS}, where also a more explicit combinatorial description of $\Sigma^{sc}(X)$ can be found.
\item The spherical closure of any spherical subgroup is wonderful, thanks to a deep result by Knop~\cite[Corollary~7.6]{MR1311823}.
\end{enumerate}
\end{remark}

We also recall the following proposition, part of Luna's theory of {\em augmentations}.

\begin{proposition}[{\cite[Lemme~6.3.1]{MR1896179}, see also~\cite[Section~3]{MR3317802}}]\label{prop:augmentations}
Let $K$ be a spherical subgroup of $G$. Then $J\mapsto \Xi(G/J)$ is a bijection between the set of subgroups $J$ of $G$ satisfying $K\subseteq J\subseteq \overline K$ and the set of lattices $\Xi'$ such that $\Xi(G/K)\supseteq \Xi'\supseteq\Xi(G/\overline K)=\Span_\ZZ(\Sigma^{sc}(G/K))$.
\end{proposition}

We come to equivariant morphism between spherical varieties. The following definitions and proposition are well-known.

\begin{definition}
A subset $\Delta'\subseteq \Delta(X)$, is called {\em distinguished} if there is a linear combination of elements of $\rho_X(\Delta')$, with non-negative rational coefficients, that is non-negative on all spherical roots of $X$. In addition, $\Delta'$ is called {\em parabolic} if there is such a linear combination that is strictly positive on all spherical roots.
\end{definition}

\begin{definition}\label{def:morphisms}
Let $\varphi\colon X\to Y$ be a dominant $G$-equivariant morphism with connected fibers between two spherical $G$-varieties $X$, $Y$. We denote by $\Delta_\varphi$ the set of colors of $X$ mapped dominantly to $Y$.
\end{definition}

\begin{proposition}\label{prop:morphisms}
Let $\varphi$, $X$, and $Y$ be as in Definition~\ref{def:morphisms}. Then $\Delta_\varphi$ is a distinguished subset of $\Delta(X)$, and we have
\begin{equation}\label{eq:Sigmaquotient}
\Span_{\QQ_{\geq0}}\Sigma(Y) = \left(\Span_{\QQ_{\geq0}}\Sigma(X)\right)\cap\left(\bigcap_{D\in\Delta_\varphi} \ker\rho_X(D)\right).
\end{equation}
If in addition $Y$ is homogeneous and complete, then $\Delta_\varphi$ is parabolic. Finally, given $X$ homogeneous and a distinguished subset $\Delta'$ of $\Delta(X)$, there exist unique (up to $G$-equivariant isomorphism) $Y$ and $\varphi$ such that $Y$ is a homogeneous $G$-variety $Y$ with $\Span_\QQ\Sigma(Y)=\Xi(Y)_\QQ$, and $\varphi\colon X\to Y$ is a $G$-equivariant morphism with connected fibers such that $\Delta_\varphi=\Delta'$. If $\Delta'$ is parabolic, then $Y$ is complete.
\end{proposition}
\begin{proof}
We show how to reduce the proposition to known results in the literature. First, one may assume $X$ and $Y$ homogeneous by replacing them with their respective open $G$-orbits. Also, notice that $\Xi(Y)\subseteq \Xi(X)$ by pulling back $B$-semiinvariant rational functions from $Y$ to $X$.

The subset $\Delta_\varphi$ of $\Delta(X)$ is distinguished because the convex cone generated by $\rho_X(\Delta_\varphi)$ together with $\V(X)\cap N_\varphi$ is a vector subspace of $N(X)$. Here $N_\varphi$ is the vector subspace of $N(X)$ of the elements that are zero on $\Xi(Y)$, and this result follows from~\cite[Lemma~4.3, part~b)]{MR1131314}.

In particular $N_\varphi$ is generated, as a vector space, by $\rho_X(\Delta_\varphi)$ and $N_\varphi\cap \V(X)_{\text{lin}}$, where $\V(X)_{\text{lin}}$ is the linear part of the convex cone $\V(X)$. This implies that
\begin{equation}\label{eq:XiY}
\Xi(Y)_\QQ = U \cap \left(\bigcap_{D\in\Delta_\varphi} \ker\rho_X(D)\right)
\end{equation}
where $U$ is the common kernel of all elements of $N_\varphi\cap \V(X)_{\text{lin}}$. Notice that
\begin{equation}\label{eq:U}
U\supseteq \Span_{\QQ}\Sigma(X).
\end{equation}

The last assertions of Knop's Theorem~4.4 in~\cite{MR1131314} also imply that
\[
\Span_{\QQ_{\geq0}}\Sigma(Y) = \left(\Span_{\QQ_{\geq0}}\Sigma(X)\right)\cap\Xi(Y)_\QQ,
\]
which yields~(\ref{eq:Sigmaquotient}) thanks to~(\ref{eq:XiY}) and~(\ref{eq:U}).

Viceversa, let $X$ be a spherical homogeneous space and $\Delta'$ a distinguished subset of $\Delta(X)$, and denote by $N'$ the vector subspace of $N(X)$ of the elements vanishing on
\[
M'=\left(\Span_{\QQ}\Sigma(X)\right)\cap\left(\bigcap_{D\in\Delta_\varphi} \ker\rho_X(D)\right).
\]
Theorem~4.4 in loc.cit.\ implies that there exist a spherical homogeneous space $Y$ and a dominant $G$-equivariant morphism $\varphi\colon X\to Y$ with connected fibers, such that $N_\varphi=N'$ (i.e.\ $M'=\Xi(Y)_\QQ$), and $\Delta_\varphi=\Delta'$. The same theorem also states that $Y$ and $\varphi$ are uniquely determined by $\Delta_\varphi$ and $N_\varphi$, under the assumptions that $Y$ is homogeneous and that $\varphi$ is dominant with connected fibers. So $Y$ and $\varphi$ are uniquely determined by $\Delta_\varphi$ if we fix the subspace $N_\varphi=N'$.

Finally, it is well-known that a spherical homogeneous space $Y$ is complete if and only if $\Xi(Y)=\{0\}$, which is equivalent to $N'=N(X)$, which is equivalent to $\Delta'$ being parabolic.
\end{proof}

\begin{remark}
\begin{enumerate}
\item Notice that $\Sigma(Y)$ is uniquely determined by the equality~(\ref{eq:Sigmaquotient}) and $\Xi(Y)$, because by definition its elements are primitive in the lattice $\Xi(Y)$.
\item Equality~(\ref{eq:Sigmaquotient}) holds even if $\varphi$ does not have connected fibers. Indeed, in this case we replace $X$ and $Y$ by their open $G$-orbits, say resp.\ $G/K$ and $G/J$ with $K\subseteq J$, and we may assume that $\varphi$ is the natural map $G/K\to G/J$. Let $\widetilde J$ be the union of those connected components of $J$ that intersect $K$, and consider the factorization of $\varphi$ given by the natural morphisms $G/K\to G/\widetilde J\to G/J$. Then we apply equality~(\ref{eq:Sigmaquotient}) to the first morphism (which has connected fibers), and observe that $\widetilde J^\circ= J^\circ$. This implies by~\cite[Theorem~4.4 and Proof of Theorem~6.1]{MR1131314} that $G/J$ and $G/\widetilde J$ have the same spherical roots, up to replacing some elements with positive rational multiples.
\end{enumerate}
\end{remark}

\begin{corollary}\label{cor:morphisms}
Let $K,J$ be spherical subgroups of $G$ and set $X=G/K$, $Y=G/J$. Suppose that $\Span_\QQ\Sigma(Y)=\Xi(Y)_\QQ$, and that there exist
\begin{enumerate}
\item a distinguished subset $\Delta'\subseteq\Delta(X)$ such that equality~(\ref{eq:Sigmaquotient}) holds, where we substitute $\Delta'$ for $\Delta_\varphi$, and
\item a bijection $\psi\colon\Delta(Y)\to \Delta(X)\smallsetminus \Delta'$ such that $\rho_Y(\psi(D))=\rho_X(D)|_{\Xi(Y)_\QQ}$ for all $D\in\Delta(X)\smallsetminus\Delta'$, and such that any simple root moves $D$ if and only if it moves $\psi(D)$.
\end{enumerate}
If $K$ is connected, then $K$ is conjugated in $G$ to a subgroup of $J$.
\end{corollary}
\begin{proof}
By Proposition~\ref{prop:morphisms} there exists a spherical homogeneous space $\widetilde Y$ and a dominant $G$-equivariant morphism $\varphi\colon X\to \widetilde Y$ with connected fibers, such that $\Delta_\varphi=\Delta'$ and $\Span_\QQ\Sigma(Y)=\Span_\QQ\Sigma(\widetilde Y)$, hence $\Xi(Y)_\QQ=\Xi(\widetilde Y)_\QQ$. Let $\widetilde J$ be such that $\widetilde Y\cong G/\widetilde J$ and the point $eK\in X$ is sent to $e\widetilde J$ by $\varphi$, i.e.\ $K\subseteq \widetilde J$. Pullling back along $\varphi$ induces a bijection $\widetilde\psi\colon \Delta(\widetilde Y)\to \Delta(X)\smallsetminus\Delta'$ with the same above properties stated for $\psi$, where we replace $\Xi(Y)$ with $\Xi(\widetilde Y)$ (notice that the inverse image of a color of $\widetilde Y$ is a single color of $X$ because $\varphi$ has connected fibers).

We conclude that $\widetilde Y$ and $Y$ have the same homogeneous spherical datum, up to the fact that the lattices $\Xi(Y)$ and $\Xi(\widetilde Y)$ may be different (but both have both maximal rank in the same $\QQ$-vector space), and the spherical roots may differ by the rescaling with positive rational numbers. We claim that then the groups $\widetilde J^\circ$ and $J^\circ$ are conjugated in $G$, which would conclude the proof.

Let us show the claim. First we observe that $\Sigma(Y)\cap S=\Sigma(\widetilde Y)\cap S$, because they are equal up to rescaling, and we have a bijection between the colors of $Y$ and of $\widetilde Y$ respecting the property of being moved by any given simple root. This implies that the simple roots that move two colors each correspond, and that's exactly the simple roots that are spherical roots.

From Remark~\ref{rem:sphericalclosure}, part~(\ref{rem:sphericalclosure:doubling}), we conclude that $\Sigma^{sc}(Y)=\Sigma^{sc}(\widetilde Y)$, in particular $\Sigma^{sc}(Y)$ and $\Sigma^{sc}(\widetilde Y)$ are both contained in $\Xi'=\Xi(\widetilde Y)\cap \Xi(Y)$.

By Proposition~\ref{prop:augmentations}, the inclusions $\Xi(Y)\supseteq\Xi'\supseteq \Span_\ZZ\Sigma^{sc}(Y)$ induces an inclusion $J\subseteq J_0$ where $J_0$ is a subgroup of $G$ such that $G/J_0$ has the same homogeneous spherical datum of $Y$, except for the lattice $\Xi(G/J_0)=\Xi'$, and as above the spherical roots may differ by some rescaling. The same applies to $\widetilde J$, yielding another inclusion $\widetilde J\subseteq \widetilde J_0$ with $\Xi(G/\widetilde J_0)=\Xi'$. Since $\Xi'$ has finite index in both $\Xi(Y)$ and $\Xi(\widetilde J)$, by a lemma of Gandini's~\cite[Lemma~2.4]{MR2785497} we conclude that $J^\circ = J^\circ_0$ and $\widetilde J^\circ = \widetilde J^\circ_0$.

Finally, by Losev's uniqueness theorem~\cite[Theorem~1]{MR2495078}, the subgroups $J_0$ and $\widetilde J_0$ are conjugated in $G$, concluding the proof of the claim.
\end{proof}

%%%%%%%%%%%%%%%%%%%%%%%%%%%%%%%%
%%%%%                                            %%%%%%%%%%%%%%%
%%%%%          NEW SECTION          %%%%%%%%%%%%%%%
%%%%%                                           %%%%%%%%%%%%%%%
%%%%%%%%%%%%%%%%%%%%%%%%%%%%%%%%  

\section{The extended weight monoid}\label{s:extwm}

Let $K\subseteq G$ be a subgroup. If a regular function $f\in\CC[G]$ is simultaneously an eigenvector for the left translation action of $B$ on $\CC[G]$ and an eigenvector for the right translation action of $K$, then we denote by $\omega_f\in\Chi(B)$ and $\chi_f\in\Chi(K)$ the $B$-eigenvalue and the $K$-eigenvalue respectively.

\begin{definition}\label{def: EWS}
Let $K$ be a spherical subgroup of $G$. The {\em extended weight monoid}%
\footnote{It is also called the {\em extended weight semigroup}.}
of $X=G/K$ is the set of couples $(\omega_f,\chi_f)$, for $f\in\CC[G]$ varying in the set of $B$-eigenvectors for the left translation action and $K$-eigenvectors for the right translation action. It is denoted by $\widetilde\Gamma(X)$.
\end{definition}

Consider $X=G/K$ as in the definition, and $D\in\Delta(X)$. Since $G$ is simply connected, the inverse image $\widetilde D$ under the natural projection $G\to G/K$ has a global equation $f_D$, unique up to a non-zero multiplicative constant. Therefore it is well defined the element $(\omega_D,\chi_D)=(\omega_{f_D},\chi_{f_D})$.

We recall the following well-known proposition.

\begin{proposition}\label{prop:generators}
Let $X=G/K$ be a spherical homogeneous space. Then the monoid $\widetilde\Gamma(X)$ is freely generated by the elements $(\omega_D,\chi_D)$ for $D$ varying in $\Delta(X)$.
\end{proposition}
\begin{proof}
The weights $(\omega_D,\chi_D)$ are linearly independent, see Brion~\cite[Lemma~2.1.1]{MR2326138}. We prove they generate the monoid $\widetilde\Gamma(X)$, so let $f\in\CC[G]$ be a left-$B$-eigenvector and a right-$K$-eigenvector. Its divisor is invariant under left translation by $B$ and right translation by $K$, therefore it is the pull-back along the projection $G\to G/K$ of a $B$-stable divisor $\delta$ on $G/K$.

Since $B$ is connected, we have that $\delta$ is a linear combination of $B$-stable prime divisors of $G/K$, i.e.\ of colors:
\[
\delta= \sum_{D\in\Delta(X)} n_D D
\]
where $n_D$ is a non-negative integer for all $D$. Then the product $F=\prod_{D\in\Delta(X)}f_D^{n_D}$ is a regular function on $G$ having the same divisor of $f$. It follows that $F$ is a scalar multiple of $f$, yielding $(\omega_F,\chi_F)=(\omega_f,\chi_f)$. It remains to observe that by construction
\[
(\omega_F,\chi_F) = \sum_{D\in\Delta(X)} n_D(\omega_D,\chi_D).
\]
\end{proof}

The following lemma is an easy generalization of a result of Foschi's~\cite{foschi-thesis}. It provides an explicit formula for $\omega_D$ for any color $D$.

\begin{lemma}[{\cite{foschi-thesis}}]\label{lemma:foschi}
Let $X=G/K$ be a spherical homogeneous space. Let $D\in\Delta(X)$ and let $\alpha$ be a simple root moving $D$. Set $\varepsilon=2$ if $2\alpha\in\Sigma(X)$, or $\varepsilon=1$ otherwise. Then
\[
\omega_D = \varepsilon\sum_{\beta} \omega_\beta,
\]
where $\beta$ varies in the set of simple roots that move $D$.
\end{lemma}
\begin{proof}
See~\cite[Lemma 30.24]{MR2797018}.
\end{proof}

%%%%%%%%%%%%%%%%%%%%%%%%%%%%%%%%
%%%%%                                            %%%%%%%%%%%%%%%
%%%%%          NEW SECTION          %%%%%%%%%%%%%%%
%%%%%                                           %%%%%%%%%%%%%%%
%%%%%%%%%%%%%%%%%%%%%%%%%%%%%%%%  

\section{Computing the extended weight monoid}\label{s:howto}

Let $X=G/K$ be a spherical homogeneous space. In this section we explain how to compute explicitly the extended weight monoid of $X$, using the homogeneous spherical datum of $X$. For this we refer to the generators of $\widetilde \Gamma(X)$ appearing in Proposition~\ref{prop:generators}. The weights $\omega_D$ of that proposition are given in Lemma~\ref{lemma:foschi}. Here we will show how to derive the weights $\chi_D$, expressing them as restrictions of certain weights of $T$, provided we choose $K$ in its conjugacy class in a suitable way.

Let us fix a parabolic subgroup $Q$ of $G$ minimal containing $K$. Up to conjugating $Q$ and $K$ in $G$, we may assume that $Q$ contains $K$ and also $B_-$, and we denote by $L_Q$ the Levi subgroup of $Q$ containing $T$. By a standard argument, the group $K$ has a Levi subgroup $L_K$ contained, spherical, and very reductive in $L_Q$, and we have $K^u\subseteq Q^u$.

Let us denote by $\Delta'$ the parabolic subset of colors corresponding to the natural map $\pi\colon X\to G/Q$. Then the colors of $G/Q$ are the images $\pi(D)$ of the colors $D\in\Delta(X)\smallsetminus \Delta'$.

The set $\Delta(G/Q)$ is also identified with $S\smallsetminus S^p(G/Q)$, by associating a color $\pi(D)\in\Delta(G/Q)$ with the (unique) simple root $\alpha$ that moves it. Moreover, for all $D\in\Delta(X)\smallsetminus\Delta'$, a simple application of the Peter-Weyl theorem yields
\begin{equation}\label{eq:colorsofGmQ}
(\omega_{\pi(D)},\chi_{\pi(D)}) = (\omega_\alpha,-\omega^Q_\alpha),
\end{equation}
where we denote here by $\omega^Q_\alpha\colon Q\to\CC^{\times}$ the extension of $\omega_\alpha\colon T\to\CC^{\times}$ to $Q$. Denote by $\omega_\alpha^K$ the restriction $\omega_\alpha^Q|_{K}\colon K\to \CC^{\times}$. 

Consider now a color $D\in\Delta(X)\smallsetminus \Delta'$: equality~(\ref{eq:colorsofGmQ}) yields
\begin{equation}\label{eq:paraboliccolors}
(\omega_D,\chi_D) = (\omega_\alpha,-\omega^K_\alpha).
\end{equation}

To determine the weights of the other colors, we use the fact that the spherical roots of $X$ give condition on these weights, via the Cartan pairing of $X$.

\begin{lemma}\label{lemma:sphroot}
For all $\sigma\in\Sigma(X)$ we have
\begin{equation}\label{eq:sphroot}
(\sigma,0) = \sum_{D\in\Delta(X)} c_X(D,\sigma) \, (\omega_D,\chi_D).
\end{equation}
\end{lemma}
\begin{proof}
The formula is well known. It is proved e.g.\ by Brion in \cite[Section~2.1]{MR2326138}, under the assumption that $K$ is wonderful. The same argument of loc.cit.\ yields the formula for any spherical subgroup $K\subseteq G$.

One can also easily reduce the general case to the one of wonderful subgroups, by using the {\em wonderful closure} $\widehat K$ of $K$, defined in~\cite{MR3317802}. The relevant properties of $\widehat K$ are that $K\subseteq \widehat K\subseteq \overline K$, that $\Sigma(G/\widehat K)=\Sigma(X)$, and that the natural morphism $X\to G/\widehat K$ induces a bijection between $\Delta(X)$ and $\Delta(G/\widehat K)$ compatible with the Cartan pairing. Then the formula of the lemma for $G/K$ is clearly equivalent to the same formula for $G/\widehat K$.
\end{proof}

To show that this is enough to determine the weights of all colors, we will need the following known auxiliary results.

\begin{lemma}\label{lemma:vr}
Let $L$ be a very reductive subgroup of a connected reductive group $M$. Then the restriction map $\Chi(L)\to \Chi(L\cap Z(M)^\circ)$ has finite kernel.
\end{lemma}
\begin{proof}
Consider the commutative diagram
\[
\begin{array}{ccc}
\Chi(L) &\to& \Chi(L\cap Z(M)^\circ)\\
\downarrow & & \downarrow \\
\Chi(L^\circ) & \to & \Chi(L^\circ\cap Z(M)^\circ)
\end{array}
\]
given by the restriction maps. Using the fact that the vertical maps have finite kernel, we may replace $L^\circ$ by $L$ if necessary, and assume that $L$ is connected.

Let $\chi$ be a character of $L$ trivial on $L\cap Z(M)^\circ$. Then it descends to a character $\overline \chi$ of the image $\overline L$ via the quotient $M\to \overline M = M/(Z(M)^\circ)$. Notice that $\overline L$ is a connected very reductive subgroup of the semisimple group $\overline M$.

Our lemma follows now from the claim that the character group of $\overline L$ is trivial. To prove this claim, suppose for sake of contradiction that $\Chi(\overline L)$ is non-trivial. Then the center of $\overline L$ has positive dimension, because $\overline L = (\overline L,\overline L)\cdot Z(\overline L)^\circ$. 

Let $\xi\colon\CC^{\times}\to Z(\overline L)^\circ$ be a non-trivial one-parameter subgroup. Then $\overline L$ is contained in the set $P(\xi)$ of elements $m\in M$ such that the limit
\[
\lim_{a\to 0} \xi(a) m \xi(a)^{-1}
\]
exists. But $P(\xi)$ is a proper parabolic subgroup of $\overline M$ (see \cite[Proposition~8.4.5 and its proof]{MR1642713})): contradiction.
\end{proof}

We proceed by giving two useful formulae in the next lemma. The first is well-known; the second was already implicitly given by Bravi and Pezzini in~\cite[Lemma~3.2.1]{MR3198836}, and it turns out to be the core of our procedure for computing extended weight monoids.

\begin{lemma}\label{lemma:eq}
We have
\begin{eqnarray}
\rk\Chi(K) & = &|\Delta(X)|-\rk\Xi(X),\\
\label{eq:kernel} \rk\Xi(X) - \dim(\Span_\QQ\rho_X(\Delta')) & = &  \rk\Chi(Q) - \rk\Chi(K).\label{eq:rankChar}
\end{eqnarray}
\end{lemma}
\begin{proof}
The first formula is well known, let us give a proof for convenience. We may assume that $BK$ is open in $G$. Then we notice that $\Xi(X)$ is isomorphic to the quotient $\CC(X)^{(B)}/\CC^{\times}$, by associating to $\lambda\in\Xi(X)$ the class $[f]$ (modulo the constant invertible functions) of a $B$-eigenvector $f\in\CC(X)$ of $B$-eigenvalue $\lambda$. Since $BK/K$ is open in $G/K$, we may identify $\CC(X)^{(B)}/\CC^{\times}$ with the group of characters $\Chi(B/(B\cap K))$, which is isomorphic to $\Chi(B)^{B\cap K}$.

By~\cite[Lemma~2.1.1]{MR2326138} (whose proof holds for any spherical subgroup $K\subseteq G$), we have
\[
\rk(\Chi(B)\times_{\Chi(B\cap K)}\Chi(K)) = |\Delta(X)|.
\]
This, together with the exact sequence
\[
0 \to \Chi(B)^{B\cap K} \to \Chi(B)\times_{\Chi(B\cap K)}\Chi(K) \to \Chi(K) \to 0
\]
(the third map is surjective because all characters of $B\cap K$ extend to characters of $B$), yields the first formula of the lemma.

Formula~(\ref{eq:kernel}) follows easily from \cite[Lemma~3.2.1]{MR3198836} under the assumption $\Xi(X)=\Span_\ZZ\Sigma(X)$. The same proof holds without this assumption, let us check the details.

Thanks to Lemma~\ref{lemma:vr}, and the fact that the restriction map $\Chi(K)\to\Chi(L_K)$ is an isomorphism, we may replace $\Chi(K)$ in formula~(\ref{eq:kernel}) with $\Chi(L_K\cap Z(L_Q)^\circ)$. We may also replace $\Chi(Q)$ by $\Chi(Z(L_Q)^\circ)$, since $Q=Q^u\cdot (L_Q,L_Q)\cdot Z(L_Q)^\circ$.

Now, the formula follows if we prove that restricting the elements of $\Xi(X)$ to the subtorus $Z(L_Q)^\circ$ of $T$ induces an injective map
\[
\varphi\colon \bigcap_{ D\in \Delta' } \ker(\rho_X(D)) \to \Chi(Z(L_Q)^\circ)^{L_K\cap Z(L_Q)^\circ}.
\]
with finite cokernel.

Let $\gamma\in\bigcap_{ D\in \Delta' } \ker(\rho_X(D))$, and let $f_\gamma\in\CC(X)$ be a $B$-eigenvector of $B$-eigenvalue $\gamma$. Call $F_\gamma$ the pull-back of $f_\gamma$ on $G$. By our assumptions on $\gamma$, the function $F_\gamma$ has no zero nor pole except possibly for the pull-backs of colors of $G/Q$ along the projection $G\to G/Q$.

This implies that $F_\gamma$ is a $B$-eigenvector, under the action of $G$ on $G$ by left translation, and a $Q$-eigenvector under the action of right translation. Using the Peter-Weyl theorem, it is elementary to deduce that the $Q$-eigenvalue (restricted to $T$) of $F_\gamma$ is $-\gamma$. But $F_\gamma$ is also $K$-invariant under right translation, therefore $L_K$-invariant, and so $-\gamma|_{L_K\cap Z(Q)^\circ}$ is trivial. We deduce that the map $\varphi$ as above is defined.

This argument also shows that $\gamma$ extends to a character of $Q$, and if it is trivial on $Z(L_Q)^\circ$ then it trivial on $Q$. It follows that $\varphi$ is injective.

Finally, consider $\chi\in\Chi(Z(L_Q)^\circ)$, and extend it to a character $\widetilde\chi$ of $Q$. This is, again using the Peter-Weyl theorem, the $Q$-eigenvalue of some $Q$-eigenvector $F\in\CC(G)$ under right translation, such that $F$ is also a $Q_+$-eigenvector of weight $-\widetilde\chi$ under left translation. Here $Q_+$ is the parabolic subgroup of $G$ opposite to $Q$ with respect to $T$.

Suppose that $\chi$ (and hence $\widetilde\chi$) is trivial on $L_K\cap Z(Q)^\circ$. By Lemma~\ref{lemma:vr}, a multiple $n\widetilde \chi$ for some nonzero $n\in\ZZ$ is trivial on $L_K$, and hence on $K$. We deduce that $F^n$ is fixed under right translation by $K$, i.e.\ it descends to a $B$-eigenvector of $\CC(X)$ of $B$-eigenvalue $\gamma=-n\widetilde\chi|_{B}$, with the property that $-\gamma|_{Z(L_Q)^\circ}=n\chi$. This shows that $\varphi$ has finite cokernel, and the proof is complete.
\end{proof}

\begin{theorem}\label{thm:howto}
Let $\Delta'\subseteq\Delta(X)$ be a minimal parabolic subset of colors. Then equalities~(\ref{eq:paraboliccolors}) for all $D\in\Delta'$, and equalities~(\ref{eq:sphroot}) for all $\sigma\in\Sigma(X)$, determine $\chi_D|_{K^\circ}$ uniquely for all $D\in\Delta(X)$.
\end{theorem}
\begin{proof}
We first prove the theorem under the assumption that $\Xi(X)_\QQ=\Span_\QQ\Sigma(X)$.

Let us denote for brevity $\chi_D|_{K^\circ}$ by $\chi'_D$. Equalities~(\ref{eq:paraboliccolors}) determine $\chi'_D$ for all $D\in\Delta(X)\smallsetminus\Delta'$, it remains to prove that the elements $\chi'_E$ for $E\in\Delta'$ are also uniquely determined. We consider then~(\ref{eq:sphroot}) for $\sigma$ varying in $\Sigma(X)$ as a system of inhomogeneous linear equations with unknowns $\chi'_E$.

These unknowns take values in $\Chi(K^\circ)$, which is a free abelian group of finite rank. It is harmless to consider the unknowns as taking values in $\Chi(K^\circ)_\QQ$, so that our system of equations has vector unknowns, vector constant terms, and scalar coefficients. To prove that it has a unique solution, we must show that the number of unknowns is equal to the rank of the matrix of the homogeneous system, that is, the matrix
\[
A = \left( c_X(D,\sigma) \right)_{D\in\Delta',\sigma\in\Sigma(X)}.
\]

We put together the equalities of Lemma~\ref{lemma:eq}, together with $\rk(Q)=|\Delta(G/Q)| = |\Delta(X)|-|\Delta'|$ (where the first equality follows from the first equality of Lemma~\ref{lemma:eq} applied to the spherical subgroup $Q$ of $G$), and obtain
\begin{equation}\label{eq:last}
|\Delta'| = \dim(\Span_\QQ\rho_X(\Delta')).
\end{equation}
Since $\Sigma(X)$ is a basis $\Xi_\QQ$, the right hand side of~(\ref{eq:last}) is equal to the rank of $A$. Therefore the theorem holds.

To finish the proof, it remains to reduce the general case to the case where $\Xi(X)_\QQ=\Span_\QQ\Sigma(X)$, so let $K$ be any spherical subgroup of $G$. Let us consider again the wonderful closure $\widehat K$ of $K$, and set $Y=G/\widehat K$. Thanks to the recalled properties of $\widehat K$, in particular the given identification between $\Delta(X)$ and $\Delta(Y)$, the pull-back to $G$ of the colors of $X$ and of $Y$ are the same. So, for all $D\in\Delta(X)$, one can take the same global equation to define the weight $\chi_D\in\Chi(K)$ for $G/K$ and the weight for $G/\widehat K$, let us denote it by $\widehat \chi_D$. This means that $\widehat \chi_D|_K=\chi_D$.

The equalities~(\ref{eq:paraboliccolors}) and~(\ref{eq:sphroot}) are the same for $X$ and for $Y$. The first part of the proof, applied to $Y$, yields that these equalities determine uniquely $(\omega_D,\widehat\chi_D|_{(\widehat K)^\circ})$ for all $D$. But $K^\circ$ is contained in $(\widehat K)^\circ$, hence $(\omega_D,\widehat\chi_D|_{K^\circ})$, which is equal to $(\omega_D,\chi_D|_{K^\circ})$, is uniquely determined too.
\end{proof}

We end this section discussing the equalities~(\ref{eq:sphroot}) in case we include into the picture groups that are not simply connected.

Let then $G_0$ be a quotient of $G$ by a finite central subgroup. As before, let $K$ be a spherical subgroup of $G$, denote by $K_0$ the image of $K$ in $G_0$, and let $\pi\colon G/K\to G_0/K_0$ be the natural induced morphism.

\begin{proposition}\label{prop:G0}
Suppose that $\Sigma(G_0/K_0)\cap 2S=\varnothing$. Then $D\mapsto \pi(D)$ is a bijection between $\Delta(G/K)$ and $\Delta(G_0/K_0)$ compatible with the Cartan pairing. Moreover, the equalities~(\ref{eq:sphroot}) hold for $\sigma$ varying in $\Sigma(G_0/K_0)$; if one considers them as a system of equalities in the unknowns $\chi_D$ taking values in $\Chi(K^\circ)_\QQ$, then the system is equivalent to the one where $\sigma$ varies in $\Sigma(G/K)$.
\end{proposition}
\begin{proof}
We consider $G_0/K_0$ naturally as a $G$-variety, i.e.\ as $G/\widetilde K$ where $\widetilde K$ is the inverse image of $K_0$ in $G$. Since $\widetilde K$ normalizes $K$, by~\cite[Theorem~4.4 and Theorem~6.1]{MR1131314} the spherical roots of $G_0/K_0$ and $G/K$ are the same, up to replacing some elements by positive multiples. Moreover, if two corresponding spherical roots are different, then the spherical root of $G_0/K_0$ is the double of the spherical root of $G/K$ (see e.g.\ Wasserman's tables in~\cite{MR1424449}), and we conclude that $\Sigma(G/K)\cap S = \Sigma(G_0/K_0)\cap S$. Then, by the combinatorial description of the set of colors of a spherical variety in \cite[Section~2.3]{MR1896179}, we obtain the first assertion of the proposition. The second assertion follows from the same considerations, because the equalities~(\ref{eq:sphroot}) are linear in $\sigma$.
\end{proof}

%%%%%%%%%%%%%%%%%%%%%%%%%%%%%%%%
%%%%%                                            %%%%%%%%%%%%%%%
%%%%%          NEW SECTION          %%%%%%%%%%%%%%%
%%%%%                                           %%%%%%%%%%%%%%%
%%%%%%%%%%%%%%%%%%%%%%%%%%%%%%%%  

\section{On parabolic subgroups of spherical subgroups}

Let $H$ and $P$ be spherical subgroups of $G$, such that $P\subseteq H$. In this section we prove a sufficient combinatorial condition for $P$ to be a parabolic subgroup of $H$. Set $X=G/P$ and $Y=G/H$.

\begin{proposition}\label{prop:PinH}
Suppose $|\Sigma(X)|=\rk\Xi(X)$ and $|\Sigma(Y)|=\rk\Xi(Y)$, and suppose that for all proper subsets $\Sigma'\subsetneq \Sigma(X)$ we have
\[
\Sigma(Y) \not\subseteq \Span_{\QQ_{\geq 0}}\Sigma'.
\]
Then $P$ is a parabolic subgroup of $H$.
\end{proposition}
\begin{proof}
Since $|\Sigma(X)|=\rk\Xi(X)$ and $\Sigma(X)$ is linearly independent, the dual cone of $\QQ_{\geq0}\Sigma(X)$ in $\Hom_\ZZ(\Xi(X),\QQ)$ is strictly convex. By \cite[Theorem~3.1]{MR1131314} there exists a $G$-equivariant completion $\overline X$ of $X$ (sometimes called the {\em canonical completion}) with a single closed $G$-orbit, such that no color of $\overline X$ contains any $G$-orbit, and such that $\rho_X(D)$ generate the convex cone $\V(X)$, for $D$ varying among the $G$-stable prime divisors of $\overline X$. Thanks to  Knop~\cite[Proposition~6.1]{MR3218807}, for all $G$-orbits $Z$ of $\overline X$ different from $X$ we have $\Sigma(Z)\subsetneq \Sigma(X)$.

The same results can be applied to $Y$, obtaining a $G$-equivariant completion $\overline Y$ with similar properties. By~\cite[Theorem~4.1]{MR1131314}, the natural morphism $\varphi\colon X\to Y$ extends to a $G$-equivariant morphism $\overline\varphi\colon \overline X\to \overline Y$. Then the fiber $\overline\varphi\inv(eH)$ over the point $eH\in Y\subseteq \overline Y$ is a completion of $H/P$.

We claim that $\overline\varphi\inv(eH)=H/P$, which will conclude the proof because it implies that $H/P$ is a complete variety. To show the claim, we proceed by contradiction. Suppose $\overline\varphi\inv(eH)\supsetneq H/P$, then there exists a $G$-orbit $Z$ contained in $\overline X$ and different from $X$, such that $\overline\varphi(Z)=Y$. For the morphism $\overline\varphi|_Z\colon Z\to Y$, the equality~(\ref{eq:Sigmaquotient}) takes the form
\[
\Span_{\QQ_{\geq0}}\Sigma(Y) = \left(\Span_{\QQ_{\geq0}}\Sigma(Z)\right)\cap\left(\bigcap_{D\in\Delta_{\overline\varphi|_Z}} \ker\rho_Z(D)\right).
\]
In particular $\Sigma(Y)\subseteq\Span_{\QQ_{\geq 0}}\Sigma(Z)$, contradicting our hypotheses and concluding the proof.
\end{proof}

%%%%%%%%%%%%%%%%%%%%%%%%%%%%%%%%
%%%%%                                         %%%%%%%%%%%%%%
%%%%%          NEW SECTION          %%%%%%%%%%%%%%
%%%%%                                         %%%%%%%%%%%%%%
%%%%%%%%%%%%%%%%%%%%%%%%%%%%%%%%

\section{An example}\label{s:example}

In this section we give an example of a multiplicity free system $(G,H,P)$ and the computation of the corresponding extended weight monoid. We fix $G=\Spin(2n+2)$ with $n\geq 3$, and we want to consider all possible spherical homogeneous space $G/P$ where $P$ is a parabolic subgroup of $H$, and the latter is a symmetric subgroup of semisimple type $\mathsf D_n$. Thanks to Proposition~\ref{prop:G0}, it would also be possible to use the classical group $\SO(2n+2)$ instead of its universal cover $G$, as long as one works in the vector spaces $\Chi(T)_\QQ$ and $\Chi(P)_\QQ$.

In $G$ we fix $B$, $B^-$ and $T$ as in the previous sections; we denote the simple roots of $G$ by $\alpha_1,\ldots,\alpha_{n+1}$, numbered as in Bourbaki, with the corresponding fundamental dominant weights denoted by $\omega_1,\ldots,\omega_{n+1}$.

We take $H$ to be the Levi subgroup containing $T$ of the parabolic subgroup of $G$ containing $B$ and obtained by omitting the first simple root $\alpha_{1}$ of $G$. Then $T$ is a maximal torus of $H$, and $B\cap H$ is a Borel subgroup of $H$.

We denote the corresponding simple roots of $H$ by $\beta_1,\ldots,\beta_n$, and by $\varpi_1,\ldots,\varpi_n$ the corresponding fundamental dominant weights. The latter are defined as the elements of $\Chi(T)_\QQ$ that have the correct pairing with the simple coroots of $H$ and are zero on the subspace of $\Chi(T)_\QQ^*$ corresponding to the connected center of $H$. Then we have $\beta_{i}=\alpha_{i+1}$ and $\varpi_{i}=\omega_{i+1}$ for $i\in\{1,\ldots,n\}$. In addition, we set $\varpi_0=\omega_1$, which is the restriction to $T$ of a generator of the character group of $H$.  

The parabolic subgroups of $H$ are taken to contain $B^-\cap H$. Given such a subgroup $P$, we denote by $I$ the set of simple roots of $H$ that are not roots of $P$.

We have the following possibilities for $I$ (see~\cite{MR3127988}):
\begin{enumerate}
\item\label{sym3a} $I=\{\beta_{n-1}\}$,
\item\label{sym3b} $I=\{\beta_n\}$.
\end{enumerate}

The subgroup $P$ of case~(\ref{sym3a})
is obtained by applying the outer automorphism of $G$ that interchanges $\alpha_{n}$ and $\alpha_{n+1}$ to the subgroup $P$ of case~(\ref{sym3b}). Also, these two parabolic subgroups of $H$ are conjugated in $G$ (see below). For this reason we only discuss case~(\ref{sym3b}), and we notice that the corresponding $P$ appears as case 53 of \cite{MR2183057}, with ``$n$'' in loc.cit.\ having the same meaning as here.

This yields
\[
\Sigma(G/P) = \{\alpha_1,\sigma_1=\alpha_{2,n}, \sigma_2=\alpha_{2,n-1}+\alpha_{n+1}\}
\]
and
\[
\Delta(G/P) = \{ D_1^+, D_1^-, D_2, D_{n}, D_{n+1} \},
\]
where for any $i\in\NN$ we denote by $D_i$ a color moved only by the simple root $\alpha_i$ if $\alpha_i$ is not a spherical root, and we denote by $D_i^+,D_i^-$ the two colors moved by the simple root $\alpha_i$ if $\alpha_i$ is also a spherical root. 

The Cartan pairing is
\[
\begin{array}{c|ccc}
      & \alpha_1 & \sigma_1 & \sigma_2 \\
\hline
D_1^+ & 1 & -1 & 0 \\
D_1^- & 1 & 0 & -1 \\
D_2 & -1 & 1 & 1 \\
D_{n} & 0 & 1 & -1 \\
D_{n+1} & 0 & -1 & 1
\end{array}
\]
The lattice $\Xi(G/P)$ is generated by $\Sigma(G/P)$. For later use, it is useful to notice that $\Chi(T)/\Xi(G/P)$ is not torsion-free: its torsion subgroup has order $2$ and it is generated by the class of the element $\tau=\frac12(\sigma_1-\sigma_2)$. This element takes non-integer values on two of the above colors.

These data for $G/P$ can also be verified directly using our results, let us give the details. It is elementary to check that they satisfy Luna's axioms in~\cite[Section~2]{MR1896179}. So, by the classification of spherical homogeneous spaces, they correspond to a spherical subgroup $K$ of $G$. The shape of the lattice $\Xi(G/K)$ implies that $K$ is connected. Indeed, the quotient $\Xi(G/K^\circ)/\Xi(G/K)$ is always torsion, and all elements of $\Xi(G/K^\circ)$ must take anyway integer values on the valuations of the colors of $G/K$ (this is clear just considering the pull-back of divisors from $G/K$ to $G/K^\circ$). Thanks to the above remark about $\Xi(G/P)$ and the element $\tau$, we conclude that $\Xi(G/K^\circ)=\Xi(G/K)$. Then $K^\circ=K$ by~\cite[Lemma~2.4]{MR2785497}.

The set $\{ D_n, D_{n+1} \}$ is easily seen to be distinguished, so it corresponds to an inclusion $K\subseteq R$ with $R$ a connected spherical subgroup of $G$. The spherical roots, the colors and the lattice of $G/R$ are computed using Proposition~\ref{prop:morphisms}, and they turn out to be the same as $H$ (see~\cite{MR3345829}, case~14), so we may assume $R=H$. Proposition~\ref{prop:PinH} assures that $K$ is conjugated to a parabolic subgroup of $H$. It remains to decide whether it's our $P$ or the one given by $I=\{\beta_{n-1}\}$, but we notice that the homogeneous spherical datum of $G/K$ remains unchanged if we exchange $\alpha_n$ with $\alpha_{n+1}$. It follows that these two parabolic subgroups of $H$ are both conjugated to $K$ in $G$, and we simply assume $K=P$.

We take $Q$ to be the parabolic subgroup of $G$ containing $B^-$ and such that only $\alpha_1$ and $\alpha_{n+1}$ are not simple roots of its Levi subgroup $L_Q$ (with the same notations as in Section~\ref{s:extwm}). Then $Q$ is a parabolic subgroup of $G$ minimal containing $P$. By looking at which simple root moves which color, one sees immediately that the colors of $G/P$ mapping not dominantly to $G/Q$ are $D_{n+1}$ and either $D_1^+$ or $D_1^-$. Since $\{D_1^+,D_2,D_{n+1}\}$ is not a parabolic subset of colors, we conclude that $D_1^-$ is mapped not dominantly to $G/Q$.

The above implies that $\chi_{D_1^-}$ and $\chi_{D_{n+1}}$ are the extension to $P$ of the characters resp.\ $-\omega_1$ and $-\omega_{n+1}$ of $T$, thanks to the Peter-Weyl theorem. For simplicity we do not distinguish notationally these characters from their extensions, and it is convenient to express them in terms of the above weights of $H$, obtaining 
\[
\begin{array}{lcl}
\chi_{D_1^-} & = & -\varpi_0,\\
\chi_{D_{n+1}} & = & -\varpi_{n}.
\end{array}
\]

The equalities~(\ref{eq:sphroot}) of Lemma~\ref{lemma:sphroot} yield the $P$-weights of the other colors:
\[
\begin{array}{lcl}
\chi_{D_1^+} & = & \varpi_0,\\
\chi_{D_2} & = & 0,\\
\chi_{D_{n}} & = & \varpi_0-\varpi_{n},\\
\end{array}
\]

Finally, Lemma~\ref{lemma:foschi} yields
\[
\begin{array}{lcl}
\omega_{D_1^+} & = & \omega_1,\\
\omega_{D_1^-} & = & \omega_{1}\\
\omega_{D_2} & = & \omega_{2},\\
\omega_{D_{n}} & = & \omega_{n},\\
\omega_{D_{n+1}} & = & \omega_{n+1}.
\end{array}
\]

This finishes the computation of $\widetilde\Gamma(G/P)$.

%%%%%%%%%%%%%%%%%%%%%%%%%%%%%%%%
%%%%%                                         %%%%%%%%%%%%%%
%%%%%          NEW SECTION          %%%%%%%%%%%%%%
%%%%%                                         %%%%%%%%%%%%%%
%%%%%%%%%%%%%%%%%%%%%%%%%%%%%%%%

\section{Application to orthogonal polynomials}\label{s:ops}
 
Throughout this section $(G,H,P)$ is a strictly indecomposable multiplicity free system with $G$ simply connected and $G,H$ connected, and $\chi\in\Chi_{+}(T_{H})\cap\Chi(P)$. As we have discussed in Section \ref{section:Introduction}, the irreducible $H$-representation $\pi_{H,\chi}:H\to\GL(V_{H,\chi})$ of highest weight $\chi$ induces multiplicity free to $G$. In this section we study the space
$$E^{\chi}:=(\bbC[G]\otimes\End(V_{H,\chi}))^{H\times H},$$
where $H\times H$ acts on the space $\bbC[G]$ by the left and right regular representation and on the space $\End(V_{H,\chi})$ by left and right matrix-multiplication.
In particular, if $\chi=0$, then $E^{0}$ is the algebra of $H$-biinvariant functions on $G$. The space $E^{\chi}$ is a module over $E^{0}$. In this section we study the algebra structure of $E^{0}$ and the $E^{0}$-module structure of $E^{\chi}$. It turns out that in most of the cases $E^{0}$ is a polynomial algebra and that $E^{\chi}$ is freely and finitely generated as an $E^{0}$-module. This structure gives rise to an abundance of examples of matrix-valued orthogonal polynomials in several variables, generalizing the existing examples from \cite{MR3550926,MR4053617,MR3801483}.

\subsection{Multiplicity free induction}

To describe the irreducible $G$-representations $\pi_{G,\lambda}$ that contain $\pi_{H,\chi}$ upon restriction to $H$ we introduce the following definition.

\begin{definition}
With the notation from above,
$$\calX_{+}(T_{G};H,\chi):=\{\lambda\in \Chi_{+}(T_{G})|\,\dim\Hom_{H}(V_{H,\chi},V_{G,\lambda})=1\}.$$
We refer to this set as the {\em $\chi$-well}.
\end{definition}
Note that $\lambda\in\calX_{+}(T_{G};H,\chi)$ if and only if $(\lambda,-\chi)\in\widetilde{\Gamma}(G/P)$.

\subsubsection{Analysis of the $0$-well.}
The $0$-well $\calX(T_{G};H,0)$ of the spherical pair $(G,H)$ is equal to the weight monoid $\Gamma(G/H)$ consisting of highest weights of irreducible $G$-submodules of $\bbC[G/H]$. If $H$ has a zero-dimensional center, then $\Gamma(G/H)$ can be identified with $\widetilde{\Gamma}(G/H)$ and it follows from Proposition \ref{prop:generators} that the weight monoid of $G/H$ is freely and finitely generated.
The generators of $\widetilde{\Gamma}(G/H)$ are recorded in \cite[Tabelle 1]{MR528837} for $G$ simple and \cite[Table 1]{MR2779106} for $G$ semisimple but not simple.

If the center of $H$ is not zero-dimensional, then it follows from inspection of the classification of strictly indecomposable spherical pairs \cite{MR906369,MR528837}, that the center of $H$ is one-dimensional. If $(G,H)$ is symmetric, then $\Gamma(G/H)$ is freely generated \cite[Lemme 3.4]{MR1076251}.

The spherical pairs $(G,H)$ that are left, i.e.~the spherical but non-symmetric ones for which $H$ has a one-dimensional center, are recorded in Table \ref{table: affine spherical dimZ(H)=1}. In the third column we provide the generators of $\widetilde{\Gamma}(G/H)$, which are taken from \cite{MR2779106,MR3053223,MR1424449}

\begin{tiny}
\begin{table}[h!]
\begin{center}
$$
\begin{array}{c|c|c|c}
& G & H & \mbox{generators of $\widetilde{\Gamma}(G/H)$}\\ \hline
1& \SL(2n+1) & \Sp(2n)\times\bbC^{\times} & (\omega_{2i-1},-\frac{n+1-i}{n}\chi), 1\leq i \leq n\\
& n\ge2 && (\omega_{2j},-\frac{j}{n}\chi), 1\leq j \leq n\\ \hline
2 & \Sp(2n) & \Sp(2n-2)\times\bbC^{\times}& (\omega_{1},\omega_{1}), (\omega_{1},-\omega_{1}),(\omega_{2},0)\\
\hline
3 & \Spin(10) & \Spin(7)\times\bbC^{\times} & (\omega_{1},2\chi),(\omega_{1},-2\chi)\\ 
&&& (\omega_{2},0),(\omega_{4},\chi),(\omega_{5},-\chi)\\ \hline \hline
4 & \SL(n+1)\times\SL(n) & \SL(n)\times\bbC^{\times} & (\omega_{i}+\omega'_{n+1-i},(n+1-i)\chi),1\leq i\leq n\\
&&& (\omega_{j}+\omega'_{n-j},-j\chi),1\leq j\leq n \\\hline
5 & \SL(n)\times\Sp(2m) & \bbC^{\times}\cdot(\SL(n-2)\times\SL(2)\times\Sp(2m-2)) &
(\omega_{n-2},2\chi),(\omega_{n-1}+\omega'_{1},\chi), \\
&n\ge3,m\ge1&&(\omega_{1}+\omega'_{1},-\chi),(\omega_{2},-2\chi) \\
&&&(\omega_{1}+\omega_{n-1},0)\quad(n\ge4)\\
&&&(\omega'_{2},0)\quad m\ge2
\end{array}
$$
\caption{The strictly indecomposable spherical pairs that are not symmetric with $\dim(Z(H))=1$ and the generators of their extended weight monoids. In the fourth row we employ the convention $\omega'_{0}=\omega'_{n}=0$. }
\label{table: affine spherical dimZ(H)=1}
\end{center}
\end{table}
\end{tiny}

\begin{lemma}\label{lemma: freely generated} The pair $(\Sp(2n),\Sp(2n-2)\times\bbC^{\times})$ is the only pair $(G,H)$ in Table \ref{table: affine spherical dimZ(H)=1} whose weight monoid $\Gamma(G/H)$ is freely generated.
\end{lemma}

\begin{proof}
First note that the weight monoid of the pair $(\Sp(2n),\Sp(2n-2)\times\bbC^{\times})$ is freely generated by $2\omega_{1}$ and $\omega_{2}$. For the pair in the first line we have
$$\Chi_{+}(T_{G};H,0)=\left\{\sum_{i=1}^{2n}a_{i}\omega_{i}\left|\sum_{j=1}^{n}(n-j+1)a_{2j-1}=\sum_{j=1}^{n}ja_{2j}\right.\right\}.$$
The elements that are sums of two or three fundamental weights are
\begin{itemize}
\item $\omega_{i}+\omega_{2n+1-i}, 1\leq i\leq n$,
\item $\omega_{2i-1}+\omega_{2k}+\omega_{2\ell}$ for $1\le i,k,\ell\le n$ with $k+\ell+i=n+1$,
\item $\omega_{2k-1}+\omega_{2\ell-1}+\omega_{2i}$ with $1\le i,k,\ell\le n$ with $k+\ell+i=2(n+1)$,
\end{itemize}
and unless $n=1$ there are strictly more than $2n$ of these elements. Hence, for $n>1$, the weight monoid is not freely generated. Indeed, if it were, then one of these elements woul not be a generator. But none of these elements can be expressed as a linear combination of the others with coefficients in $\bbZ_{\ge0}$, a contradiction.
The arguments for lines 3--5 are similar and the details are left for the reader.
% Crude arguments are given in an earlier version (20181120).
\end{proof}

\begin{remark}\label{remark: two colors}
Note that for a strictly indecomposable spherical pair $(G,H)$ the weight monoid $\Gamma(G/H)$ is freely generated if and only if there are at most two colors on $G/H$ with non-trivial $H$-weight. 
For the symmetric pairs this follows from \cite[Lemme 3.4]{MR1076251}, for the non-symmetric ones this follows from Lemma \ref{lemma: freely generated}.
Furthermore, note that whenever $G/H$ has more than $2$ colors with non-trivial $H$-weight, there are no non-trivial parabolic subgroups $P\subseteq H$ that remain spherical in $G$.
\end{remark}

\subsubsection{Analysis of a more general $\chi$-well.}

In this paragraph we assume that $\Gamma(G/H)$ is freely generated.
If $\lambda\in\calX_{+}(T_{G};H,\chi)$ and $\sigma\in\calX_{+}(T_{G};H,0)$, then $\lambda+\sigma\in\calX_{+}(T_{G};H,\chi)$ by means of the Cartan projection $V_{G,\lambda}\otimes V_{G,\sigma}\to V_{G,\lambda+\sigma}$, which is surjective and $G$-equivariant. This suggests that the $\chi$-well has certain minimal elements.

\begin{definition}\label{def: bottom}
Let
$$\calB_{+}(T_{G};H,\chi)=\{\lambda\in\calX_{+}(T_{G};H,\chi)|\forall \sigma\in\calX_{+}(T_{G};H,0):\lambda-\sigma\not\in\calX_{+}(T_{G};H,\chi)\},$$
which we call the {\em bottom} of the $\chi$-well.
\end{definition}
\begin{example}\label{example:G2A2}
Consider the spherical pair $(\mathsf G_2,\SL(3))$. The parabolic subgroup $P_{\{\alpha_{1}\}}\subset\SL(3)$ remains spherical in $\mathsf{G_{2}}$ and the extended weight monoid $\widetilde{\Gamma}(\SL(3)/P_{\{\alpha_{1}\}})$ is generated by $(\omega_{1},0),(\omega_{1},-\varpi_{1}),(\omega_{2},-\varpi_{1})$. Consider the dominant weight $3\varpi_{1}$ for $\SL(3)$. The $3\varpi_{1}$-well is given by $$a\omega_{1}+b\omega_{1}+c\omega_{2}\quad\mbox{with $a,b,c\in\bbZ_{\ge0}$ and $b+c=3$.}$$
The bottom of the $3\varpi_{1}$-well is given by
$$3\omega_{1},2\omega_{1}+\omega_{2},\omega_{1}+2\omega_{2},3\omega_{2}.$$
The $3\varpi_{1}$-well is illustrated in Figure \ref{figure: tensor G2}. The roots of $\SL(3)$ are the long roots of $\mathsf{G_{2}}$.
\end{example}
%-----
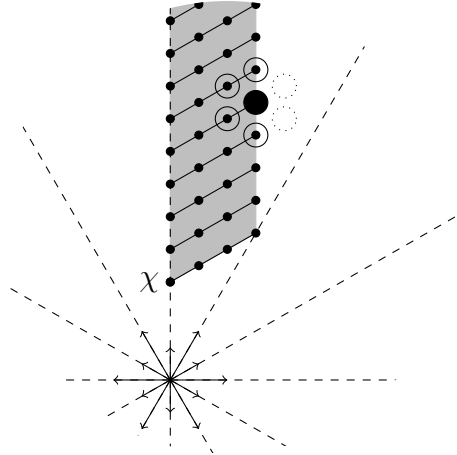
\begin{figure}[ht]

% spherical tensot product for G_2, A_2

\begin{center}

\begin{tikzpicture}[scale=.75]
\pgfmathsetmacro\ax{1}
\pgfmathsetmacro\ay{0}
\pgfmathsetmacro\bx{1 * cos(120)}
\pgfmathsetmacro\by{1 * sin(120)}
\pgfmathsetmacro\lax{2*\ax/3 + \bx/3}
\pgfmathsetmacro\lay{2*\ay/3 + \by/3}
\pgfmathsetmacro\lbx{\ax/3 + 2*\bx/3}
\pgfmathsetmacro\lby{\ay/3 + 2*\by/3}

%-------------------------------------------

%----- lattice to be CLIPPED properly ------

%-------------------------------------------

\begin{scope}
\clip (1,2.7) circle (4);
\draw [fill,lightgray] (0,10) -- (0,3*\lby) -- (3*\lax+3*\lbx,3*\lay+3*\lby) -- (3*\lax,15) -- cycle ;

\foreach \k in {1,...,12}
\draw[dashed] (0,0) -- (\k * 30 + 30:12);

\draw[thin,->] (0,0) -- (\ax,\ay);
\draw[thin,->] (0,0) -- (\bx,\by);
\draw[thin,->] (0,0) -- (-\ax,-\ay);
\draw[thin,->] (0,0) -- (-\bx,-\by);
\draw[thin,->] (0,0) -- (\ax+\bx,\ay+\by);
\draw[thin,->] (0,0) -- (-\ax-\bx,-\ay-\by);
\draw[thin,->] (0,0) -- (\lax,\lay);
\draw[thin,->] (0,0) -- (\lbx,\lby);
\draw[thin,->] (0,0) -- (-\lax,-\lay);
\draw[thin,->] (0,0) -- (-\lbx,-\lby);
\draw[thin,->] (0,0) -- (\lax-\lbx,\lay-\lby);
\draw[thin,->] (0,0) -- (\lbx-\lax,\lby-\lay);

%
%%----------------------------------
%
%%----- end of basis ---------------
%
%%----------------------------------
%
%
%
\foreach \a in {0,...,10}
\draw[thin] (\a*\lbx+3*\lbx, \a*\lby+3*\lby) -- (3*\lax+3*\lbx+\a*\lbx,3*\lay+3*\lby+\a*\lby);
\foreach \v in {0,...,3}
\foreach \d in {0,...,10}
\draw[fill](3*\lbx+\d*\lbx+\v*\lax,3*\lby+\d*\lby+\v*\lay) circle (2pt);
\draw[fill] (0,3*\lby) circle (2pt) node[black,left] {\(\chi\)};
\draw (3*\lax+6*\lbx,3*\lay+6*\lby) circle (6pt);
\draw[fill, black] (3*\lax+7*\lbx,3*\lay+7*\lby) circle (6pt);
\draw (3*\lax+8*\lbx,3*\lay+8*\lby) circle (6pt);
\draw (2*\lax+7*\lbx,2*\lay+7*\lby) circle (6pt);
\draw (2*\lax+8*\lbx,2*\lay+8*\lby) circle (6pt);
\draw[dotted] (4*\lax+6*\lbx,4*\lay+6*\lby) circle (6pt);
\draw[dotted] (4*\lax+7*\lbx,4*\lay+7*\lby) circle (6pt);
\end{scope}
\end{tikzpicture}
\end{center}
\caption{The $\chi=3\varpi_{1}$-well for the pair $(\mathsf{G_{2}},\SL(3))$. The nodes are explained in Example \ref{example:tensor}}\label{figure: tensor G2}
\end{figure}

%-----

If the center of $H$ is one-dimensional, then there are two colors $D_{0}$ and $D_{1}$ of $G/H$ whose $H$-characters are non-trivial, by Remark \ref{remark: two colors}. Let $(\omega_{D_{0}},\mu)$ and $(\omega_{D_{1}},-\mu)$ be the corresponding generators of $\widetilde{\Gamma}(G/H)$. Let $r=\rk(G/H)$ denote the spherical rank. If $r>1$, then the other colors of $G/H$ are denoted by $D_{2},\ldots,D_{r}$ and the corresponding generators of $\widetilde{\Gamma}(G/H)$ are denoted by $(\omega_{D_{i}},0)$ with $2\leq i\leq r$.

If the center of $H$ is zero-dimensional, then the colors of $G/H$ are denoted by $D_{1},\ldots, D_{r}$. The corresponding generators of $\widetilde{\Gamma}(G/H)$ are denoted by $(\omega_{D_{i}},0)$ with $1\leq i\leq r$.

In both cases the colors of $G/H$ can be pulled back to $G/P$ where they correspond to generators of $\widetilde{\Gamma}(G/P)$ given by the same weights. We denote the other colors of $G/P$ by $E_{1},\ldots,E_{s}$ and the corresponding generators by $(\omega_{E_{j}},\chi_{E_{j}})$ with $1\leq i \leq s$.

For an element $\lambda\in\calX_{+}(T_{G};H,\chi)$ we write 
\begin{eqnarray}\label{eqn: coeffs in EWS center}
(\lambda,-\chi)=a_{0}(\omega_{D_{0}},\mu)+a_{1}(\omega_{D_{1}},-\mu)+\sum_{i=2}^{r}a_{i}(\omega_{D_{i}},0)+\sum_{j=1}^{s}b_{j}(\omega_{E_{i}},\chi_{E_{j}})
\end{eqnarray}
if $H$ has a one dimensional center and
\begin{eqnarray}\label{eqn: coeffs in EWS no center}
(\lambda,-\chi)=\sum_{i=1}^{r}a_{i}(\omega_{D_{i}},0)+\sum_{j=1}^{s}b_{j}(\omega_{E_{i}},\chi_{E_{j}})
\end{eqnarray}
if $H$ has a zero-dimensional center, where the coefficients $a_{i}, b_{j}$ are in $\ZZ_{\geq0}$. 

\begin{lemma}\label{lemma: bottom}
An element $\lambda\in\calX_{+}(T_{G};H,\chi)$ is contained in $\calB_{+}(T_{G};H,\chi)$ if and only if $\min(a_{0},a_{1})=a_{2}=\ldots=a_{r}=0$ in case $\dim Z(H)=1$ or $a_{1}=\ldots=a_{r}=0$ in case $\dim Z(H)=0$. 
\end{lemma}

\begin{proof}
If $\dim Z(H)=0$, then the statement is clear from the definition of the bottom of the $\chi$-well. Suppose that $\dim Z(H)=1$ and $\lambda\in\calB_{+}(T_{G};H,\chi)$. If $a_{i}\ne 0$ for some $2\leq i\leq r$, then $\lambda-\omega_{D_{i}}\in\calX_{+}(T_{G};H,\chi)$, a contradiction to $\lambda\in \calB_{+}(T_{G};H,\chi)$. If $\min(a_{0},a_{1})>0$, then $\lambda-(\omega_{D_{0}}+\omega_{D_{1}})\in\calX_{+}(T_{G};H,\chi)$, again a contradiction to $\lambda\in \calB_{+}(T_{G};H,\chi)$. The other direction is clear.
\end{proof}

\begin{proposition}\label{prop: chi well}
The $\chi$-well can be written as
$$\Chi_{+}(T_{G};H,\chi)=\calB_{+}(T_{G};H,\chi)+\calX_{+}(T_{G};H,0)$$
and the induced decomposition of any element is unique.
\end{proposition}

\begin{proof}
The equality of the two sets is clear, the one inclusion follows from Definition \ref{def: bottom}, while the other follows from the Cartan projection. To show uniqueness
we write $\lambda=b(\lambda)+s(\lambda)$ with $b(\lambda)\in\calB_{+}(T_{G};H,\chi)$ and $s(\lambda)\in\calX_{+}(T_{G};H,0)$.
By Lemma \ref{lemma: bottom} the coefficients of $(b(\lambda),-\chi)$ and $(s(\lambda),0)$ are uniquely determined by the coefficients of $(\lambda,-\chi)$. Hence the decomposition is unique.
\end{proof}

The action of $H$ on the annihilator $\lah^{\perp}\subseteq\lag^{*}$ of $\lah$ has a generic stabilizer $H_{*}\subseteq H$ which is reductive, see \cite{MR1029388}. It is known \cite{MR3801483} that the decomposition of the restriction of $\pi_{H,\chi}$ to $H_{*}$ is multiplicity free. Moreover, we can choose $T_{G}$ in such a way that $T_{H_{*}}=T_{G}\cap H_{*}$ is a maximal torus of $H_{*}$ and the weights $\lambda\in\Gamma(G/H)$ vanish on $T_{H_{*}}$. The reason is that $H_{*}$ is contained in a Levi subgroup of the parabolic subgroup adapted to $G/H$ (i.e.\ corresponding to the subset of simple roots $S^p(G/H)$), but it also contains the commutator subgroup of this Levi subgroup. Note that we can also choose a compatible Borel subgroup $B_{H_{*}}$ of $H_{*}$ that is contained in a Borel subgroup of $G$. This implies that restriction to $T_{H_{*}}$ induces a map $\Chi_{+}(T_{G})\to\Chi_{+}(T_{H_{*}}),\lambda\mapsto\lambda_{*}$.

As explained in \cite[Proposition 2.4]{MR3801483}, the restriction of a weight $\lambda\in\calX_{+}(T_{G};H,\chi)$ to $T_{H_{*}}$ is the highest weight of an irreducible representation of $H_{*}$ that is contained in the restriction of $\pi_{H,\chi}$ to $H_{*}$. In loc.~cit.~it is shown that all irreducible $H_{*}$-representations in the decomposition of $\pi_{H,\chi}$ can be obtained in this way. For later reference we record the following result. 

\begin{lemma}\label{lemma on M-types}
The elements in $\calB_{+}(T_{G};H,\chi)$ are in one-to-one-correspondence with the $H_{*}$-types in the restriction $\pi^{H}_{\chi}|_{H_{*}}$, via the map $\lambda\mapsto\lambda_{*}$.
\end{lemma}

\begin{proof}
It remains to show that the map $\calB_{+}(T_{G};H,\chi)\to\Chi_{+}(T_{H_{*}})$ is injective. Let $\lambda,\lambda'\in\calB_{+}(T_{G};H,\chi)$ with $\lambda_{*}=\lambda'_{*}$. Then $\lambda-\lambda'\in\bbZ\Gamma(G/H)\subseteq\Xi(G/H)$ because $(\lambda,-\chi)-(\lambda',-\chi)=(\lambda-\lambda',0)$. If $\lambda-\lambda'\ne0$, then one of the coefficients $a_{i}$ or $a_{i}'$ of $(\lambda,-\chi)$ or $(\lambda',-\chi)$ in \eqref{eqn: coeffs in EWS center} of \eqref{eqn: coeffs in EWS no center} is non-zero. By Lemma \ref{lemma: bottom} this can only happen when $\dim(Z(H))=1$. Without loss of generality we assume that $a_{0}>0$ and $a_{1}=0$. Since $\lambda-\lambda'\in\bbZ\Gamma(G/H)$, we must have $a'_{0}-a'_{1}=a_{0}$ to make sure that the second component is zero. At the same time $\min(a'_{0},a'_{1})=0$ by Lemma \ref{lemma: bottom}. It follows that $a_{1}=0$ and $a'_{0}=a_{0}$, i.e.~$\lambda=\lambda'$.
\end{proof}

%%% NEW SUBSECTION %%%

\subsection{Spherical functions}\label{subsection zonal}
We proceed to investigate the algebra structure of $E^{0}$ and the $E^{0}$-module structure of $E^{\chi}$. We recall the definition of a spherical function in our framework.

\begin{definition}
Let $\lambda\in\calX(T_{G};H,\chi)$. Let $j_{\lambda}^{\chi}:V_{H,\chi}\to V_{G,\lambda}$ and $p^{\chi}_{\lambda}:V_{G,\lambda}\to V_{H,\chi}$ be an $H$-equivariant inclusion and projection with $p^{\chi}_{\lambda}\circ j^{\chi}_{\lambda}=\mathrm{Id}_{V_{H,\chi}}$. The {\em spherical function} of type $\chi$ associated to $\lambda$ is defined by $\Phi^{\mu}_{\lambda}:G\to\End(V_{H,\chi}):g\mapsto p^{\chi}_{\lambda}\circ\pi_{G,\lambda}(g)\circ j^{\chi}_{\lambda}$. If $\chi=0$, then the spherical function of type $0$ associated to $\lambda\in\calX(T_{G};H,0)$ is called a {\em zonal spherical function} associated to $\lambda$ and it is denoted by $\phi_{\lambda}$.
\end{definition}

For later reference we introduce the following notation. If $K\subseteq G$ is a subgroup and $\mu:K\to\bbC^{\times}$, then we denote by
$$(V_{G,\lambda})^{(K)}_{(\mu)}:=\{v\in V_{G,\lambda}:kv=\mu(k)v\}$$
the weight space of $K$ of weight $\mu$. We denote $(V_{G,\lambda})^{K}:=(V_{G,\lambda})^{(K)}_{(0)}$, the space of $K$-invariants.

If $v\in V^{G}_{\lambda}$ and $\eta\in(V^{G}_{\lambda})^{*}$ then $m^{\lambda}_{v\otimes\eta}\in\bbC[G]$ denotes the matrix coefficient defined by $m^{\lambda}_{v\otimes\eta}(g)=\eta(g^{-1}v)$.

\begin{lemma}\label{lemma: mc}
Let $\lambda\in\calX(T_{G};H,0)$ and $\lambda'\in\calX(T_{G};H,\chi)$. Let $m^{\lambda}_{v\otimes\eta},m^{\lambda'}_{v'\otimes\eta'}\in\bbC[G]$ with $v\in (V_{G,\lambda})^{H}$, $\eta\in ((V_{G,\lambda})^{*})^{H}$, $v'\in V_{G,\lambda'}$ and $\eta'\in((V_{G,\lambda'})^{*})^{(P)}_{(-\chi)}$. 
Then
$$m^{\lambda}_{v\otimes\eta}m^{\lambda'}_{v'\otimes\eta'}=\sum_{\beta\in\ZZ_{\geq0}\Sigma(G/P)}c^{\lambda,\lambda'}_{\lambda+\lambda'-\beta}m^{\lambda+\lambda'-\beta}_{v''\otimes\eta''}$$
for some $v''\in V_{G,\lambda+\lambda'-\beta}$ and $\eta''\in ((V_{G,\lambda+\lambda'-\beta})^{*})^{(P)}_{(-\chi)}$. Moreover, $c^{\lambda,\lambda'}_{\lambda+\lambda'}\ne0$.
\end{lemma}

\begin{proof}
Let $M,M'\subseteq\bbC[G]$ be the irreducible $G$-modules that contain the matrix coefficients $m^{\lambda}_{v\otimes\eta},m^{\lambda'}_{v'\otimes\eta'}$ respectively. If $M''\subseteq M\cdot M'$ is an irreducible submodule of highest weight $\lambda''$, then $\lambda+\lambda'-\lambda''$ is a linear combination of spherical roots of $G/P$ with coefficients in $\ZZ_{\geq0}$. The non-vanishing follows from the fact that the Cartan projection $V_{G,\lambda}\otimes V_{G,\lambda'}\to V_{G,\lambda+\lambda'}$  is surjective and that simple tensors are not in the kernel.
\end{proof}

%% new subsubsection

\subsubsection{Zonal spherical functions}

The zonal spherical function associated to $\lambda\in\Gamma(G/H)$ can be written as a matrix coefficient. Indeed, let $v^{H}\in V_{G,\lambda}$ and $\eta^{H}\in(V_{G,\lambda})^{*}$ be two non-trivial $H$-fixed vectors for which $m^{\lambda}_{v^{H}\otimes\eta^{H}}(e)=1$. Then $m^{\lambda}_{v^{H}\otimes\eta^{H}}=\phi_{\lambda}$.

Let $\lambda_{1},\ldots,\lambda_{s}$ be elements of $\Gamma(G/H)$ and let $\phi_{\lambda_{1}},\ldots,\phi_{\lambda_{s}}$ be the corresponding zonal spherical functions. The elements $\lambda_{1},\ldots,\lambda_{s}$ generate $\Gamma(G/H)$ as a monoid if and only if the zonal spherical functions $\phi_{\lambda_{1}},\ldots,\phi_{\lambda_{s}}$ generate $E^{0}$ as an algebra.

To see this, suppose $\Gamma(G/H)$ is generated by $\lambda_{1},\ldots,\lambda_{s}$. If $\lambda\in\Gamma(G/H)$, then Lemma \ref{lemma: mc} implies
$$\phi_{\lambda}\phi_{\lambda_{i}}=\sum_{\beta\in\ZZ_{\geq0}\Sigma_{G/H}}c_{\lambda,\lambda_{i},\beta}\phi_{\lambda+\lambda_{i}-\beta}$$
with $c_{\lambda,\lambda_{i},0}\ne0$. Note that $\lambda+\lambda_{i}-\beta\le\lambda+\lambda_{i}$, where $\le$ denotes the standard partial ordering on $\Chi(T_{G})$. Since there are only finitely many elements $\mu\in\Gamma(G/H)$ with $\mu\le\lambda+\lambda_{i}$,
an induction argument shows that every zonal spherical function can be expressed as a polynomial of the zonal spherical functions $\phi_{\lambda_{i}}$ with $1\leq i\leq s$.
Conversely, if the $\phi_{\lambda_{1}},\ldots,\phi_{\lambda_{s}}$ generate $E^{0}$, then it is clear that $\Gamma(G/H)$ is generated by $\lambda_{1},\ldots,\lambda_{s}$.

\begin{lemma}
The algebra $E^{0}$ is freely generated by $\phi_{\lambda_{1}},\ldots,\phi_{\lambda_{r}}$ if and only if $\Gamma(G/H)$ is freely generated by $\lambda_{1},\ldots,\lambda_{r}$.
\end{lemma}

\begin{proof}
Suppose that $E^{0}$ is not freely generated. Consider the map $\zeta:\bbC[z_{1},\ldots,z_{r}]\to E^{0}$ given by $z^{a}\mapsto\prod_{i}\phi_{\lambda_{i}}^{a_{i}}$ and let $p(z)=\sum_{a\in\ZZ_{\geq0}^{r}}c_{a}z^{a}$ with $\zeta(p)=0$, where we use multi-index notation.
Let $C(p)=\{a\in\ZZ_{\geq0}^{r}:c_{a}\ne0\}$. For $a\in\ZZ_{\geq0}^{r}$ let $\lambda(a)=\sum_{i=1}^{r}a_{i}\lambda_{i}$.
If $C(p)=\emptyset$, then $p=0$. Suppose $p\ne0$ so that $C(p)\ne\emptyset$. The set $C'(p)=\{(\lambda(a),a):a\in C(p)\}$ is non-empty. Let $(\lambda(a),a)\in C'(p)$ with $\lambda(a)$ maximal in the partial ordering $\le$. Then $\zeta(p)=c\phi_{\lambda(a)}+$ other terms, with $c\ne0$. Since $\zeta(p)=0$ and the spherical functions are linearly independent, there must be $a'\in C(p)$ different from $a$ with $(\lambda(a'),a')\in C'(p)$ and $\lambda(a')=\lambda(a)$, because $c\phi_{\lambda(a)}$ must be canceled. It follows that $\Gamma(G/H)$ is not free.

To prove the converse implication, suppose that the generators $\lambda_{1},\ldots,\lambda_{r}$ are not free. Then there exist two disjoint subsets of indices $1\le i_{1}<\cdots<i_{s}\le r$ and $1\le j_{1}<\ldots<j_{t}\le r$ and coefficients $a_{i_{1}},\ldots,a_{i_{s}},b_{j_{1}},\ldots,b_{j_{t}}$ for which
$$\sum_{k=1}^{s}a_{i_{k}}\lambda_{i_{k}}=\sum_{\ell=1}^{s}b_{j_{\ell}}\lambda_{j_{\ell}}$$
and for which $\lambda=\sum_{k=1}^{s}a_{i_{k}}\lambda_{i_{k}}$ is minimal with respect to $\le$.
Let $p(z_{1},\ldots,z_{r})=\prod_{k=1}^{s}z_{i_{k}}^{a_{i_{k}}}$ and $q(z_{1},\ldots,z_{r})=\prod_{\ell=1}^{t}z_{j_{\ell}}^{b_{j_{\ell}}}$. Then
$$
p(\phi_{\lambda_{1}},\ldots,\phi_{\lambda_{r}})=\sum_{\lambda'\le\lambda}c(\lambda',\lambda)\phi_{\lambda'},\quad
q(\phi_{\lambda_{1}},\ldots,\phi_{\lambda_{r}})=\sum_{\lambda'\le\lambda}d(\lambda',\lambda)\phi_{\lambda'}$$
with $c(\lambda,\lambda)\ne0$ and $d(\lambda,\lambda)\ne0$. Hence
$$\zeta(d(\lambda,\lambda)p-c(\lambda,\lambda)q)=\sum_{\lambda'<\lambda}e(\lambda',\lambda)\phi_{\lambda'}$$
for some coefficients $e(\lambda,\lambda')$.
Let $r_{\lambda'}\in\bbC[z_{1},\ldots,z_{r}]$ be so that $\zeta(r_{\lambda'})=\phi_{\lambda'}$ for $\lambda'<\lambda$. Then for all monomials $z^{f}$ with $f\in C(r_{\lambda'})$ we have $\lambda(f)\le\lambda'$. In particular we note that the monomials $p$ and $q$ do not occur in any of the polynomials $r_{\lambda'}$. Hence the polynomial $d(\lambda,\lambda)p-c(\lambda,\lambda)q-\sum_{\lambda'<\lambda}e(\lambda',\lambda)r_{\lambda'}$ is non-trivial and it is mapped to $0$ by $\zeta$.
\end{proof}

We conclude that under our assumptions on $(G,H,P)$, the algebra $E^{0}$ is a polynomial algebra if and if $\Gamma(G/H)$ is freely generated.

\subsubsection{Spherical functions of type $\chi$}

We retain the assumptions on $(G,H,P)$ and $\chi$. Moreover, in this paragraph we assume that $E^{0}$ is a polynomial algebra, i.e.~the weight monoid of $G/H$ is freely generated.

\begin{theorem}\label{theorem: ffg module of sf}
The space
$E^{\chi}$
is freely and finitely generated as an $E^{0}$-module.
\end{theorem}

\begin{proof}
Let $\lambda\in\calX_{+}(T_{G};H,\chi)$ and let $\lambda_{i}\in\Gamma(G/H)$ be a generator. Upon writing the spherical function $\Phi^{\chi}_{\lambda}$ in coordinates subject to a weight basis of $V_{H,\chi}$, it will have a matrix coefficient $m^{\lambda}_{v\otimes\eta}$ among its entries, where $v$ and $\eta$ are $P$-eigenvectors.
It follows from Lemma \ref{lemma: mc} that 
$$\phi_{\lambda_{i}}\Phi^{\chi}_{\lambda}=\sum_{\beta\in\ZZ_{\geq0}\Sigma_{G/P}}c^{\lambda,\lambda_{i}}_{\lambda+\lambda_{i}-\beta}(\chi)\Phi^{\chi}_{\lambda+\lambda_{i}-\beta}$$
with $c^{\lambda,\lambda_{i}}_{\lambda+\lambda_{i}}\ne0$. 
Since the number of $\lambda'\in\calX_{+}(T_{G};H,\chi)$ with $\lambda'\le\lambda$ is finite, an induction argument shows that
$\Phi^{\chi}_{\lambda}$ can be expressed as an $E^{0}$-linear combination of the spherical functions $\Phi^{\chi}_{b}$ with $b\in\calB_{+}(T_{G};H,\chi)$.
This shows that $E^{\chi}$ is finitely generated as an $E^{0}$-module.
Let $p_{b}\in E^{0}$ with $b\in\calB_{+}(T_{G};H,\chi)$ be a family of polynomials and consider
\begin{equation}\label{eq: free module}
\sum_{b\in\calB_{+}(T_{G};H,\chi)}p_{b}\Phi^{\chi}_{b}=\sum_{\lambda\in\calX(T_{G};H,\chi)}c(\lambda)\Phi^{\chi}_{\lambda}.
\end{equation}
Similarly write $p_{b}\Phi^{\chi}_{b}=\sum_{\lambda\in\calX(T_{G};H,\chi)}c_{b}(\lambda)\Phi^{\chi}_{\lambda}$.
The polynomial $p_{b}$ can be written as a linear combination of zonal spherical functions, say with spherical weights in the set $P(p_{b})$. If $s\in P(p_{b})$ is maximal, then $c_{b}(b+s)\ne0$ by Lemma \ref{lemma: mc}. Let $\max(P(p_{b}))$ denote the set of maximal elements in $P(p_{b})$. Let $\lambda\in\cup_{b}(b+\max(P(p_{b})))$ be maximal and write $\lambda=b(\lambda)+s(\lambda)$ according to Proposition \ref{prop: chi well}.

We claim that $c_{b}(\lambda)=0$ if $b\ne b(\lambda)$. Suppose that $c_{b}(\lambda)\ne0$. Then $\lambda=b+s-\beta$ for some $s\in P(p_{b})$ and $\beta\in\bbZ_{\ge0}\Sigma(G/P)$, by Lemma \ref{lemma: mc}. If $s\in P(p_{b})$ is not maximal, then there is an element $s'\in P(p_{b})$ with $s'=s+\gamma$, for $\gamma$ a linear combination of roots with non-negative integer coefficients. Hence $\lambda=b+s'-\gamma-\beta$. If $\gamma+\beta=0$ then $b=b(\lambda)$ by Proposition \ref{prop: chi well}, contradicting the assumption that $b\ne b(\lambda)$. Hence $\gamma+\beta\ne0$. But this implies $\lambda<b+s'$ while at the same time $b+s'\in\cup_{b}(b+\max(P(p_{b})))$. This contradicts $\lambda$ being maximal and we conclude that $c_{b}(\lambda)=0$.

These arguments also show that if $p_{b}\ne0$ for some $b\in\calB_{+}(T_{G};H,\chi)$, then there is a non-zero coefficient $c(\lambda)$ in \eqref{eq: free module}. Conversely, if \eqref{eq: free module} is equal to zero, then all $p_{b}$ must be equal to zero. This shows that $E^{\chi}$ is freely generated over $E^{0}$.
\end{proof}

\begin{example}\label{example:tensor}
The algebra of $\SL(3)$-biinvariant functions on $\mathsf{G_{2}}$ is $E^{0}=\bbC[\phi_{\omega_{1}}]$. Let $\lambda=4\omega_{1}+3\omega_{2}$, the black node in Figure \ref{figure: tensor G2}. The function $\phi_{\omega_{1}}\Phi^{3\varpi_{1}}_{\lambda}$ can be expressed as a linear combination of the spherical functions 
$$\Phi^{3\varpi_{1}}_{\lambda+\omega_{1}},\Phi^{3\varpi_{1}}_{\lambda+2\omega_{1}-\omega_{2}},\Phi^{3\varpi_{1}}_{\lambda+\omega_{1}-\omega_{2}}\quad\mbox{and}\quad\Phi^{3\varpi_{1}}_{\lambda-\omega_{1}}.$$
Conversely, the spherical function $\Phi^{3\varpi_{1}}_{\lambda+\omega_{1}}$ can be expressed as an $E^{0}$-linear combination of spherical functions of type $3\varpi_{1}$ associated to dominant weights $<\lambda+\omega_{1}$. This illustrates the induction argument of Theorem \ref{theorem: ffg module of sf}.
\end{example}

\subsection{Orthogonal polynomials}

Let $(G,H,P)$ be a strictly indecomposable multiplicity free system for which $E^{0}$ is a polynomial algebra. Theorem \ref{theorem: ffg module of sf} implies that  $E^{\chi}$ is isomorphic to $E^{0}\otimes\bbC^{d_{\chi}}$, where $d_{\chi}$ is the number of elements in $\calB_{+}(T_{G};H,\chi)$.
This means that for $\lambda\in\Chi_{+}(T_{G};T_{H},\chi)$ there exist uniquely determined polynomials $p_{\lambda,b}^{\chi}\in E^{0}$ with $b\in\calB_{+}(T_{G};H,\chi)$ for which
$$\Phi^{\chi}_{\lambda}=\sum_{b\in\calB_{+}(T_{G};H,\chi)}p_{\lambda,b}^{\chi}\Phi^{\chi}_{b}.$$
This construction yields a family of vector-valued polynomial on $G$, labeled by the weights $\lambda\in\Chi_{+}(T_{G};T_{H},\chi)$. In fact, for each $\sigma\in\Gamma(G/H)$ we can group the vector-valued polynomials with label $b+\sigma, b\in\calB_{+}(T_{G};H,\chi)$ into a matrix-valued polynomial on $G$. The size of the matrices is $d_{\chi}\times d_{\chi}$. Note that $d_{\chi}=\dim(\End_{H_{*}}(V_{H,\chi}))$ by Lemma \ref{lemma on M-types}. 
The orthogonality for these vector- and matrix-valued polynomials comes from Schur orthogonality, which is given by integration over a suitable maximal compact subgroup $G_{c}\subseteq G$. For the multiplicity free systems $(G,H,P)$ where $(G,H)$ is symmetric, this integration can be further reduced to an integration over a compact torus $A_{c}\subseteq G_{c}$ by means of the Cartan decomposition of $G_{c}$.

The general set-up to obtain matrix-valued functions from multiplicity free systems has been alluded to in \cite{MR3801483} and was worked out in \cite{MR4053617} based on the specific case of $(\SL(n)\times\SL(n),\diag(\SL(n)),P)$, where $P$ is obtained by leaving out the first or last simple root. The three conditions in loc.cit.~guarantee the existence of families of matrix-valued polynomials. The first condition requires multiplicity free induction, the second that $\Chi_{+}(T_{G};T_{H},\chi)$ is of the form of Proposition \ref{prop: chi well} and the third requires a degree function on $\Chi_{+}(T_{G};T_{H},\chi)$ that was used in an argument to show that $E^{\chi}$ is finitely generated as an $E^{0}$-module. This degree function was actually a multi-degree that that established a partial ordering which was needed in the argument, but it turns out that the multi-degree does not exist in some of the other cases. In this paper, we have replaced the partial ordering from \cite{MR4053617} by the usual partial ordering on the weight lattice and everything works fine.

\section{Tables}
In the table below we gather all the indecomposable multiplicity free systems $(G,H,P)$ with the generators of the extended weight monoid $\widetilde{\Gamma}(G/P)$.

The data in the table are all derived with lenghty but elementary applications of our results of Sections~\ref{s:howto}, as done in the example of Section~\ref{s:example}. The detailed computations are in Appendices~\ref{s:sym} and~\ref{s:nonsym}.

For the notations, see Appendices~\ref{s:preliminaries},~\ref{s:sym}, and~\ref{s:nonsym}. As usual, $P$ is indicated by the missing simple root(s).

\begin{tiny}
\begin{center}
\begin{longtable}{|c|c|c|c|c|}
%\caption{}\\
\hline
 & $G$ & $H$ & $P$ & Generators $\widetilde{\Gamma}(G/P)$\\
\hline
\endfirsthead
\multicolumn{5}{c}%
{\tablename\ \thetable\ -- \textit{Continued from previous page}} \\
\hline
 & $G$ & $H$ & $P$ & Generators $\widetilde{\Gamma}(G/P)$\\\hline
\endhead
\hline \multicolumn{5}{r}{\textit{Continued on next page}} \\
\endfoot
\hline
\endlastfoot

\ref{sym1aS}&$\SL(2n)$ & $\Sp(2n)$  & $\{\beta_{1}\}$ &  $(\omega_{2i},0)$, $1\leq i\leq n-1$\\
&&&&$(\omega_{2j-1},-\varpi_{1}), 1\leq j\leq n$\\
\hline
\ref{sym1bS}&$\SL(6)$ & $\Sp(6)$  & $\{\beta_{3}\}$  & $ (\omega_{1}+\omega_{3}+\omega_{5},-\varpi_{3}),$\\
&&&& $(\omega_{1}+\omega_{4},-\varpi_{3}),$\\
&&&& $(\omega_{2}+\omega_{5},-\varpi_{3}),$ \\
&&&& $(\omega_{2},0),(\omega_{3},-\varpi_{3}),(\omega_{4},0)$ \\ \hline
\ref{sym1cS}&$\SL(4)$ & $\Sp(4)$  & $\{\beta_{1},\beta_{2}\}$  & $(\omega_{1}+\omega_{3},-\varpi_{2}),(\omega_{1},-\varpi_{1}),(\omega_{2},0)$\\
&&&& $(\omega_{2},-\varpi_{2}),(\omega_{3},-\varpi_{1})$\\
\hline \hline

\ref{sym2-1S}&$\SL(q+3)$ & $\mathrm{S}(\mathrm{L}(2)\times\mathrm{L}(q+1))$ & $\{\beta_{1}\}$ & $(\omega_{1}+\omega_{q+2},0),(\omega_{1},-\varpi_{1}),(\omega_{2},-\varpi_{2}),$ \\
&$1\le q$&&&$(\omega_{q+1},\varpi_{2}),(\omega_{q+2},\varpi_{2}-\varpi_{1})$\\
\hline

\ref{sym2-2S}&$\SL(p+q+2)$ & $\mathrm{S}(\mathrm{L}(p+1)\times\mathrm{L}(q+1))$ & $\{\beta_{1}\}$ & $(\omega_{i}+\omega_{n+1-i},0), 1\leq i\leq p,$ \\
&$2\le p\le q$&&& $(\omega_{j}+\omega_{n+2-j},-\varpi_{1}),$\\
&&&&$2\leq j\leq p,$\\
&$n=p+q+1$&&& $(\omega_{q+2},-\varpi_{1}+\varpi_{p+1}),(\omega_{1},-\varpi_{1}),$\\
&&&& $(\omega_{q+1},\varpi_{p+1}),(\omega_{p+1},-\varpi_{p+1})$\\ \hline

\ref{sym2-3S}&$\SL(p+q+2)$ & $\mathrm{S}(\mathrm{L}(p+1)\times\mathrm{L}(q+1))$ &  $\{\beta_{p}\}$ & $(\omega_{i}+\omega_{n+1-i},0), 1\leq i\leq p,$ \\
&$2\le p\le q$&&& $(\omega_{j}+\omega_{n-j},-\varpi_{p}+\varpi_{p+1}),$\\
&&&&$1\leq j\leq p-1,$\\
&$n=p+q+1$&&& $(\omega_{p},-\varpi_{p}),(\omega_{n},-\varpi_{p}+\varpi_{p+1}),$\\
&&&& $(\omega_{p+1},-\varpi_{p+1}),(\omega_{q+1},\varpi_{p+1})$\\ \hline 

\ref{sym2-4S}&$\SL(p+q+2)$ & $\mathrm{S}(\mathrm{L}(p+1)\times\mathrm{L}(q+1))$ & $\{\beta_{p+1}\}$ & $(\omega_{i}+\omega_{n+1-i},0), 1\leq i\leq p,$ \\
&$1\le p< q$&&& $(\omega_{j}+\omega_{n+2-j},\varpi_{p+1}-\varpi_{p+2}),$\\
&&&&$2\leq j\leq p+1,$\\
&$n=p+q+1$&&& $(\omega_{1},\varpi_{p+1}-\varpi_{p+2}),(\omega_{p+1},-\varpi_{p+1}),$\\
&&&& $(\omega_{p+2},-\varpi_{p+2}),(\omega_{q+1},\varpi_{p+1})$\\ \hline

\ref{sym2-5S}&$\SL(p+q+2)$ & $\mathrm{S}(\mathrm{L}(p+1)\times\mathrm{L}(q+1))$ & $\{\beta_{p+q}\}$ & $(\omega_{i}+\omega_{n+1-i},0), 1\leq i\leq p,$ \\
&$1\le p< q$&&& $(\omega_{j}+\omega_{n-j},-\varpi_{n}),$\\
&&&&$1\leq j\leq p,$\\
&$n=p+q+1$&&& $(\omega_{n},-\varpi_{n}),(\omega_{q+1},\varpi_{p+1}),$\\
&&&& $(\omega_{q},\varpi_{p+1}-\varpi_{n}),(\omega_{p+1},-\varpi_{p+1})$\\ \hline

\ref{sym2-6S}&$\SL(q+3)$ & $\mathrm{S}(\mathrm{L}(2)\times\mathrm{L}(q+1))$ & $\{\beta_{i}'\}$ & $(\omega_{1}+\omega_{i+1},-\varpi_{i+2}),(\omega_{1}+\omega_{n},0),$ \\
&$1< q$&&$2\leq i\leq q$& $(\omega_{2},-\varpi_{2}),(\omega_{i},\varpi_{2}-\varpi_{i+2}),$\\
&$n=q+2$&&& $(\omega_{i+1}+\omega_{n},\varpi_{2}-\varpi_{i+2}),$\\
&&&& $(\omega_{i+1},-\varpi_{i+2}),(\omega_{n-1},\varpi_{2})$\\ \hline

\ref{sym2-7S}&$\SL(q+2)$ & $\mathrm{S}(\mathrm{L}(1)\times\mathrm{L}(q+1))$ & $S_{H}$ & $(\omega_{i},-\varpi_{i}), 1\leq i\leq n,$ \\
&$1\le q$&&& $(\omega_{j},\varpi_{1}-\varpi_{j+1}),1\leq j\leq n-1,$\\
&$n=q+1$&&& $(\omega_{n},\varpi_{1})$\\
\hline
\hline

\ref{sym3aS}&$\SO(2n+2)$ & $\SO(2)\times\SO(2n)$ & $\{\beta_{n-1}\}$ & $(\omega_{1},\varpi_{0}),(\omega_{1},-\varpi_{0}),(\omega_{2},0),$ \\
&&&& $(\omega_{n+1},\varpi_{0}-\varpi_{n-1}),(\omega_{n},-\varpi_{n-1})$\\ \hline
\ref{sym3bS}&$\SO(2n+2)$ & $\SO(2)\times\SO(2n)$  & $\{\beta_{n}\}$  & $(\omega_{1},\varpi_{0}),(\omega_{1},-\varpi_{0}),(\omega_{2},0),$ \\
&&&& $(\omega_{n},\varpi_{0}-\varpi_{n}),(\omega_{n+1},-\varpi_{n})$\\ \hline\hline

\ref{sym4}&$\SO(2n+1)$ & $\SO(2n)$ & $S_{H}$ & $(\omega_{i},-\varpi_{i}),i\in \{1,\ldots,n-2,n\},$\\ 
&$n\ge3$&&&$(\omega_{n-1},-\varpi_{n-1}-\varpi_{n}),(\omega_{1},0)$\\
&&&&$(\omega_{j},-\varpi_{j-1}),2\leq j\leq n,$\\ \hline \hline

\ref{sym5}&$\SO(2n+2)$ & $\SO(2n+1)$  &  $S_{H}$ & $(\omega_{i},-\varpi_{i}), 1\leq i\leq n,$\\ 
&$n\ge3$&&&$(\omega_{n+1},-\varpi_{n}),(\omega_{1},0),$\\
&&&&$(\omega_{j},-\varpi_{j-1}),2\leq j\leq n$\\  \hline \hline

\ref{sym6S}&$\SO(2n+2)$ & $\GL(n+1)$ & $\{\beta_{1}\}$ & $(\omega_{2i-1},-\varpi_{1}),i=1,\ldots,\frac{n}{2},$\\ 
&$n\ge 2$ even&&& $(\omega_{2j},0),i=1,\ldots,\frac{n}{2}-1,$\\
&&&& $(\omega_{n},\frac{1}{2}\varpi_{n+1}),(\omega_{n+1},-\frac{1}{2}\varpi_{n+1}),$\\
&&&&$(\omega_{n+1},\frac{1}{2}\varpi_{n+1}-\varpi_{1})$\\\hline

\ref{sym6S}&$\SO(2n+2)$ & $\GL(n+1)$ & $\{\beta_{1}\}$ & $(\omega_{2i-1},-\varpi_{1}),i=1,\ldots,\frac{n-1}{2},$\\ 
&$n\ge 3$ odd&&& $(\omega_{2j},0),i=1,\ldots,\frac{n-1}{2},$\\
&&&& $(\omega_{n},\frac{1}{2}\varpi_{n+1}-\varpi_{1}),(\omega_{n+1},-\frac{1}{2}\varpi_{n+1}),$\\
&&&&$(\omega_{n+1},\frac{1}{2}\varpi_{n+1})$\\ \hline 

\ref{sym6S}&$\SO(2n+2)$ & $\GL(n+1)$ & $\{\beta_{n}\}$ & $(\omega_{2i-1},-\varpi_{n}),i=1,\ldots,\frac{n}{2},$\\ 
&$n\ge 2$ even&&& $(\omega_{2j},0),i=1,\ldots,\frac{n}{2}-1,$\\
&&&& $(\omega_{n},-\frac{1}{2}\varpi_{n+1}),(\omega_{n+1},\frac{1}{2}\varpi_{n+1}),$\\
&&&&$(\omega_{n+1},-\frac{1}{2}\varpi_{n+1}-\varpi_{n})$\\\hline

\ref{sym6S}&$\SO(2n+2)$ & $\GL(n+1)$ & $\{\beta_{2}\}$ & $(\omega_{2i-1},-\varpi_{n}),i=1,\ldots,\frac{n-1}{2},$\\ 
&$n\ge 3$ odd&&& $(\omega_{2j},0),i=1,\ldots,\frac{n-1}{2},$\\
&&&& $(\omega_{n},-\frac{1}{2}\varpi_{n+1}-\varpi_{n}),(\omega_{n+1},\frac{1}{2}\varpi_{n+1}),$\\
&&&&$(\omega_{n+1},-\frac{1}{2}\varpi_{n+1})$\\ \hline \hline

\ref{sym7aS} & $\Sp(2p+2q)$ & $\Sp(2p)\times \Sp(2q)$ & $\{\beta_{1}\}$ & $(\omega_{2i},0),1\leq i\leq q,$ \\
 &$p>q$&&& $(\omega_{2j-1},-\varpi_{1}),1\leq i\leq q+1$\\ \hline

\ref{sym7aS} & $\Sp(2p+2q)$ & $\Sp(2p)\times \Sp(2q)$ & $\{\beta_{1}\}$ & $(\omega_{2i},0),1\leq i\leq p,$ \\
 &$p\le q$&&& $(\omega_{2j-1},-\varpi_{1}),1\leq i\leq p$\\ \hline 

% p=3, q\ge 3

\ref{sym7cS} & $\Sp(6+2q)$ & $\Sp(6)\times \Sp(2q)$ & $\{\beta_{3}\}$ & $(\omega_{1}+\omega_{4},-\varpi_{3}), (\omega_{1}+\omega_{3}+\omega_{5},-\varpi_{3}),$\\
 &$q\ge3$&&& $(\omega_{3},-\varpi_{3}),(\omega_{4},0),(\omega_{6},0),$\\ 
 &&&&$(\omega_{2}+\omega_{5},-\varpi_{3}),(\omega_{2},0)$\\ \hline

% p=3, q=2

\ref{sym7cS} & $\Sp(10)$ & $\Sp(6)\times \Sp(4)$ & $\{\beta_{3}\}$ & $(\omega_{1}+\omega_{4},-\varpi_{3}),(\omega_{1}+\omega_{3}+\omega_{5},-\varpi_{3}),$\\
 &&&& $(\omega_{2}+\omega_{5},-\varpi_{3}),(\omega_{2},0),(\omega_{3},-\varpi_{3}),(\omega_{4},0)$\\ \hline

% p=2 q\ge2

\ref{sym7cS} & $\Sp(4+2q)$ & $\Sp(4)\times \Sp(2q)$ & $\{\beta_{2}\}$ & $(\omega_{1}+\omega_{3},-\varpi_{2}),(\omega_{2},0),$\\
 &&&&$(\omega_{2},-\varpi_{2}),(\omega_{4},0)$ \\ \hline

%q=2, p\ge4

\ref{sym7cS} & $\Sp(2p+4)$ & $\Sp(2p)\times \Sp(4)$ & $\{\beta_{2}\}$ & $(\omega_{1}+\omega_{q+1},-\varpi_{p}),(\omega_{1}+\omega_{3}+\omega_{p+2},-\varpi_{p}),$\\
 &$p\ge4$&&&$(\omega_{2}+\varpi_{p+2},-\varpi_{p}),(\omega_{2},0),$\\
&&&&$(\omega_{3}+\omega_{p+1},-\varpi_{p}),(\omega_{4},0),(\omega_{p},-\varpi_{p})$\\\hline

%p\ge2,q=1

\ref{sym7cS} & $\Sp(2p+2)$ & $\Sp(2p)\times \Sp(2)$ & $\{\beta_{2}\}$ & $(\omega_{1}+\omega_{p+1},-\varpi_{p}),(\omega_{2},0),(\omega_{p},-\varpi_{p})$\\
 &$p\ge2$&&&\\\hline

% p\ge3,q=1

\ref{sym7eS} & $\Sp(2p+2)$ & $\Sp(2p)\times \Sp(2)$ & $\{\beta_{i}\}$ & $(\omega_{1}+\omega_{i+1},-\varpi_{i}),(\omega_{2},0),$\\
 &$p\ge3$&&$2\leq i\leq p-1$& $(\omega_{i},-\varpi_{i}),(\omega_{i+2},-\varpi_{i})$\\\hline

% p=2,q\ge2
\ref{sym7fS} & $\Sp(4+2q)$ & $\Sp(4)\times \Sp(2q)$ & $\{\beta_{1},\beta_{2}\}$ & $(\omega_{1}+\omega_{3},-\omega_{2}),(\omega_{1},-\varpi_{1}),(\omega_{2},-\varpi_{2})$\\
 &$q\ge2$&&&$(\omega_{2},0),(\omega_{3},-\varpi_{1}),(\omega_{4},0)$\\\hline

%1

\ref{sym7hS} & $\Sp(2p+2)$ & $\Sp(2p)\times \Sp(2)$ & $\{\beta_{i},\beta_{j}\}$ & $(\omega_{1}+\omega_{i+1},-\varpi_{i}),(\omega_{1}+\omega_{j+1},-\omega_{j}),(\omega_{2},0),$ \\
 & $p\ge2$ && $1<i<j-1<p-1$& $(\omega_{i},-\varpi_{i}),(\omega_{i+1}+\omega_{j+1},-\varpi_{i}-\varpi_{j}),$\\
&&&& $(\omega_{i+2},-\varpi_{i}),(\omega_{j},-\varpi_{j}),(\omega_{j+2},-\varpi_{j})$\\\hline

%2

\ref{sym7hS} & $\Sp(2p+2)$ & $\Sp(2p)\times \Sp(2)$ & $\{\beta_{i},\beta_{j}\}$ & $(\omega_{1}+\omega_{j+1},-\varpi_{1}),(\omega_{1},-\varpi_{1}),(\omega_{2},0),$\\
 &$p\ge2$&& $1=i<j-1<p-1$ & $(\omega_{2}+\omega_{j+1},-\varpi_{1}-\varpi_{j}),(\omega_{3},-\varpi_{1}),$\\
&&&&$(\omega_{j},-\varpi_{j}),(\omega_{j+2},-\varpi_{j})$\\\hline

%3

\ref{sym7hS} & $\Sp(2p+2)$ & $\Sp(2p)\times \Sp(2)$ & $\{\beta_{i},\beta_{j}\}$ & $(\omega_{1}+\omega_{j+1},-\omega_{j}),(\omega_{1}+\omega_{i},-\omega_{j}),(\omega_{2},0),$\\
 &$p\ge2$&& $1<i=j-1<p-1$ & $(\omega_{i},-\varpi_{i}),(\omega_{j},-\varpi_{j}),$ \\
 &&&&$(\omega_{j+1},-\varpi_{i}),(\omega_{j+2},-\varpi_{j})$\\\hline

%4

\ref{sym7hS} & $\Sp(2p+2)$ & $\Sp(2p)\times \Sp(2)$& $\{\beta_{i},\beta_{j}\}$ & $(\omega_{1}+\omega_{i+1},-\omega_{i}),(\omega_{1}+\omega_{p+1},-\omega_{p}),(\omega_{2},0),$\\
 &$p\ge2$&& $1<i<j-1=p-1$ & $((\omega_{i},-\omega_{i})),(\omega_{i+1}+\omega_{p+1},-\omega_{i}-\varpi_{p}),$ \\
&&&& $(\omega_{i+2},-\omega_{i}),(\omega_{p},-\omega_{p})$\\\hline

%5

\ref{sym7hS} & $\Sp(2p+2)$ & $\Sp(2p)\times \Sp(2)$ & $\{\beta_{i},\beta_{j}\}$& $(\omega_{1}+\omega_{3},-\varpi_{2}),(\omega_{1},-\varpi_{1}),(\omega_{2},-0),$\\
 &$p\ge2$&& $1=i=j-1<p-1$ & $(\omega_{2},-\varpi_{2}),(\omega_{3},-\varpi_{1}),(\omega_{4},-\varpi_{2})$ \\\hline

%6

\ref{sym7hS} & $\Sp(2p+2)$ & $\Sp(2p)\times \Sp(2)$ & $\{\beta_{i},\beta_{j}\}$& $(\omega_{1}+\omega_{p},-\varpi_{p-1}),(\omega_{1}+\omega_{p+1},-\varpi_{p}),$\\
 &$p\ge2$&& $1<i=j-1=p-1$ & $(\omega_{2},0),(\omega_{p-1},-\varpi_{p-1}),$\\
&&&&$(\omega_{p},-\varpi_{p}),(\omega_{p+1},-\varpi_{p-1})$\\\hline

%7

\ref{sym7hS} & $\Sp(2p+2)$ & $\Sp(2p)\times \Sp(2)$ & $\{\beta_{i},\beta_{j}\}$& $(\omega_{1}+\omega_{3},-\varpi_{2}),(\omega_{1},-\varpi_{1}),(\omega_{2},-\varpi_{2}),$\\
 &$p\ge2$&& $1=i=j-1=p-1$ & 
$(\omega_{2},0),(\omega_{3},-\varpi_{1})$\\\hline

%1

\ref{sym7iS} & $\Sp(2+2q)$ & $\Sp(2)\times \Sp(2q)$ & $\{\beta_{1},\beta_{i}'\}$ & $(\omega_{1}+\omega_{i+1},-\varpi_{i}'),(\omega_{1},-\varpi_{1}),(\omega_{2},0),$\\
 &$q\ge1$&& $1<i<q$ & $(\omega_{i},-\varpi_{i}'),(\omega_{i+1},-\varpi_{i}'-\omega_{1}),(\omega_{i+2},-\varpi_{i}')$\\\hline

%2

\ref{sym7iS} & $\Sp(2+2q)$ & $\Sp(2)\times \Sp(2q)$ & $\{\beta_{1},\beta_{i}'\}$ & $(\omega_{1},-\varpi_{1}),(\omega_{1},-\varpi_{1}'),(\omega_{2},0),$\\
 &$q\ge1$&& $1=i<q$ & $(\omega_{2},-\varpi_{1}'-\varpi_{1}),(\omega_{3},-\varpi_{1}')$\\\hline

%3

\ref{sym7iS} & $\Sp(2+2q)$ & $\Sp(2)\times \Sp(2q)$ & $\{\beta_{1},\beta_{i}'\}$ & $(\omega_{1}+\omega_{p+1},-\varpi_{q}'),(\omega_{1},-\varpi_{1}),(\omega_{2},0),$\\
 &$q\ge1$&& $1<i=q$ & $(\omega_{p},-\varpi_{q}'),(\omega_{p+1},-\varpi_{q}'-\varpi_{1})$ \\\hline

%4

\ref{sym7iS} & $\Sp(2+2q)$ & $\Sp(2)\times \Sp(2q)$ & $\{\beta_{1},\beta_{i}'\}$ & $(\omega_{1},-\varpi_{1}),(\omega_{1},-\varpi_{1}'),$\\
 &$q\ge1$&& $1=i=q$ & $(\omega_{2},0),(\omega_{2},-\varpi_{1}'-\varpi_{1})$\\\hline\hline

\ref{sym8aS}&$\mathsf{F}_{4}$& $\Spin(9)$   & $\{\beta_{1},\beta_{2}\}$ & $(\omega_{1},-\varpi_{2}),(\omega_{2},-\varpi_{2}),(\omega_{3},-\varpi_{2})$\\
&&&&$(\omega_{3},-\varpi_{1}),(\omega_{4},-\varpi_{1}),(\omega_{4},0)$\\ \hline 

\ref{sym8bS}&$\mathsf{F}_{4}$& $\Spin(9)$   & $\{\beta_{3}\}$ & $(\omega_{1}+\omega_{4},-\varpi_{3}),(\omega_{1}+\omega_{3},-\varpi_{3}),(\omega_{2},-\varpi_{3})$\\
&&&&$(\omega_{3},-\varpi_{3}),(\omega_{4},0)$\\ \hline 

\ref{sym8cS}&$\mathsf{F}_{4}$& $\Spin(9)$   & $\{\beta_{4}\}$ & $(\omega_{1},-\varpi_{4}),(\omega_{3},-\varpi_{4}),$\\
&&&&$(\omega_{4},-\varpi_{4}),(\omega_{4},0)$\\ \hline \hline

\ref{sym8S}&$\mathsf{E}_{6}$& $\SO(10)\times\bbC^{\times}$ & $\{\beta_{1}\}$  & $(\omega_{1},-\varpi_{1}+2\epsilon),(\omega_{1},-4\epsilon),(\omega_{2},0),$\\
&&&&$(\varpi_{6},-\varpi_{1}-2\epsilon),(\omega_{6},4\epsilon),$\\
&&&&$(\omega_{5},-\varpi_{1}+2\epsilon),(\omega_{3},-\varpi_{1}-2\epsilon),$\\ \hline \hline

\ref{sym10S}&$\mathsf{E}_{6}$&$\mathsf{F}_{4}$&$\{\beta_{1}\}$&$(\omega_{1},0),(\omega_{2},-\varpi_{1}),(\omega_{3},-\varpi_{1})$\\
&&&&$(\omega_{4},-\varpi_{1}),(\omega_{5},-\varpi_{1}),(\omega_{6},0)$\\ \hline\hline

\ref{sym11S}&$\SL(n)\times\SL(n)$ & $\diag(\SL(n))$ & $\{\beta_{1}\}$  & $(\omega_{i}+\omega_{n-i}',0), 1\leq i\leq n-1$,\\
&&&& $(\omega_{i}+\omega_{n+1-i}',-\varpi_{1}), 1\leq i\leq n$.\\ \hline

\ref{sym11S}&$\SL(n)\times\SL(n)$ & $\diag(\SL(n))$ & $\{\beta_{n-1}\}$  & $(\omega_{i}+\omega_{n-i}',0), 1\leq i\leq n-1$,\\
&&&& $(\omega_{n+1-i}+\omega_{i}',-\varpi_{n-1}), 1\leq i\leq n$.\\\hline\hline

% 1

\ref{sph1S}&$\SL(p+q+2)$ & $\SL(p+1)\times\SL(q+1)$ & $\{\beta_{1}\}$ & $(\omega_{i}+\omega_{n+1-i},0), 1\leq i\leq p,$ \\
&$2\le p< q$&&& $(\omega_{j}+\omega_{n+2-j},-\varpi_{1}),2\leq j\leq p,$\\
&$n=p+q+1$&&& $(\omega_{q+2},-\varpi_{1}),(\omega_{1},-\varpi_{1}),$\\
&&&& $(\omega_{q+1},0),(\omega_{p+1},0)$\\ \hline
%Here $\varpi_{p+1}=0$

%1
\ref{sph1S}&$\SL(p+q+2)$ & $\SL(p+1)\times\SL(q+1)$ &  $\{\beta_{p}\}$ & $(\omega_{i}+\omega_{n+1-i},0), 1\leq i\leq p,$ \\
&$2\le p<q$&&& $(\omega_{j}+\omega_{n-j},-\varpi_{p}),1\leq j\leq p-1,$\\
&$n=p+q+1$&&& $(\omega_{p},-\varpi_{p}),(\omega_{n},-\varpi_{p}),$\\
&&&& $(\omega_{p+1},0),(\omega_{q+1},0)$\\ \hline 
%Here $\varpi_{p+1}=0$

%2
\ref{sph1S}&$\SL(q+3)$ & $\SL(2)\times\SL(q+1)$ & $\{\beta_{1}\}$ & $(\omega_{1},0),(\omega_{1},-\varpi_{1}),(\omega_{2},0),$ \\
&$q>1$&&&$(\omega_{q+1},0),(\omega_{q+2},-\varpi_{1})$\\
%&$n=p+q+1$&&&\\
\hline
%Here $\varpi_{2}=0$.

%3

\ref{sph1S}&$\SL(q+3)$ & $\SL(2)\times\SL(q+1)$ & $\{\beta_{i}'\}$ & $(\omega_{1}+\omega_{i},-\varpi_{i+1}),(\omega_{1}+\omega_{n},0),$ \\
&$q\ge4$&&$2\leq i\leq q-2$& $(\omega_{2},0),(\omega_{i-1},-\varpi_{i+1}),$\\
&$n=q+2$&&& $(\omega_{i}+\omega_{n},-\varpi_{i+n}),$\\
&&&& $(\omega_{i+1},-\varpi_{i+1}),(\omega_{n-1},0)$\\ \hline
%Here $\varpi_{2}=0$.

%4
\ref{sph1S}&$\SL(q+2)$ & $\SL(q+1)$ & $S_{H}\smallsetminus\{\beta_{k}\}$ & $(\omega_{i},-\varpi_{i}), 2\leq i\leq n,$ but $i\ne k$ \\
&$q\ge1$&&$1\leq k\leq q$& $(\omega_{j},-\varpi_{j+1}),1\leq k\leq n-1,$ but $j\ne k$,\\
&$n=q+1$&&& $(\omega_{1},0),(\omega_{n},0)$\\
\hline
\hline
%Here $\varpi_{1}=0$.

\ref{sph2S}&$\SO(4n+2)$ & $\SL(2n+1)$ & $\{\beta_{1}\}$ & $(\omega_{2i-1},-\varpi_{1}),1\leq i\leq n,$\\ 
&$n\ge 2$&&& $(\omega_{2j},0),1\leq i\leq n-1,$\\
&&&& $(\omega_{2n},0),(\omega_{2n+1},0),(\omega_{2n+1},-\varpi_{1})$\\\hline

\ref{sph2S}&$\SO(4n+2)$ & $\SL(2n+1)$ & $\{\beta_{2n}\}$ & $(\omega_{2i-1},-\varpi_{2n}),1\leq i\leq n,$\\ 
&$n\ge 2$&&& $(\omega_{2j},0),1\leq i\leq n-1,$\\
&&&& $(\omega_{2n},0),(\omega_{2n+1},0),(\omega_{2n+1},-\varpi_{2n})$\\ \hline\hline

\ref{sph3S}&$\Spin(9)$&$\Spin(7)$&$\{\beta_{1}\}$&$(\omega_{1},0),(\omega_{2},-\varpi_{1}),(\omega_{3},-\varpi_{1})$\\ 
&&&&$(\omega_{4},0),(\omega_{4},-\varpi_{1})$\\ \hline \hline

\ref{sph4aS}&$\Spin(7)$&$\mathsf{G}_{2}$&$\{\beta_{1}\}$&$(\omega_{1},-\varpi_{1}),(\omega_{2},-\varpi_{1})$\\ 
&&&&$(\omega_{3},0),(\omega_{3},-\varpi_{1})$\\ \hline
\ref{sph4bS}&$\Spin(7)$&$\mathsf{G}_{2}$&$\{\beta_{2}\}$&$(\omega_{1}+\omega_{2},-\varpi_{2}),(\omega_{2},-\varpi_{2})$\\ 
&&&&$(\omega_{1}+\omega_{3},-\varpi_{2}),(\omega_{3},0)$\\ \hline \hline

\ref{sph5S}&$\mathsf{G}_{2}$&$\SL(3)$&$\{\beta_{1}\}$&$(\omega_{1},0),(\omega_{1},-\varpi_{1}),(\omega_{2},-\varpi_{1})$\\ 

\hline
\ref{sph5S}&$\mathsf{G}_{2}$&$\SL(3)$&$\{\beta_{2}\}$&$(\omega_{1},0),(\omega_{1},-\varpi_{2}),(\omega_{2},-\varpi_{2})$\\  
\hline

\end{longtable}
\end{center}
\end{tiny}

\begin{tiny}
\begin{center}
\begin{longtable}{|c|c|c|c|c|}
\caption{$(G,H)$ spherical, not symmetric, $G$ not simple.}\\
\hline
 & $G$ & $H$ & $P$ & Generators $\widetilde{\Gamma}(G/P)$\\
\hline
\endfirsthead
\multicolumn{5}{c}%
{\tablename\ \thetable\ -- \textit{Continued from previous page}} \\
\hline
 & $G$ & $H$ & $P$ & Generators $\widetilde{\Gamma}(G/P)$\\\hline
\endhead
\hline \multicolumn{5}{r}{\textit{Continued on next page}} \\
\endfoot
\hline
\endlastfoot

%1
\ref{sph6aS}&$\Sp(2m)\times \Sp(2n)$&$\Sp(2m-2)\times \SL(2)\times \Sp(2n-2)$&$\{\beta_{m-1}\}$&$(\omega_{1}+\omega_{m},-\varpi_{m-1}),(\omega_{1}+\omega_{1}',0),$\\ 
&$n>1,m>2$&&&$(\omega_{2},0),(\omega_{m-1},-\varpi_{m-1}),$\\
&&&&$(\omega_{m}+\omega_{1}',-\varpi_{m-1}),(\omega_{2}',0)$\\ \hline
%2
\ref{sph6aS}&$\Sp(2m)\times \Sp(2)$&$\Sp(2m-2)\times \SL(2)$&$\{\beta_{m-1}\}$&$(\omega_{1}+\omega_{m},-\varpi_{m-1}),(\omega_{1}+\omega_{1}',0),$\\ 
&$m>2$&&&$(\omega_{2},0),(\omega_{m-1},-\varpi_{m-1}),$\\
&&&&$(\omega_{m}+\omega_{1}',-\varpi_{m-1})$\\ \hline
%3
\ref{sph6aS}&$\Sp(4)\times \Sp(2n)$&$\Sp(2)\times \SL(2)\times \Sp(2n-2)$&$\{\beta_{1}\}$&$(\omega_{1},-\varpi_{1}),(\omega_{1}+\omega_{1}',0),$\\ 
&$n\ge2$&&&$(\omega_{2},0),$\\
&&&&$(\omega_{2}+\omega_{1}',-\varpi_{1}),(\omega_{2}',0)$\\ \hline
%4
\ref{sph6aS}&$\Sp(4)\times \Sp(2)$&$\Sp(2)\times \SL(2)$&$\{\beta_{1}\}$&$(\omega_{1},-\varpi_{1}),(\omega_{1}+\omega_{1}',0),$\\ 
&&&&$(\omega_{2},0),(\omega_{2}+\omega_{1}',-\varpi_{1}),$\\ \hline
%5
\ref{sph6aS}&$\Sp(2m)\times \Sp(2n)$&$\Sp(2m-2)\times \SL(2)\times \Sp(2n-2)$&$\{\beta_{i}\}$&$(\omega_{1}+\omega_{i+1},-\varpi_{i}),(\omega_{1}+\omega_{1}',0),$\\ 
&$1<i<m-1$&&&$(\omega_{2},0),(\omega_{i},-\varpi_{i}),(\omega_{2}',0),$\\
&$n>1$&&&$(\omega_{i+1}+\omega_{1}',-\varpi_{i}),(\omega_{i+2},-\varpi_{i})$\\ \hline
%6
\ref{sph6aS}&$\Sp(2m)\times \Sp(2)$&$\Sp(2m-2)\times \SL(2)$&$\{\beta_{i}\}$&$(\omega_{1}+\omega_{i+1},-\varpi_{i}),(\omega_{1}+\omega_{1}',0),$\\ 
&$1<i<m-1$&&&$(\omega_{2},0),(\omega_{i},-\varpi_{i}),$\\
&&&&$(\omega_{i+1}+\omega_{1}',-\varpi_{i}),(\omega_{i+2},-\varpi_{i})$\\ \hline
%7
\ref{sph6aS}&$\Sp(2m)\times \Sp(2n)$&$\Sp(2m-2)\times \SL(2)\times \Sp(2n-2)$&$\{\beta_{1}\}$&$(\omega_{1}+\omega_{1}',0),(\omega_{1},-\varpi_{1})$\\ 
&$m\ge3,n>1$&&&$(\omega_{2},0),(\omega_{2}+\omega_{1}',-\varpi_{1}),$\\
&&&&$(\omega_{3},-\varpi_{1}),(\omega_{2}',0)$\\ \hline
%8
\ref{sph6aS}&$\Sp(2m)\times \Sp(2)$&$\Sp(2m-2)\times \SL(2)$&$\{\beta_{1}\}$&$(\omega_{1}+\omega_{1}',0),(\omega_{1},-\varpi_{1})$\\ 
&$m\ge3$&&&$(\omega_{2},0),(\omega_{2}+\omega_{1}',-\varpi_{1}),$\\
&&&&$(\omega_{3},-\varpi_{1})$\\ \hline\hline

\ref{sph6bS}&$\Sp(2m)\times \Sp(2n)$&$\Sp(2m-2)\times \SL(2)\times \Sp(2n-2)$&$\{\beta_{1}'\}$&$(\omega_{1}+\omega_{1}',0),$\\ 
&$n>1,m>1$&&&$(\omega_{1},-\varpi_{1}'),(\omega_{1}',-\varpi_{1}'),$\\
&&&&$(\omega_{2},0),(\omega_{2}',0)$\\ \hline

\ref{sph6bS}&$\Sp(2m)\times \Sp(2n)$&$\Sp(2m-2)\times \SL(2)\times \Sp(2n-2)$&$\{\beta_{1}'\}$&$(\omega_{1}+\omega_{1}',0),$\\ 
&$n>1,m>1$&&&$(\omega_{1},-\varpi_{1}'),(\omega_{1}',-\varpi_{1}'),$\\
&&&&$(\omega_{2},0),(\omega_{2}',0)$\\ \hline

\ref{sph6bS}&$\Sp(2m)\times \Sp(2)$&$\Sp(2m-2)\times \SL(2)$&$\{\beta_{1}'\}$&$(\omega_{1}+\omega_{1}',0),$\\ 
&$m>1=n$&&&$(\omega_{1},-\varpi_{1}'),(\omega_{1}',-\varpi_{1}'),$\\
&&&&$(\omega_{2},0)$\\ \hline

\end{longtable}
\end{center}

\end{tiny}

\begin{remark}
We take the opportunity to fix a mistake in \cite{MR3801483}. In Remark 4.4(4) of loc.cit.~it is claimed that for the case $\SL(n)/\mathrm{S}(\mathrm{L}(n-2)\times\mathrm{L}(2))$, an irreducible representation of $\SL(n)$ cannot contain two irreducible $\mathrm{S}(\mathrm{L}(n-2)\times\mathrm{L}(2))$-modules of highest weight $p\omega_{n-1}$ and $q\omega_{n-1}$ with $p\ne q$, upon restriction. This is not true. Indeed, the extended weight semigroup $\widetilde{\Gamma}(\SL(n)/P)$, where $P\subset \mathrm{S}(\mathrm{L}(n-2)\times\mathrm{L}(2))$ is the parabolic subgroup obtained by leaving out the last simple root, has $((\omega_{n-1},0)$ and $(\omega_{n-1},-\varpi_{n-1}))$ among its generators, see \ref{sym2-6S}. Hence, if $\lambda$ is the highest weight of an irreducible $\SL(n)$-representation that contains $p\varpi_{n-1}$, then $\lambda+\omega_{n-1}$ is the highest weight of an irreducible $\SL(n)$-representation that contains irreducible representations of $\mathrm{S}(\mathrm{L}(n-2)\times\mathrm{L}(2))$ of highest weight $p\varpi_{n-1}$ and $(p+1)\varpi_{n-1}$.

The induction arguments of loc.cit.~is still valid but it has to be applied more carefully. The result is that the subgroup $\GL(n-2)\times \SL(2)\times \Sp(2m-2)\subset\SL(n)\times \Sp(2m)$ does not admit a proper parabolic subgroup that remains spherical in $\SL(n)\times \Sp(2m)$. This is also in line with Remark \ref{remark: two colors}.
\end{remark}

%%%%%%%%%%%%%%%%%%%%%%%%%%%%%%%%
%%%%%                                            %%%%%%%%%%%%%%%
%%%%%          NEW SECTION          %%%%%%%%%%%%%%%
%%%%%                                           %%%%%%%%%%%%%%%
%%%%%%%%%%%%%%%%%%%%%%%%%%%%%%%%  

\appendix

\section{Preliminaries on the cases}\label{s:preliminaries}

In the next two sections we consider spherical homogeneous spaces of the form $G/H$, where $G$ is semisimple simply connected and $H$ is a connected reductive proper subgroup, such that $H$ has a proper parabolic subgroup $P$ that is spherical in $G$. In other words $(G,H,P)$ is a multiplicity free system. For such triples, we compute the extended weight monoid of $G/P$. The choice of $H$ in its conjugacy class is often relevant for our computations and will be specified in each case.

We use many of the notations of Section~\ref{s:example}, let us recall them here and add some further ones. We denote by $\alpha_1,\alpha_2,\ldots$ the simple roots of $G$, numbered as in Bourbaki. The corresponding fundamental dominant weights will be denoted by $\omega_1,\omega_2,\ldots$. For $i\leq j$ we set $\alpha_{i,j} = \alpha_i+\ldots+\alpha_j$, and if $i>j$ then we set $\alpha_{i,j}=0$. For convenience in some formulae, if $G=\SL(n)$ we also set $\omega_n=0$.

If $G\subseteq \GL(n)$ is a classical group, it will be defined in such a way that $B$ (resp.\ $T$) can be taken to be the set of upper triangular (resp.\ diagonal) matrices in $G$.

If $G$ is the universal cover of a classical group, to simplify notations, we will implicitly replace $G$ with the classical group $G_0$ and the subgroups $H$, $P$ with their images $H_0$, $P_0$ in $G_0$. Thanks to Proposition~\ref{prop:G0}, it will be enough to notice that no element of $2S$ appears among the spherical roots of $G_0/P_0$, to assure that our computations carried out for $G_0$ are equivalent to those for $G$.

Unless otherwise stated, we denote by $S_H=\{\beta_1,\beta_2,\ldots,\beta_1',\beta_2',\ldots\}$ the simple roots of $H$, grouped according to the various simple normal subgroups of $H$ and numbered as in Bourbaki, corresponding as in the introduction to a choice of a Borel subgroup $B_H$ and a maximal torus $T_H\subseteq B_H$ of $H$. We recall that we take $B_H\subseteq B$ and $T_H\subseteq T$, and $P$ containing the opposite Borel subgroup $B_H^-$ of $B_H$ with respect to $T_H$.

We denote by $\varpi_1,\varpi_2,\ldots,\varpi_1',\varpi_2',\ldots$ the corresponding fundamental dominant weights, which we define as the elements of $\Chi(T_H)_\QQ$ that have the correct pairing with the simple coroots of $H$ and are zero on the subspace corresponding to the connected center of $H$. We call them the {\em fundamental weights} of $H$.

We will often restrict characters of groups to subgrups, or extend them when possible to characters of larger groups. To simplify notations, we will denote with the same symbol the original character and the restriction or extension, if no confusion arises.

We denote by $I$ the set of simple roots of $H$ such that the Levi subgroup of $P$ containing $T_H$ has set of simple roots $S_H\smallsetminus I$.

It will also be possible to choose a parabolic subgroup $Q$ of $G$, containing the Borel subgroup $B^-$ opposite to $B$ with respect to $T$, and minimal among the parabolic subgroups of $G$ containing a $G$-conjugate of $P$. This will enable us to apply the results of Section~\ref{s:howto} to the subgroup $P$. As before, we will denote by $L_Q$ the Levi subgroup of $Q$ containing $T$. Notice that $P$ is connected, therefore Theorem~\ref{thm:howto} yields the characters $\chi_D$ for all $D\in\Delta(G/P)$.

If $J\subseteq S$, then we will also use the notation $Q_J$ instead of simply $Q$, where $Q_J$ is the parabolic subgroup of $G$ containing $B_-$ and such that its Levi subgroup has simple roots $S\smallsetminus J$.

If a simple root $\alpha_i$ moves only one color, then the latter will be denoted by $D_i$ or $E_i$. If $\alpha_i$ moves two colors, they will be denoted by $D_i^+$ and $D_i^-$, or by $D_i$ and $E_i$.

If it doesn't create ambiguities, we will allow the abuse of notation of denoting in the same way elements of the Weyl groups of reductive groups and a choice of representatives in the normalizer of the chosen maximal torus.

Finally we record the following observations.

\begin{remark}\label{remark: duals}
Let $(G,H,P)$ be a strictly indecomposable multiplicity free system and let $P^{\mathrm{op}}$ be opposite to $P$ with respect to the maximal torus $T_{H}\subseteq H$. Then $(G,H,P^{\mathrm{op}})$ is also a strictly indecomposable multiplicity free system and the generators of $\widetilde{\Gamma}(G/P)$ are related to those of $\widetilde{\Gamma}(G/P^{\mathrm{op}})$ by $(\omega_{D},\chi_{D})\leftrightarrow(\omega_{D}^{*},\chi_{D}^{*})$ where 
$\omega_{D}^{*}$ is the highest weight of the dual of $\pi^{G}_{\omega_{D}}$ and
$\chi_{D}^{*}$ is the lowest weight of the dual representation of $\pi^{H}_{\chi_{D}}$.
\end{remark}

\begin{remark}\label{remark:subgroups}
Let $(G,H,P_1)$ and $(G,H,P_2)$ be multiplicity free systems with $P_1\subseteq P_2$. Then $\widetilde{\Gamma}(G/P_2)$ is the subset of $\widetilde{\Gamma}(G/P_1)$ of the couples $(\lambda,\omega)$ such that $\omega$ extends to a weight of $P_2$. For this reason, during the computations in Appendices~\ref{s:sym} and~\ref{s:nonsym}, sometimes we will only discuss those parabolic subgroups $P\subseteq H$ that are minimal such that $(G,H,P)$ is a multiplicity free system. We cannot just deal everywhere only with minimal cases though, because some computations for non-minimal subgroups will be necessary to complete computations for the minimal ones.
\end{remark}

\begin{remark}\label{remark:H'}
Let $(G,H,P)$ be a multiplicity free system, and let $\widetilde H\subseteq H$ be a connected reductive subgroup containing the commutator $(H,H)$. In this case it is harmless to assume that $H$ is the product of $\widetilde H$ and a subtorus $\widetilde T$ of $T$. The root systems of $H$ and $\widetilde K$ are naturally identified, so that $P$ corresponds to a parabolic subgroup $\widetilde P$ of $\widetilde H$ contained in $P$ and such that $\widetilde P \cdot \widetilde T=P$. Suppose now that $(G,\widetilde H,\widetilde P)$ is a multiplicity free system, and notice that then the natural map $G/\widetilde P\to G/P$ induces a bijection between the respective sets of colors, since $\widetilde T\subseteq B$. Proposition~\ref{prop:generators} implies that restriction of weights from $P$ to $\widetilde P$ induces an isomorphism $\widetilde\Gamma(G/P)\to\widetilde\Gamma(G/\widetilde P)$.
\end{remark}

%%%%%%%%%%%%%%%%%%%%%%%%%%%%%%%%
%%%%%                                            %%%%%%%%%%%%%%%
%%%%%          NEW SECTION          %%%%%%%%%%%%%%%
%%%%%                                           %%%%%%%%%%%%%%%
%%%%%%%%%%%%%%%%%%%%%%%%%%%%%%%%  

\section{Symmetric cases}\label{s:sym}

%    9.1 

\subsection{$\SL(2n)/\Sp(2n)$ with $n\ge2$}\label{sym1} 
We define $\Sp(2n)$ to be the stabilizer of the skew-symmetric bilinear form given by the matrix
\[
\left(
\begin{array}{ccccc}
0 & 0 & \ldots & 0 & 1 \\
0 & 0 & \ldots & 1 & 0 \\
\vdots & \vdots & \ddots & \vdots & \vdots \\
0 & -1 & \ldots & 0 & 0 \\
-1 & 0 & \ldots & 0 & 0 \\
\end{array}
\right)
\]
In this way $B_{H}=B\cap\Sp(2n)$ is a Borel subgroup of $\Sp(2n)$ and $T_{H}=T\cap\Sp(2n)\subseteq B_{H}$ is a maximal torus of $\Sp(2n)$ contained in $B_{H}$. The simple roots of $\Sp(2n)$ are given by $\beta_{i}(t)=t_{i}t_{i+1}^{-1}$ with $1\leq i\leq n-1$ and $\beta_{n}(t)=t_{n}^{2}$, where $t=(t_{1},\ldots,t_{n},t_{n}^{-1},\ldots,t_{1}^{-1})\in T_{H}$.

The group $H$ is simply connected, and the fundamental dominant weights $\varpi_1,\ldots,\varpi_n$ have the property that $\varpi_i$ is the restriction of $\omega_i$ and also of $\omega_{2n-i}$ to $T_H$, for all $i\in\{1,\ldots,n\}$.

We have three possibilities for $I$:
\begin{enumerate}
\item\label{sym1a} $I=\{\beta_1\}$ with any $n\geq3$,
\item\label{sym1b} $I=\{\beta_3\}$ with $n=3$,
\item\label{sym1c} $I=$ any subset of roots of $H$ with $n=2$.
\end{enumerate}

\subsubsection{$I=\{\beta_1\}$ with any $n\geq3$}\label{sym1aS}
The subgroup $P$ appears in~\cite{MR2183057} as case 6. The parameter ``$n$'' appearing in loc.cit.\ is equal here to $2n-1$. This gives
\[
\Sigma(G/P)=\{\alpha_1+\alpha_2,\ldots,\alpha_{2n-2}+ \alpha_{2n-1}\}
\]
and $\Delta(G/P)=\{D_{1},\ldots,D_{2n-1}\}$, so that $\alpha_i$ moves $D_i$ for all $i\in\{1,\ldots,2n-1\}$, and the Cartan pairing is given by $\rho(D_i)=\alpha_i^\vee|_{\Xi(G/P)}$.

We can take $Q$ to be the parabolic subgroup such that $L_Q$ has simple roots $\alpha_2,\ldots,\alpha_{2n-2}$. Then $Q$ is minimal for containing $P$ and $G/Q$ has two colors whose inverse images in $G/P$ are $D_1$ and $D_{2n-1}$. Since $\omega_1^P = \omega_{2n-1}^P=\varpi_{1}$ we have
\[
\chi_{D_1} = \chi_{D_{2n-1}} = -\varpi_{1}.
\]

Using the Cartan pairing one deduces from equalities~(\ref{eq:sphroot}) of Lemma~\ref{lemma:sphroot} the following system of equations:
\[
\left\{
\begin{array}{rcl}
0 & = & -\varpi_{1} + \chi_{D_2} - \chi_{D_3}\\
0 & = & \varpi_{1} + \chi_{D_2} + \chi_{D_3} - \chi_{D_4}\\
0 & = & - \chi_{D_2} + \chi_{D_3} + \chi_{D_4} - \chi_{D_5}\\
 & \vdots & \\
0 & = & -\chi_{D_{2n-5}} + \chi_{D_{2n-4}} + \chi_{D_{2n-3}} - \chi_{D_{2n-2}}\\
0 & = & -\chi_{D_{2n-4}} + \chi_{D_{2n-3}} + \chi_{D_{2n-2}} + \varpi_{1}\\
0 & = & - \chi_{D_{2n-3}} + \chi_{D_{2n-2}} - \varpi_{1}\\
\end{array}
\right.
\]
from which we obtain
\[
\begin{array}{lclcl}
\chi_{D_2} & = \ldots = & \chi_{D_{2n-2}} &=& 0,\\
\chi_{D_1} & = \ldots = & \chi_{D_{2n-1}} &=& -\varpi_{1}.
\end{array}
\]

\subsubsection{$I=\{\beta_3\}$ with $n=3$}\label{sym1bS}
The subgroup $P$ appears in~\cite{MR2183057} as case 27, giving
\[
\Sigma(G/P) = S,
\]
and $\Delta(G/P)=\{D_1^+=D_3^+=D_5^+, D_1^-=D_4^-, D_2^+, D_2^-=D_5^-, D_3^-, D_4^+ \}$. We can take $Q$ such that only $\alpha_3$ is not a simple root of $L_Q$. Then $Q$ is minimal for containing $P$. The inverse image in $G/P$ of the unique color of $G/Q$ is the unique color of $G/P$ moved only by $\alpha_3$, i.e.\ $D_3^-$. Then
\[
\chi_{D_3^-} = -\omega_3^P=-\varpi_{3}.
\]

The Cartan pairing of $G/P$ is
\[
\begin{array}{c|ccccc}
      & \alpha_1 & \alpha_2 & \alpha_3 & \alpha_4 & \alpha_5 \\
\hline
D_1^+ & 1 & -1 & 1 & -1 & 1 \\
D_1^- & 1 & 0 & -1 & 1 & -1 \\
D_2^+ & 0 & 1 & 0 & 0 & -1 \\
D_2^- & -1 & 1 & -1 & 0 & 1 \\
D_3^- & -1 & 0 & 1 & 0 & -1 \\
D_4^+ & -1 & 0 & 0 & 1 & 0
\end{array}
\]
which yields, thanks to the equalities~(\ref{eq:sphroot}) of Lemma~\ref{lemma:sphroot},
\[
\begin{array}{lclclcl}
\chi_{D_1^+} & = & \chi_{D_1^-} & = & \chi_{D_2^-} & = & -\varpi_{3},\\
& & \chi_{D_2^+} & = & \chi_{D_4^+} & = & 0.
\end{array}
\]

\subsubsection{$I=$ any subset of roots of $H$, $n=2$}\label{sym1cS}
We discuss only the minimal case $I=\{\beta_{1},\beta_{2}\}$; the subgroup $P$ corresponding to $I=\{\beta_1\}$ is found in~\cite{MR1424449} as the first case 5 in Table A, and the one corresponding to $I=\{\beta_2\}$ is found in~\cite{MR1424449} as the second case 7 in Table A.

We claim that the subgroup $P=B_H$ corresponding to $I=S_{H}$ has lattice $\Xi(G/B_H)=\Chi(T_G)$, the following spherical roots
\[
\Sigma(G/B_H) = S
\]
and colors $\Delta(G/B_H)=\{D_1^+=D_3^+, D_1, D_2^+, D_2^-, D_3^-\}$ with Cartan pairing
\[
\begin{array}{c|ccc}
      & \alpha_1 & \alpha_2 & \alpha_3 \\
\hline
D_1^+ & 1 & -1 & 1 \\
D_1^- & 1 & 0 & -1 \\
D_2^+ & 0 & 1 & 0 \\
D_2^- & -1 & 1 & -1 \\
D_3^- & -1 & 0 & 1 
\end{array}
\]

To prove this, we recall that $G/H$ has one spherical root $\frac12(\alpha_1+2\alpha_2+\alpha_3)$, which generates $\Xi(G/H)$, and one color $D_2'$ with valuation coinciding with $\alpha_2^\vee$ on the lattice of $G/H$ (see~\cite{MR1424449}, case 5A in Table 1). Using Corollary~\ref{cor:morphisms} and Proposition~\ref{prop:PinH} one checks that a subgroup $P$ corresponding to the above data is indeed conjugated in $G$ to a parabolic subgroup of $H$, so we may assume $P\supseteq B_H$. Formula~(\ref{eq:rankChar}) assures that its character group has rank $2$, hence $P=B_H$.

The subset of colors $\{D_1^+, D_2^+\}$ is the only one that corresponds to a $G$-equivariant morphism $G/B_H\to G/B_G$, therefore it corresponds to the inclusion $B_H\subseteq B_G$, and the colors $D_1^-, D_2^-, D_3^-$ are the inverse images of the three colors of $G/B_G$. We conclude
\[
\begin{array}{lclclcl}
\chi_{D_1^-} & = & -\omega_1^{B_H} & = &  -\varpi_{1},\\
\chi_{D_2^-} & = & -\omega_2^{B_H} & = &  -\varpi_{2},\\
\chi_{D_3^-} & = & -\omega_3^{B_H} & = &  -\varpi_{1},\\
\end{array}
\]

The equalities~(\ref{eq:sphroot}) of Lemma~\ref{lemma:sphroot} yield
\[
\begin{array}{lcl}
\chi_{D_1^+} & = & -\varpi_{2},\\
\chi_{D_2^+} & = & 0.
\end{array}
\]

%        9.2

\subsection{$\SL(p+q+2)/\mathrm S(\mathrm L(p+1)\times \mathrm L(q+1))$ with $p+q\ge1$ and $0\le p\le q$}\label{sym2} Set $n=p+q+1$.
We take the subgroup $H$ to be the matrices with blocks on the diagonal of sizes resp.\ $p+1$ and $q+1$, and zeros elsewhere. In this way $B_{H}=B\cap H$ is a Borel subgroup of $H$ with maximal torus $T_{H}=T$. The simple roots of $H$ are $\beta_{i}=\alpha_{i}$ for $1\leq i \leq p$ and $\beta'_j=\beta_{p+1+j}=\alpha_{p+1+j}$ for $1\leq j\leq q$. We set $\varpi_{i}=\omega_{i}$ for $1\leq i \leq n$, and $\varpi_{n+1}=0$. Then the $\varpi_{i}$'s are the fundamental weights of $H$, except for $\varpi_{n+1}$, and also except for $\varpi_{p+1}$, which is the restriction to $T$ of a generator of the character group of $H$. The dominant integral weights of $H$ are of the form $\sum_{k=1}^{n}b_{k}\omega_{k}$ with $b_{i}\in\ZZ_{\geq0}$ for all $i$ except for $b_{p+1}\in\bbZ$.

We have the following possibilities for $I$:

\begin{enumerate}
\item\label{sym2-1} $I=\{\beta_{1}\}$ with $p=1\le q$.
\item\label{sym2-2} $I=\{\beta_{1}\}$ with $2\leq p\le q$.
\item\label{sym2-3} $I=\{\beta_{p}\}$ with $2\leq p\le q$.
\item\label{sym2-4} $I=\{\beta'_{1}\}$ with $1\le p<q$.
\item\label{sym2-5} $I=\{\beta'_{q}\}$ with $1\le p<q$.
\item\label{sym2-6} $I=\{\beta'_{i}\}$ with $p=1,q\ge3$ and $i\in\{2,\ldots,q-1\}$.
\item\label{sym2-7} $I$= any subset of simple roots of $H$, with $p=0$.
\end{enumerate}

\subsubsection{$I=\{\beta_{1}\}$ with $p=1\le q$, i.e.\ $n=q+2$}\label{sym2-1S}
The subgroup $P$ is conjugated to the subgroup $\widetilde{P}$ that appears in \cite{MR2183057} as case 9; the parameter ``$p$'' of loc.cit.\ is equal here to $q$. We have $gPg^{-1}=\widetilde{P}$ with $g=w_{0}^{G}s_{1}$, where $w^{G}_{0}\in W_{G}$ is the longest element.

We have
\[
\Sigma(G/P)=\{\alpha_1, \alpha_{2,q+1}, \alpha_{q+2} \}
\]
and $\Delta(G/P)=\{ D_1^+=D_{q+2}^+, D_1^-, D_2, E_{q+1}, D_{q+2}^- \}$ with Cartan pairing
\[
\begin{array}{c|ccc}
      & \alpha_1 & \alpha_{2,q+1} & \alpha_{q+2} \\
\hline
D_1^+ & 1 & -1 & 1 \\
D_1^- & 1 & 0 & -1 \\
D_2 & -1 & 1 & 0 \\
E_{q+1} & 0 & 1 & -1 \\
D_{q+2}^- & -1 & 0 & 1 
\end{array}
\]
We can take $Q$ so that only $\alpha_{1}$ and $\alpha_{2}$ are not roots of $L_{Q}$, and the subset of colors corresponding to the inclusion $P\subseteq Q$ is $\{ D_1^+, E_{q+1}, D_{q+2}^-\}$. This is obvious if $q>1$, just by looking at which color is moved by which simple root. If $q=1$ the only other possibility $\{ D_1^+, D_2, D_{q+2}^-\}$ is excluded because it is not distinguished.

So the inverse images of the colors of $G/Q$ are $D_1^-$ and $D_2$, which yields
\[
\begin{array}{lclclcl}
\chi_{D_1^-} & = & -\omega_1^{P} & = &  -\varpi_{1},\\
\chi_{D_2} & = & -\omega_2^{P} & = &  -\varpi_{2}.
\end{array}
\]

The equalities~(\ref{eq:sphroot}) of Lemma~\ref{lemma:sphroot} yield
\[
\begin{array}{lcl}
\chi_{D_1^+} & = & 0,\\
\chi_{E_{q+1}} & = & \varpi_2,\\
\chi_{D_{q+2}^-} & = & \varpi_2-\varpi_1.
\end{array}
\]

\subsubsection{$I=\{\beta_{1}\}$ or $I=\{\beta_{p}\}$ with $2\leq p\le q$}\label{sym2-2S}\label{sym2-3S}
In view of Remark \ref{remark: duals} it is enough to consider only case \eqref{sym2-3}, i.e.~$I=\{\beta_{p}\}$.
The subgroup $P$ is conjugated to the subgroup $\widetilde{P}$ that appears in \cite{MR2183057} as case 11. The parameters ``$p$'' and ``$q$'' of loc.cit.\ are equal here respectively to $p$ and $q+1-p$. We have $gPg^{-1}=\widetilde{P}$ with $g=w_{0}^{G}w_{0}^{H}$.

The set of spherical roots and the set of colors are given by
\[
\Sigma(G/P) = \{\alpha_1,\ldots,\alpha_p, \alpha_{p+1,q+1}, \alpha_{q+2},\ldots,\alpha_n \},
\]
\[
\begin{array}{lcl}
\Delta(G/P) &=& \{ D_1^+ = D_n^+, D_2^+=D_{n-1}^+, \ldots,D_{p}^+=D_{q+2}^+, \\
            & & D_1^-= D_{n-1}^-, \ldots, D_{p-1}^-=D_{q+2}^-,D_{p}^-,D_{n}^-,\\
            & & D_{p+1}, E_{q+1} \}.
\end{array}
\]

The Cartan pairing is given by the following matrix:
\begin{scriptsize}
\[
\begin{array}{c|cccccccccccccc}
& \alpha_1 & \alpha_2 & \alpha_3 & \alpha_4 & \ldots & \alpha_{p-1} & \alpha_p & \alpha_{p+1,q+1} & \alpha_{q+2} & \alpha_{q+3}& \ldots & \alpha_{p+q-1} & \alpha_{p+q} & \alpha_{p+q+1} \\
\hline
D_1^+ & 1 & 0 & 0 & 0 &\ldots & 0 & 0 & 0 & 0 & 0 & \ldots & 0 & -1 & 1 \\
D_1^- & 1 & -1 & 0 & 0 &\ldots & 0 & 0 & 0 & 0 & 0 & \ldots & 0 & 1 & -1 \\
D_2^+ & -1 & 1 & 0 & 0 &\ldots & 0 & 0 & 0 & 0 & 0 & \ldots & -1 & 1 & 0 \\
D_2^- & 0 & 1 & -1 & 0 &\ldots & 0 & 0 & 0 & 0 & 0 & \ldots & 1 & -1 & 0 \\
D_3^+ & 0 & -1 & 1 & 0 &\ldots & 0 & 0 & 0 & 0 & 0 & \ldots &  1 & 0 & 0\\
D_3^- & 0 & 0 & 1 & -1 &\ldots & 0 & 0 & 0 & 0 & 0 & \ldots & -1 & 0 & 0\\
\vdots & \vdots &\vdots &\vdots &\vdots &\vdots &\vdots &\vdots &\vdots &\vdots &\vdots &\vdots &\vdots &\vdots &\vdots \\
D_{p-1}^+ & 0 & 0 & 0 & 0 &\ldots & 1 & 0 & 0 & -1 & 1 & \ldots & 0 & 0 & 0\\
D_{p-1}^- & 0 & 0 & 0 & 0 &\ldots & 1 & -1 & 0 & 1 & -1 & \ldots & 0 & 0 & 0\\
D_{p}^+ & 0 & 0 & 0 & 0 &\ldots & -1 & 1 & -1 & 1 & 0 & \ldots & 0 & 0 & 0\\
D_{p}^- & 0 & 0 & 0 & 0 &\ldots & 0 & 1 & 0 & -1 & 0 & \ldots & 0 & 0 & 0\\
D_{p+1} & 0 & 0 & 0 & 0 &\ldots & 0 & -1 & 1 & 0 & 0 & \ldots & 0 & 0 & 0\\
E_{q+1} & 0 & 0 & 0 & 0 &\ldots & 0 & 0 & 1 & -1 & 0 & \ldots & 0 & 0 & 0\\
D_{n}^- & -1 & 0 & 0 & 0 &\ldots & 0 & 0 & 0 & 0 & 0 & \ldots & 0 & 0 & 1\\
\end{array}
\]
\end{scriptsize}

We can take $Q$ so that only $\alpha_p$ and $\alpha_{p+1}$ are not simple roots of $L_Q$. Then $Q$ is minimal parabolic containing $P$.
The inverse images in $G/P$ of the two colors of $G/Q$ are $D_p^-$ and $D_{p+1}$. This is obvious if $p<q$, since these are the only colors of $G/P$ that are moved only by resp.~$\alpha_{p}$ and $\alpha_{p+1}$. If $p=q$, then one excludes the other possibility, namely $D_{p}^{-}$ and $E_{p+1}$, checking that the set $\Delta(G/P)\smallsetminus\{D_{p}^{-},E_{p+1}\}$ is not parabolic. This can be seen by looking just at the two spherical roots $\alpha_{p+1}$ and $\alpha_{p}$: the first implies $a_{p+1}>a_{p}^{+}$ (where $a_i^{(\pm)}$ is the coefficient of $D_i^{(\pm)}$ in a linear combination that is positive on all spherical roots) while the second implies $a_{p}^{+}-a_{p+1}-a_{p-1}^{-}>0$, but this yields a contradiction. Therefore
\[
\begin{array}{lclclcl}
\chi_{D_p^-} & = & -\omega_{p}^{P} & = &  -\varpi_{p},\\[5pt]
\chi_{D_{p+1}} & = & -\omega_{p+1}^{P} & = &  -\varpi_{p+1}.
\end{array}
\]

We obtain
\[
\begin{array}{lclclclclcl}
\chi_{D_1^-} & = & \chi_{D_2^-} & = & \ldots & = & \chi_{D_{p-1}^-} & = & \chi_{D_{n}^-} &=&-\varpi_p+\varpi_{p+1},\\[5pt]
& & \chi_{D_1^+} & = & \chi_{D_2^+} & = & \ldots & = & \chi_{D_p^+} & = & 0,\\[5pt]
& & & & & & & & \chi_{E_{q+1}} & = & \varpi_{p+1},
\end{array}
\]
which is easily checked to satisfy equalities~(\ref{eq:sphroot}) of Lemma~\ref{lemma:sphroot}.
\subsubsection{$I=\{\beta'_{1}\}$ or $I=\{\beta'_{q}\}$ with $1\le p<q$}\label{sym2-4S}\label{sym2-5S}
In view of Remark \ref{remark: duals} it is enough to consider only case \eqref{sym2-4}, i.e.~$I=\{\beta_{1}'\}$.
The subgroup $P$ is conjugated to the subgroup $\widetilde{P}$ that appears in \cite{MR2183057} as case 12, the parameters ``$p$'' and ``$q$'' of loc.cit.\ are equal here respectively to $p+1$ and $q-p$. We have $gPg^{-1}=\widetilde{P}$ with $g=w_{0}^{G}w_{0}^{H}$, where $w_{0}^{G}\in W_{G}$ and $w_{0}^{H}\in W_{H}$ are the longest Weyl group elements.

This gives
\[
\Sigma(G/P) = \{\alpha_1,\ldots,\alpha_{p+1}, \alpha_{p+2,q+1}, \alpha_{q+2},\ldots,\alpha_n \}
\]
and
\[
\begin{array}{lcl}
\Delta(G/P) &=& \{ D_1^+, D_2^+=D_{n}^+, \ldots,D_{p+1}^+=D_{q+2}^+, \\
            & & D_1^-= D_{n}^-, \ldots, D_{p}^-=D_{q+2}^-,D_{p+1}^-,\\
            & & D_{p+2}, E_{q+1} \}.
\end{array}
\]
The Cartan pairing is
\begin{scriptsize}
\[
\begin{array}{c|cccccccccccccc}
& \alpha_1 & \alpha_2 & \alpha_3 & \alpha_4 & \ldots & \alpha_{p} & \alpha_{p+1} & \alpha_{p+2,q+1} & \alpha_{q+2} & \alpha_{q+3}& \ldots & \alpha_{p+q-1} & \alpha_{p+q} & \alpha_{p+q+1} \\
\hline
D_1^+ & 1 & 0 & 0 & 0 &\ldots & 0 & 0 & 0 & 0 & 0 & \ldots & 0 & 0 & -1 \\
D_1^- & 1 & -1 & 0 & 0 &\ldots & 0 & 0 & 0 & 0 & 0 & \ldots & 0 & 0 & 1 \\
D_2^+ & -1 & 1 & 0 & 0 &\ldots & 0 & 0 & 0 & 0 & 0 & \ldots & 0 & -1 & 1 \\
D_2^- & 0 & 1 & -1 & 0 &\ldots & 0 & 0 & 0 & 0 & 0 & \ldots & 0 & 1 & -1 \\
D_3^+ & 0 & -1 & 1 & 0 &\ldots & 0 & 0 & 0 & 0 & 0 & \ldots &  -1 & 1 & 0\\
D_3^- & 0 & 0 & 1 & -1 &\ldots & 0 & 0 & 0 & 0 & 0 & \ldots & 1 & -1 & 0\\
\vdots & \vdots &\vdots &\vdots &\vdots &\vdots &\vdots &\vdots &\vdots &\vdots &\vdots &\vdots &\vdots &\vdots &\vdots \\
D_{p}^+ & 0 & 0 & 0 & 0 &\ldots & 1 & 0 & 0 & -1 & 1 & \ldots & 0 & 0 & 0\\
D_{p}^- & 0 & 0 & 0 & 0 &\ldots & 1 & -1 & 0 & 1 & -1 & \ldots & 0 & 0 & 0\\
D_{p+1}^+ & 0 & 0 & 0 & 0 &\ldots & -1 & 1 & -1 & 1 & 0 & \ldots & 0 & 0 & 0\\
D_{p+1}^- & 0 & 0 & 0 & 0 &\ldots & 0 & 1 & 0 & -1 & 0 & \ldots & 0 & 0 & 0\\
D_{p+2} & 0 & 0 & 0 & 0 &\ldots & 0 & -1 & 1 & 0 & 0 & \ldots & 0 & 0 & 0\\
E_{q+1} & 0 & 0 & 0 & 0 &\ldots & 0 & 0 & 1 & -1 & 0 & \ldots & 0 & 0 & 0
\end{array}
\]
\end{scriptsize}

We can take $Q$ so that only $\alpha_{p+1}$ and $\alpha_{p+2}$ are not simple roots of $L_Q$. The inverse images in $G/P$ of the two colors of $G/Q$ are $D_{p+1}^-$ and $D_{p+2}$ (if $q=1$ the other possibility $D_{p+1}^+$ and $E_{q+1}$ is excluded as before). Therefore
\[
\begin{array}{lclclcl}
\chi_{D_{p+1}^-} & = & -\omega_{p+1}^{P} & = &  -\varpi_{p+1},\\[5pt]
\chi_{D_{p+2}} & = & -\omega_{p+2}^{P} & = &  -\varpi_{p+2}.
\end{array}
\]

We obtain
\[
\begin{array}{lclclclcl}
\chi_{D_1^-} & = & \chi_{D_2^-} & = & \ldots & = & \chi_{D_{p}^-} & = & 0,\\
\chi_{D_1^+} & = & \chi_{D_2^+} & = & \ldots & = & \chi_{D_{p+1}^+} & = & \varpi_{p+1}-\varpi_{p+2},\\
 & & & & & & \chi_{E_{q+1}} & = & \varpi_{p+1}.
\end{array}
\]
which is easily checked to satisfy equalities~(\ref{eq:sphroot}) of Lemma~\ref{lemma:sphroot}.

\subsubsection{$I=\{\beta'_{i}\}$ with  $p=1,q\ge3$,  $i\in\{2,\ldots,q-1\}$}\label{sym2-6S}
A conjugate $\widetilde P$ of $P$ appears as case 7 of \cite{MR2183057}; the parameters ``$p$'' and ``$q$'' of loc.cit.\ are equal here respectively to $i-1$ and $q-i$. Denoting by $\widetilde H=\mathrm S(\mathrm L(i+2)\times \mathrm L(q+1-i))$ and by $g=w_0^{\widetilde H}$ the longest element of its Weyl group, then $\widetilde P=gP g\inv$. This yields
\[
\Sigma(G/P) = \{\alpha_1,\alpha_{2,i}, \alpha_{i+1},\alpha_{i+2,n-1},\alpha_n \}
\]
and
\[
\Delta(G/P) = \{ D_1^+ = D_{i+1}^+, D_1^-=D_{n}^+, D_{2}, E_{i}, D_{i+1}^-=D_{n}^-, D_{i+2}, E_{n-1} \},
\]
with Cartan pairing

\begin{scriptsize}
\[
\begin{array}{c|cccccc}
& \alpha_1 & \alpha_{2,i} & \alpha_{i+1} & \alpha_{i+2,n-1} & \alpha_n \\
\hline
D_1^+ & 1 & -1 & 1 & 0 & -1  \\
D_1^- & 1 & 0 & -1 & 0 & 1  \\
D_2 & -1 & 1 & 0 & 0 & 0  \\
E_{i} & 0 & 1 & -1 & 0 & 0  \\
D_{i+1}^- & -1 & 0 & 1 & -1 & 1  \\
D_{i+2} & 0 & 0 & -1 & 1 & 0  \\
E_{n-1} & 0 & 0 & 0 & 1 & -1  \\
\end{array}
\]
\end{scriptsize}

We can take $Q$ so that only $\alpha_{2}$ and $\alpha_{i+2}$ are not simple roots of $L_Q$, and the inverse images in $G/P$ of the two colors of $G/Q$ are $D_{2}$ and $D_{i+2}$. This is shown easily if $i>2$ or if $i<q-1$, as above: one excludes the other possibilities by looking at which color is moved by which simple root. Let us postpone the discussion of the case $i=2=q-1$.

This yields
\[
\begin{array}{lclclcl}
\chi_{D_{2}} & = & -\omega_{2}^{P} & = &  -\varpi_{2},\\[5pt]
\chi_{D_{i+2}} & = & -\omega_{i+2}^{P} & = &  -\varpi_{i+2}
\end{array}
\]
which implies
\[
\begin{array}{lclcl}
\chi_{E_{i}} & = & \chi_{D_{i+1}^-} & = & \varpi_{2}-\varpi_{i+2},\\
& & \chi_{D_1^+} & = & -\varpi_{i+2},\\
& & \chi_{D_1^-} & = & 0,\\
& & \chi_{E_{n-1}} & = & \varpi_{2}.
\end{array}
\]

If $i=2=q-1$, the computation above is the same, but there is another possible parabolic subset of colors, namely $\{ E_2, E_4\}$. It does correspond to a $G$-equivariant morphism $G/P\to G/Q$, which is not the one induced by the inclusion $P\subseteq Q$. It is rather obtained from the latter by twisting it with a $G$-equivariant automorphism of $G/P$. However, it is not necessary to prove these claims: for us, it is enough to notice that the resulting generators of the extended weight monoid of $G/P$ turn up to be the same (except for the ``labeling'' with the colors), as the reader may easily check, if one takes $E_2$ and $E_4$ instead of our choice $D_{2}$ and $D_{i+2}$.

\subsubsection{$p=0$ and $I=$ any subset of simple roots of $H$}\label{sym2-7S}
If $I=S_H$ then $P=B_H^-$: the spherical roots, the colors, the Cartan pairing, and also the extended weight monoid of $G/P$ were computed in this case in \cite{MR3366923}. Thanks to Remark~\ref{remark:subgroups}, the other cases of the subset $I$ follow from this one. However, for completeness, let us recompute the extended weight monoid of $G/B_H^-$ in a direct manner and see how the other cases of $I$ can be derived from our computation.

We do this by finding global equations in $G$ of the pull-back of the colors of $G/B_H^-$. For this task it is more convenient to first replace $G$ with $\GL(n+1)$, and replace the groups $B$ and $B_H^-$ with their inverse images in $\GL(n+1)$.

Consider the regular functions $a_i,b_i\in \CC[\GL(n+1)]$, for $i\in\{1,\ldots,n\}$ defined as follows. For $g\in \GL(n+1)$, we set $a_i(g)$ to be the determinant of the lower right minor of size $n+1-i$, and we set $b_i(g)$ to be the determinant of the minor of $g$ obtained by taking the last $n+1-i$ rows, the first column and the last $n-i$ columns.

The functions $a_i$ and $b_i$ are left $B$-semiinvariant and right $B^-_H$-semiinvariant, and they are given by irreducible polynomials. Their zero sets are moved only by $\alpha_i$. Therefore the images of their zero sets in $G/B_H^-$, denoted resp.\ by $D_i^+$ and $D_i^-$,  are the two colors moved by $\alpha_i$. Since no other colors can be moved by any simple root of $G$, we conclude
\[
\Delta(G/B_H^-) = \{D_1^+,D_1^-,\ldots, D_n^+,D_n^-\}.
\]
We get back to $G=\SL(n+1)$, where one easily deduces
\[
\chi_{D_i^+} = -\omega_i^{B_H^-} = -\varpi_i
\]
and
\[
\chi_{D_i^-} = \omega_1^{B_H^-}-\omega_{i+1}^{B_H^-}=\varpi_1-\varpi_{i+1}
\]
for all $i$.

We can now compute the extended weight monoid of $G/P$, with $P$ corresponding to any $I$ as above: it is enough to determine which colors of $G/B_H^-$ are mapped to colors of $G/P$ under the natural map $G/B_H^-\to G/P$ (i.e.\ are not mapped dominantly), which is elementary to accomplish given the global equations on $G$ we have given.

We obtain that $D_1^+$ is mapped to a color of $G/P$, and for $i\in\{2,\ldots,n\}$ the color $D_i^+$ is mapped to a color of $G/P$ if and only if $\beta_{i-1}\in I$. Moreover $D_n^-$ is mapped to a color of $G/P$, and for $i\in\{1,\ldots,n-1\}$ the color $D_i^-$ is mapped to a color of $G/P$ if and only if $\beta_i\in I$.

%%                        9.3 

\subsection{$\SO(2n+2)/(\SO(2n)\times \SO(2))$ with $n\geq 3$}\label{sym3}\label{sym3aS}\label{sym3bS}
This case is discussed in Section~\ref{s:example}.

%    9.4

\subsection{$\SO(2n+1)/\SO(2n)$ with $n\ge3$}\label{sym4}

Here $I$ is any subset of simple roots of $H$. We discuss the case $P=B_H^-$, and deal with the general case using Remark~\ref{remark:subgroups}.

It is useful to define $G=\SO(2n+1)$ to be the stabilizer of the symmetric bilinear form given by $(e_i,e_{2n-j+2})=\delta_{i,j}$ for all $i,j\in\{1,\ldots,2n+1\}$, where $(e_1,\ldots,e_{2n+1})$ is the standard basis of $\CC^{2n+1}$. We take $H=\SO(2n)$ to be the stabilizer in $G$ of $e_{n+1}$. We recall that, with this choice, the subgroups $B$, $T$, $B_H$, $B_H^-$, and $T_H=T$ can be taken as in Appendix~\ref{s:preliminaries}. We have $\alpha_1=\beta_1$,$\ldots$, $\alpha_{n-1}=\beta_{n-1}$, $\alpha_{n-1}+2\alpha_n=\beta_n$, and $\omega_1=\varpi_1$,$\ldots$, $\omega_{n-2}=\varpi_{n-2}$, $\omega_{n-1}=\varpi_{n-1}+\varpi_{n}$, $\omega_{n}=\varpi_{n}$.

The case $P=B_H^-$ belongs again to the class of subgroups studied in~\cite{MR3366923}. Here, for convenience, we determine the extended weight monoid of $G/P$ more directly, by finding global equations on $G$ of its colors.

For $i\in\{1,\ldots,n-1\}$, set $h=2n+1-i$ and let $f_i(A)$ be the determinant of the $h\times h$-minor in the lower right corner of the matrix $A\in G$. Then $f_i\in\CC[G]$ is $B$-semiinvariant for the left translation and $B^-$-semiinvariant for the right translation on $G$, and the $B$-weight is the $i$-th fundamental dominant weight. It follows that $f_i$ is a global equation of the (pull-back on $G$ of the) color $D_i$ of $G/B^-$ moved by $\alpha_i$. One defines similarly $f_n$ and $D_n$, and for the same reason the pull-back of $f_n$ to $\Spin(2n+1)$ is the square of a global equation for $D_{n}$.

The pull-back of these colors from $G/B^-$ to $G/B_H^-$ via the inclusion $B_H^-\subseteq B^-$ produces $n$ colors of $G/B_H^-$, denoted by $D^+_1,\ldots,D^+_n$, where for all $i$ the color $D^+_i$ is moved exactly by $\alpha_i$. Their $B$-weights are evidently the fundamental dominant weights, and the $B_-$-weights are the opposites.  To summarize, we have
\[
\chi_{D_i^+}  =  -\omega_{i}^{B_H^-} \quad\forall i\in\{1,\ldots,n\}
\]
i.e.\
\[
\begin{array}{lcl}
\chi_{D_i^+}  &=&  -\varpi_{i} \qquad\forall i\in\{1,\ldots,n-2,n\},\\
\chi_{D_{n-1}^+} &=& - \varpi_{n-1}-\varpi_{n}
\end{array}
\]

We claim that $G/B_H^-$ has other $n$ colors $D_1^-,\ldots,D_n^-$. For $j\in\{1,\ldots,n \}$, set $k=2n+1-j$ and denote by $F_j(A)$ the determinant of the $k\times k$-minor involving the last $k$ rows of the matrix $A\in G$ and the columns $j$ through $n$ and $n+2$ through $2n+1$. The function $F_j\in\CC[G]$ is $B$-semiinvariant under left translation and $B_H^-$-semiinvariant under right translation. If $j<n$ then its $B$-weight is $\omega_j$, and the $B$-weight of $F_n$ is $2\omega_n$.

From Lemma~\ref{lemma:foschi}, it follows that $F_1,\ldots,F_{n-1}$ are global equations of $n-1$ colors $D^-_1,\ldots,D^-_{n-1}$ of $G/B_H^-$, where $D^-_j$ is moved exactly by $\alpha_j$. Since a simple root can move up to $2$ colors, there is at most one other color in $G/B_H^-$, and in this case it is moved by $\alpha_n$. The divisor of the function $F_n$ has irreducible components that are colors moved only by $\alpha_n$, and it does not coincide (as a set) with $D_n^+$. So the additional color moved by $\alpha_n$ exists and $F_n$ vanishes on it, let us denote it by $D_n^-$. Given the $B$-weight of $F_n$, its divisor is either $D_n^+ + D_n^-$ or $2D_n^-$. The second possibility is excluded because the $B_H^-$-weight of $F_n$ is not divisible by $2$.

This yields
\[
\begin{array}{lcl}
\chi_{D_1^-}  &=&  0,\\
\chi_{D_i^-}  &=&  -\omega_{i-1}^{B_H^-} \qquad\forall i\in\{2,\ldots,n-1\},\\
\chi_{D_n^-}  &=&  -\omega_{n-1}^{B_H^-}+\omega_{n}^{B_H^-}.
\end{array}
\]
i.e.\
\[
\begin{array}{lcl}
\chi_{D_1^-}  &=&  0,\\
\chi_{D_i^-}  &=&  -\varpi_{i-1} \qquad\forall i\in\{2,\ldots,n\}.
\end{array}
\]
Since $B_H^-$ is minimal among parabolic subgroups of $H$, we may skip the other possibilities for $P$ (see Remark~\ref{remark:subgroups}).

%    9.5

\subsection{$\SO(2n+2)/\SO(2n+1)$ with $n\ge3$}\label{sym5} Here $I$ is any subset of simple roots of $H$. We discuss the case $P=B_H^-$, and deal with the general case using Remark~\ref{remark:subgroups}.

It is useful to define $G=\SO(2n+2)$ to be the stabilizer of the symmetric bilinear form given by $(e_i,e_{2n-j+3})=\delta_{i,j}$ for all $i,j\in\{1,\ldots,2n+2\}$, where $(e_1,\ldots,e_{2n+2})$ is the standard basis of $\CC^{2n+2}$. We take $H=\SO(2n+1)$ to be the stabilizer in $G$ of $e_{n+1}-e_{n+2}$. We recall that, with this choice, the subgroups $B$, $T$, $B_H$, $B_H^-$, and $T_H$ can be taken as in Appendix~\ref{s:preliminaries}. The simple roots of $G$ are $\alpha_1=\varepsilon_1-\varepsilon_2,\ldots,\alpha_n=\varepsilon_n-\varepsilon_{n+1},\alpha_{n+1}=\varepsilon_n+\varepsilon_{n+1}$, where $\varepsilon_i$ is the character of $T$ given by taking the $i$-th entry on the diagonal. We have $\alpha_i|_{T_H}=\beta_i$ and $\omega_i|_{T_H}=\varpi_i$ for all $i\in\{1,\ldots,n\}$, and $\alpha_{n+1}|_{T_H}=\beta_n$, $\omega_{n+1}|_{T_H}=\varpi_n$,  since $\varepsilon_{n+1}$ is trivial on $T_H$.

To discuss the case $P=B_H^-$ we use the results on the extended weight monoid in \cite{MR3366923}. Following loc.cit., the first step is to compute the {\em active roots} of $B_H^-$, i.e.\ the positive roots $\alpha$ of $G$ such that the root space $\fg_{-\alpha}$ of the Lie algebra $\fg$ of $G$ is not contained in the Lie algebra of $B_H^-$. It is elementary to compute the active roots in this case, and they are
\[
\begin{array}{rr}
\alpha_n, & \alpha_{n+1}, \\
\alpha_{n-1}+\alpha_n, & \alpha_{n-1}+\alpha_{n+1}, \\
\alpha_{n-2}+\alpha_{n-1}+\alpha_n, & \alpha_{n-2}+\alpha_{n-1}+\alpha_{n+1}, \\
\vdots & \vdots \\
\alpha_{1}+\alpha_{2}+\ldots+\alpha_{n-1}+\alpha_n, & \alpha_{1}+\alpha_{2}+\ldots+\alpha_{n-1}+\alpha_{n+1}.
\end{array}
\]
Notice that active roots on the same line in the above list have the same restriction to $T_H$. Denote by $\tau$ the restriction of characters from $T$ to $T_H$, and set $\varphi_i=\tau(\alpha_i+\ldots+\alpha_n)$. Then we observe that $\tau(\alpha_i)=\varphi_i-\varphi_{i+1}$ for all $i\in\{1,\ldots,n-1\}$, and $\tau(\alpha_n)=\tau(\alpha_{n+1})=\varphi_n$.

By the results of loc.cit., we have the following:
\[
\Sigma(G/B_H^-) = S
\]
and
\[
\Delta(G/B_H^-) = \left\{ D_i^+, D_i^-\;|\; i\in \{1,\ldots,n+1 \}\right\}
\]
where $D^-_n=D^-_{n+1}$. The Cartan pairing is as follows (all the other values that are not explicitly given are zero):
\begin{enumerate}
\item $\<\rho(D_1^+),\alpha_1\> =\<\rho(D_1^-),\alpha_1\>=1$, and $\<\rho(D_1^+),\alpha_2\> =-1$;
\item for all $i\in\{2,\ldots n-2\}$ we have $\<\rho(D_i^+),\alpha_{i}\>=\<\rho(D_i^-),\alpha_{i}\>=1$, and $\<\rho(D_i^+),\alpha_{i+1}\> =\<\rho(D_i^-),\alpha_{i-1}\>-1$;
\item we have
\[
\begin{array}{c|cccc}
      & \alpha_{n-2} & \alpha_{n-1} & \alpha_{n} &\alpha_{n+1}\\
\hline
D_{n-1}^+ & 0 & 1 & -1 & -1 \\
D_{n-1}^- & -1 & 1 & 0 & 0 \\
D_{n}^+ & 0 & 0 & 1 & -1 \\
D_{n}^- & 0 & -1 & 1 & 1 \\
D_{n+1}^+ & 0 & 0 & -1 & 1
\end{array}
\]
\end{enumerate}
Again using loc.cit., or applying our usual method, we obtain
\[
\begin{array}{lcl}
\chi_{D_i^+} & = & -\omega_{i}^{B_H^-} \quad\forall i\in\{1,\ldots,n+1\},\\
\chi_{D_1^-} & = & 0,\\
\chi_{D_j^-} & = & -\omega_{j-1}^{B_H^-} \quad\forall j\in\{2,\ldots,n\}.
\end{array}
\]
which we can rewrite as
\[
\begin{array}{lcl}
\chi_{D_i^+} & = & -\varpi_{i} \quad\forall i\in\{1,\ldots,n\},\\
\chi_{D_{n+1}^+} & = & -\varpi_{n},\\
\chi_{D_1^-} & = & 0,\\
\chi_{D_j^-} & = & -\varpi_{j-1} \quad\forall j\in\{2,\ldots,n\}.
\end{array}
\]

Since $B_H^-$ is minimal among parabolic subgroups of $H$, we may skip the other possibilities for $P$ (see Remark~\ref{remark:subgroups}).

%    8.6
\subsection{$\SO(2n+2)/\GL(n+1)$ with $n\ge 3$}\label{sym6S} The subgroup $H$ is a Levi subgroup of the parabolic subgroup of $G$ by omitting the one but last root $\alpha_{n}$ of $G$. We have the following possibilities for $I$:
\begin{enumerate}
\item\label{sym6a} $I=\{\beta_{1}\}$,
\item\label{sym6b} $I=\{\beta_n\}$.
\end{enumerate}

The two cases are related by the observation in Remark \ref{remark: duals}. For this reason we only discuss case~(\ref{sym6a}), and we notice that a conjugate of the corresponding $P$ appears as case 55 of \cite{MR2183057}: an element conjugating one subgroup to the other is the longest element of the Weyl group of the Levi subgroup of $G$ having all simple roots of $G$ except for $\alpha_1$.

We denote as usual by $\varpi_1,\ldots,\varpi_n$ the fundamental weights of $H$. Noticing that $T=T_H$, we have $\varpi_i=\omega_i$ for all $i\in\{1,\ldots,n-1 \}$, and $\varpi_{n}=\omega_n+\omega_{n+1}$. We also have $\frac{1}{2}\varpi_{n+1}=\omega_{n+1}$, whose double is a generator of the character group of $H$. This yields
\[
\Sigma(G/P) = \{\alpha_1+\alpha_2,\alpha_2+\alpha_3,\ldots,\alpha_{n-1}+\alpha_{n},\alpha_{n+1}\}
\]
and
\[
\Delta(G/P) = \{ D_1,\ldots, D_{n}, D_{n+1}^+, D_{n+1}^-\},
\]
with Cartan pairing
\begin{scriptsize}
\[
\begin{array}{c|cccccccccc}
& \alpha_1+\alpha_2 & \alpha_2+\alpha_3 & \alpha_3+\alpha_4 & \alpha_4+\alpha_5 & \ldots & \alpha_{n-3}+\alpha_{n-2} & \alpha_{n-2}+\alpha_{n-1} & \alpha_{n-1}+\alpha_{n} & \alpha_{n+1}\\
\hline
D_1 & 1 & -1 & 0 & 0 &\ldots & 0 & 0 & 0 & 0 \\
D_2 & 1 & 1 & -1 & 0 & \ldots & 0 & 0 & 0 & 0 \\
D_3 & -1 & 1 & 1 & -1 & \ldots & 0 & 0 & 0 & 0 \\
\vdots &  &  &  &  &  &  &  &  & \\
D_{n-2} & 0 & 0 & 0 & 0 & \ldots & 1 & 1 & -1 & 0 \\
D_{n-1} & 0 & 0 & 0 & 0 & \ldots & -1 & 1 & 1 & -1 \\
D_{n} & 0 & 0 & 0 & & \ldots & 0 & -1 & 1 & 0 \\
D_{n+1}^+ & 0 & 0 & 0 & 0 & \ldots & 0 & 0 & -1 & 1 \\
D_{n+1}^- & 0 & 0 & 0 & 0 & \ldots & 0 & -1 & 0 & 1 \\
\end{array}
\]
\end{scriptsize}

We can take $Q$ to be such that only $\alpha_1$ and $\alpha_{n+1}$ are not simple roots of $L_Q$. Then $Q$ is minimal containing $P$. The colors of $G/P$ mapping not dominantly to $G/Q$ are $D_1$ and either $D_{n+1}^+$ or $D_{n+1}^-$. We claim that $D_{n+1}^-$ is not mapped dominantly.

To prove the claim, we distinguish the two cases of $n$ even and $n$ odd.

If $n$ is even, we consider the natural projection $G/P\to G/H$. The homogeneous space $G/H$ appears as case 16 of \cite{MR3345829}, where the index ``$p$'' of loc.cit.\ is equal here to $n+1$. From loc.cit.\ we see that $G/H$ has a unique color $\widetilde D_{n+1}$ moved by $\alpha_{n+1}$, and it takes value $-1$ on the spherical root $\alpha_{n-3}+2\alpha_{n-2}+\alpha_{n-1}$ of $G/H$. We deduce that the inverse image of $\widetilde D_{n+1}$ in $G/P$ is $D_{n+1}^-$.

In order to apply this fact to the morphism $G/P\to G/Q$, we consider the parabolic subgroup $Q_0$ of $G$ containing $B^-$ and with Levi subgroup containing $T$ having all simple roots of $G$ except for $\alpha_{n+1}$. Then $Q_0$ contains $H$, and $G/Q_0$ has a unique color moved by $\alpha_{n+1}$. The diagram
\[
\begin{array}{ccc}
G/P & \to & G/H \\
\downarrow &  & \downarrow \\
G/Q & \to & G/Q_0
\end{array}
\]
commutes, where the maps are all induced by the inclusions of the respective subgroups. Denote by $E_{n+1}$ the color of $G/Q$ moved by $\alpha_{n+1}$.

Since the unique color of $G/Q_0$ has inverse image $\widetilde D_{n+1}$ in $G/H$ and $E_{n+1}$ in $G/Q$, it follows that the inverse images in $G/P$ of $\widetilde D_{n+1}$ and of $E_{n+1}$ coincide. In other words $D_{n+1}^-$ is mapped to $E_{n+1}$, whence our claim.

We discuss now the case $n$ odd. We use a similar argument, but with another subgroup $K\subseteq G$ instead of $H$. Namely, we consider $K= Q_1^u L_Q$, where $Q_1$ is the parabolic subgroup of $G$ containing $B^-$ and with Levi subgroup containing $T$ having all simple roots of $G$ except for $\alpha_{1}$. Notice that $Q_1\supset Q$, hence $Q_1^u\subseteq Q^u$ and $L_{Q_1}\supset L_Q$.

Our choice implies that $G/K$ is the parabolic induction of the homogeneous space $L_{Q_1}/L_{Q}$ by means of the parabolic subgroup $Q_1$ of $G$. For details on parabolic induction we refer to \cite[Section~3.4]{MR1896179}. Notice that $L_{Q_1}/L_{Q}$ is isomorphic to $\PSO(2n)/\PGL(n)$, and it is identified with the subset $Q_1/K$ of $G/K$.

The consequence is that $G/K$ has the same spherical roots as $L_{Q_1}/L_{Q}$, and the colors of $G/K$ that are moved by simple roots of $L_{Q_1}$ correspond bijectively (via the assignment $D\mapsto D\cap (Q_1/K)$) to the colors of $L_{Q_1}/L_{Q}$, in such a way to preserve the Cartan pairing and the property for a color to be moved by a simple root.

The homogeneous space $\PSO(2n)/\PGL(n)$ appears in \cite{MR3345829}, as case 16 if $n\geq 5$ (with $p$ of loc.cit.\ being equal to $n$), and as case 3 (with $p$ and $q$ loc.cit.\ being equal to resp.\ $1$ and $3$) if $n=3$. In both cases, taking into account its isomorphism with $L_{Q_1}/L_{Q}$, the simple root $\alpha_{n+1}$ moves only one color, which takes the value $1$ on the spherical root $\alpha_{n-1}+\alpha_n+\alpha_{n+1}$.

We conclude that $G/K$ has only one color $\widetilde D_{n+1}$ moved by $\alpha_{n+1}$, and it takes value $1$ on $\alpha_{n-1}+\alpha_n+\alpha_{n+1}$. This enables us to conclude the above claim, using the diagram
\[
\begin{array}{ccc}
G/P & \to & G/K \\
\downarrow &  & \downarrow \\
G/Q & \to & G/Q_1
\end{array}
\]
and by reasoning exactly as in the case $n$ even.

The claim, and equalities~(\ref{eq:sphroot}) of Lemma~\ref{lemma:sphroot}, yield
\[
\begin{array}{lclclclclcl}
\chi_{D_1} & = & \chi_{D_3} & = & \ldots & = & \chi_{D_{n-1}}  & = & -\omega_{1}^P &=& -\varpi_1,\\
& & \chi_{D_2} & = & \chi_{D_4} & = & \ldots & = & \chi_{D_{n-2}}  & = & 0,\\
 &  &  &  &  &  & \chi_{D_n} & = & \omega_{n+1}^P &=& \frac{1}{2}\varpi_{n+1},\\
 &  &  &  &  &  & \chi_{D_{n+1}^-} & = & -\omega_{n+1}^P &=&-\frac{1}{2}\varpi_{n+1},\\
 &  &  &  &  &  & \chi_{D_{n+1}^+} & = & \omega_{n+1}^P-\omega_{1}^P&=&\frac{1}{2}\varpi_{n+1}-\varpi_1
\end{array}
\]
if $n$ is even, and
\[
\begin{array}{lclclclclcl}
\chi_{D_1} & = & \chi_{D_3} & = & \ldots & = & \chi_{D_{n-2}}  & = & -\omega_{1}^P&=&-\varpi_1,\\
& &\chi_{D_2} & = & \chi_{D_4} & = & \ldots & = & \chi_{D_{n-1}}  & = & 0,\\
 &  &  &  &  &  & \chi_{D_n} & = & \omega_{n+1}^P-\omega_{1}^P&=&\frac{1}{2}\varpi_{n+1}-\varpi_{1},\\
 &  &  &  &  &  & \chi_{D_{n+1}^-} & = & -\omega_{n+1}^P &=&-\frac{1}{2}\varpi_{n+1},\\
 &  &  &  &  &  & \chi_{D_{n+1}^+} & = & \omega_{n+1}^P&=&\frac{1}{2}\varpi_{n+1}.
\end{array}
\]
if $n$ is odd.

%    8.7

\subsection{$\Sp(2p+2q)/(\Sp(2p)\times \Sp(2q))$ with $1\le p$, $1\le q$.}\label{sym7} Set $m=\min\{p,q\}$. Notice that in the paper \cite{MR3127988} the assumption $p\leq q$ is made, while here we drop it for convenience in the notations.

We define the subgroup $H$ as the stabilizer of the subspace generated by the first $p$ and the last $p$ vectors of the canonical basis of $\CC^{2p+2q}$. With this choice, the simple roots of $H$ are $\beta_i=\alpha_i$ for $i\in\{1,\ldots,p-1 \}$, $\beta_p=2\alpha_{p,p+q-1}+\alpha_{p+q}$, and $\beta'_j=\alpha_{p+j}$ for $j\in\{1,\ldots,q \}$. We also have $\varpi_i=\omega_i$ for all $i\in\{1,\ldots,p\}$, and $\varpi'_j+\varpi_p=\omega_{j+p}$ for all $j\in\{1,\ldots,q \}$.

It is useful to recall the spherical roots of $G/H$, taken from \cite[Symmetric case number 10]{MR3345829}:
\[
\begin{array}{lcl}
\Sigma(G/H) &=& \{ 	\alpha_{1}+2\alpha_2+\alpha_3, \alpha_3+2\alpha_4+\alpha_5, \ldots, \alpha_{2m-3}+2\alpha_{2m-2}+\alpha_{2m-1},\\
& & \alpha_{2m-1}+2\alpha_{2m}+ \ldots + 2\alpha_{n-1}+\alpha_{n} \}. 
\end{array}
\]

We have the following possibilities for the set $I$.

\begin{enumerate}
\item\label{sym7a} $I=\{\beta_1\}$,
\item\label{sym7c} $I=\{\beta_p\}$ assuming $p\geq 2$, and we have $p\leq 3$ or $q\leq 2$,
\item\label{sym7e} $I=\{\beta_{i}\}$ for all $i\in\{2,\ldots,p-1\}$, with $p\geq 3$ and $q=1$,
\item\label{sym7f} $I=\{\beta_{1},\beta_{2}\}$ with $p=2$ and $q\geq 2$,
\item\label{sym7h} $I=\{\beta_{i},\beta_{j}\}$ for any $i,j\in \{1,\ldots,p\}$ such that $i< j$, with $p\geq 2$ and $q=1$,
\item\label{sym7i} $I=\{\beta_{1},\beta'_{i}\}$ for any $i\in \{1,\ldots,q\}$, with $p=1$.
\end{enumerate}

It is useful to distinguish in our notations the different subgroups $P$ arising in each case. For this reason, we denote them by $P_{\ref{sym7a}},\ldots,P_{\ref{sym7i}}$. If required by the context, we will also add to our notation the indices $i,j$ appearing in the above list, writing e.g.\ $P_{\ref{sym7e}}(i)$ and $P_{\ref{sym7h}}(i,j)$ instead of simply $P_{\ref{sym7e}}$ and $P_{\ref{sym7h}}$.

Moreover, while discussing each case, we will introduce the usual parabolic subgroup $Q$, the Levi subgroup $L$ of $P$ containing $T$, and also another auxiliary subgroup $K$ of $G$. As for $P$, we denote them resp.\ by $Q_{\ref{sym7a}},\ldots,Q_{\ref{sym7i}}$, $L_{\ref{sym7a}},\ldots,L_{\ref{sym7i}}$, and $K_{\ref{sym7a}},\ldots,K_{\ref{sym7i}}$.

For $i\in\{1,\ldots,\ref{sym7i}\}$, let us choose a parabolic subgroup $Q_i$ minimal containing $B_-$ and $P_i$. We only specify which simple roots are not simple roots of $L_{Q_i}$:
\begin{itemize}
\item[-] $\alpha_1$ for $Q_{\ref{sym7a}}$,
\item[-] $\alpha_p$ for $Q_{\ref{sym7c}}$,
\item[-] $\alpha_i$ for $Q_{\ref{sym7e}}$,
\item[-] $\alpha_1$ and $\alpha_2$ for $Q_{\ref{sym7f}}$,
\item[-] $\alpha_i$ and $\alpha_j$ for $Q_{\ref{sym7h}}$,
\item[-] $\alpha_1$ and $\alpha_{i+1}$ for $Q_{\ref{sym7i}}$,
\end{itemize}
It is elementary to check that, for all $i$, the subgroup $L_i$ is a very reductive subgroup of $L_{Q_i}$. This implies that $Q_i$ is a minimal parabolic subgroup of $G$ containing $P_i$.

We discuss now the single cases.
\subsubsection{$I=\{\beta_{1}\}$.}\label{sym7aS}
%\textbf{Case \ref{sym7a}:} $I=\{\beta_{1}\}$.
We claim that the spherical roots and the colors of $G/P_{\ref{sym7a}}$ are as follows. If $p>q$ (notice that then $2q<p+q$), we have
\[
\begin{array}{lcl}
\Sigma(G/P_{\ref{sym7a}}) &=& \{ \alpha_1+\alpha_2,\alpha_2+\alpha_3,\ldots, \alpha_{2q-2}+\alpha_{2q-1}, \alpha_{2q-1}+\alpha_{2q},\\
& & \alpha_{2q}+2\alpha_{2q+1,p+q-1}+\alpha_{p+q} \}
\end{array}
\]
and
\[
\Delta(G/P_{\ref{sym7a}}) = \{ D_1,D_2,\ldots, D_{2q+1}\}
\]
with Cartan pairing $\rho(D_i)=\alpha_i^\vee|_{\Xi(G/P_{\ref{sym7a}})}$ for all $i$. 

If instead $p\leq q$ (and then $2p-1<p+q$), we have
\[
\Sigma(G/P_{\ref{sym7a}}) = \{ \alpha_1+\alpha_2,\alpha_2+\alpha_3,\ldots, \alpha_{2p-2}+\alpha_{2p-1}, \alpha_{2p-1}+2\alpha_{2p,p+q-1}+\alpha_{p+q} \}.
\]
and
\[
\Delta(G/P_{\ref{sym7a}}) = \{ D_1,D_2,\ldots, D_{2p}\}
\]
with Cartan pairing $\rho(D_i)=\alpha_i^\vee|_{\Xi(G/P_{\ref{sym7a}})}$ for all $i$. 

To prove the claim, let us denote resp.\ by $\Sigma$ and $\Delta$ the sets at the right hand side of the above equalities. By the classification of spherical homogeneous spaces, there exists a spherical homogeneous space $G/K_{\ref{sym7a}}$ such that $\Sigma=\Sigma(G/K_{\ref{sym7a}})$, $\Delta=\Delta(G/K_{\ref{sym7a}})$, $\ZZ\Sigma=\Xi(G/K_{\ref{sym7a}})$, and the Cartan pairing is as specified. By Corollary~\ref{cor:morphisms}, the subgroup $K_{\ref{sym7a}}$ is conjugated in $G$ to a subgroup of $H$, so we may assume $K_{\ref{sym7a}}\subseteq H$.

It is elementary to check the hypotheses of Proposition~\ref{prop:PinH} in this case, and conclude that $K_{\ref{sym7a}}$ is a parabolic subgroup of $H$. Also, one checks that it is a maximal proper subgroup of $H$, again by Corollary~\ref{cor:morphisms}. Hence $K_{\ref{sym7a}}$ must be conjugated in $G$ to $P_{\ref{sym7a}}$, $P_{\ref{sym7c}}$, or $P_{\ref{sym7e}}$ which are the maximal proper subgroups of $H$ in our list.

One checks that there is no parabolic subset of $\Delta$ that could correspond to a $G$-equivariant morphism $G/K_{\ref{sym7a}}\to G/Q_{\ref{sym7c}}$ nor to a $G$-equivariant morphism $G/K_{\ref{sym7a}}\to G/Q_{\ref{sym7e}}$. Hence $K_{\ref{sym7a}}$ is conjugated to $P_{\ref{sym7a}}$ in $G$.

There is only one color of $G/P_{\ref{sym7a}}$ mapping not dominantly to $G/Q_{\ref{sym7a}}$, namely $D_1$, yielding
\[
\chi_{D_1} = -\omega_1^{P_{\ref{sym7a}}} = -\varpi_1.
\]
The equalities~(\ref{eq:sphroot}) of Lemma~\ref{lemma:sphroot} yield
\[
\begin{array}{lclclclcl}
\chi_{D_2} & = & \chi_{D_4} & = & \ldots & = & \chi_{D_{2q}}  & = & 0,\\
\chi_{D_3} & = & \chi_{D_5} & = & \ldots & = & \chi_{D_{2q+1}}  & = & -\varpi_1.
\end{array}
\]
if $p>q$, and
\[
\begin{array}{lclclclcl}
\chi_{D_2} & = & \chi_{D_4} & = & \ldots & = & \chi_{D_{2p}}  & = & 0,\\
\chi_{D_3} & = & \chi_{D_5} & = & \ldots & = & \chi_{D_{2p-1}}  & = & -\varpi_1.
\end{array}
\]
if $p\leq q$.

\subsubsection{$I=\{\beta_p\}$ assuming $p\geq 2$, and we have $p\leq 3$ or $q\leq 2$}\label{sym7cS}
%\textbf{Case \ref{sym7c}:} $I=\{\beta_p\}$ assuming $p\geq 2$, and we have $p\leq 3$ or $q\leq 2$.
We claim that $G/P$ has spherical roots and colors as described below, where all possible values of $p$ and $q$ are discussed. In each case, the proof of this claim goes as before: the objects we indicate correspond in each case to a maximal proper parabolic subgroup of $H$, and there is no $G$-equivariant morphism from the homogeneous space to $G/Q_{\ref{sym7a}}$ or to $G/Q_{\ref{sym7e}}$. The claim follows.

For $p=3$ and $q\geq 3$:
\[
\Sigma(G/P_{\ref{sym7c}})= \{ \alpha_1, \ldots, \alpha_5, \sigma=\alpha_5+2\alpha_{6,q+2}+\alpha_{q+3} \}
\]
and
\[
\Delta(G/P_{\ref{sym7c}}) =  \{ D_1^+ = D_4^+, D_1^- = D_3^- = D_5^-, D_2^+ = D_5^+, D_2^-, D_3^+, D_4^-, D_6 \}
\]
with Cartan pairing
\begin{scriptsize}
\[
\begin{array}{c|ccccccc}
& \alpha_1 & \alpha_2 & \alpha_3 & \alpha_4 & \alpha_{5} & \sigma \\
\hline
D_1^+ & 1 & 0 & -1 & 1 & -1 & 0 \\
D_1^- & 1 & -1 & 1 & -1 & 1 & 0 \\
D_2^+ & -1 & 1 & -1 & 0 & 1 & 0 \\
D_2^- & 0 & 1 & 0 & 0 & -1 & 0 \\
D_3^+ & -1 & 0 & 1 & 0 & -1 & 0 \\
D_4^- & -1 & 0 & 0 & 1 & 0 & -1 \\
D_6   & 0 & 0 & 0 & 0 & -1 & 1
\end{array}
\]
\end{scriptsize}

There is only one color of $G/P_{\ref{sym7c}}$ mapping not dominantly to $G/Q_{\ref{sym7c}}$, namely $D_3^+$, yielding
\[
\chi_{D_3^+} = -\omega_3^{P_{\ref{sym7c}}} = -\varpi_3.
\]
The equalities~(\ref{eq:sphroot}) of Lemma~\ref{lemma:sphroot} yield
\[
\begin{array}{lclclcl}
\chi_{D_1^+} & = & \chi_{D_1^-} & = & \chi_{D_2^+}  & = & -\varpi_3,\\
\chi_{D_2^-} & = & \chi_{D_4^-} & = & \chi_{D_6}  & = & 0.
\end{array}
\]

For $p=3$ and $q=2$ we have
\[
\Sigma(G/P_{\ref{sym7c}})= S
\]
and
\[
\Delta(G/P_{\ref{sym7c}}) =  \{ D_1^+ = D_{4}^+, D_1^- = D_3^- = D_{5}^-, D_2^+ = D_{5}^+, D_2^-, D_3^+, D_4^-\}
\]
with Cartan pairing
\begin{scriptsize}
\[
\begin{array}{c|ccccc}
& \alpha_1 & \alpha_2 & \alpha_3 & \alpha_{4} & \alpha_{5} \\
\hline
D_1^+ & 1 & 0 & -1 & 1 & -1 \\
D_1^- & 1 & -1 & 1 & -1 & 1 \\
D_2^+ & -1 & 1 & -1 & 0 & 1 \\
D_2^- & 0 & 1 & 0 & 0 & -1 \\
D_3^+ & -1 & 0 & 1 &  0 & -1 \\
D_4^- & -1 & 0 & 0 & 1 & -1
\end{array}
\]
\end{scriptsize}
There is only one color of $G/P_{\ref{sym7c}}$ mapping not dominantly to $G/Q_{\ref{sym7c}}$, namely $D_3^+$. This yields
\[
\chi_{D_3^+} = -\omega_3^{P_{\ref{sym7c}}} = -\varpi_3.
\]
The equalities~(\ref{eq:sphroot}) of Lemma~\ref{lemma:sphroot} yield
\[
\begin{array}{lclclcl}
\chi_{D_1^+} & = & \chi_{D_1^-} &=& \chi_{D_2^+} &=& -\varpi_3,\\
      & &     \chi_{D_2^-} & = & \chi_{D_4^-} & = & 0.     
\end{array}
\]

For $p=2$ and $q\geq 2$ it holds
\[
\Sigma(G/P_{\ref{sym7c}})= \{ \alpha_1+\alpha_3, \alpha_2, \sigma=\alpha_3+2\alpha_{4,q+1}+\alpha_{q+2} \}
\]
and
\[
\Delta(G/P_{\ref{sym7c}}) =  \{ D_1 = D_3, D_2^+, D_2^-, D_4 \}
\]
with Cartan pairing
\[
\begin{array}{c|ccc}
& \alpha_1+\alpha_3 & \alpha_2 & \sigma \\
\hline
D_1 & 2 & -1 & 0  \\
D_2^+ & 0 & 1 & -1 \\
D_2^- & -2 & 1 & 0 \\
D_4 & -1 & 0 & 1 \\
\end{array}
\]

There is only one color of $G/P_{\ref{sym7c}}$ mapping not dominantly to $G/Q_{\ref{sym7c}}$, namely $D_2^-$. The only other possibility would be $D_2^+$, excluded by the fact that the other colors do not form a parabolic subset. This yields
\[
\chi_{D_2^-} = -\omega_2^{P_{\ref{sym7c}}} = -\varpi_2.
\]
The equalities~(\ref{eq:sphroot}) of Lemma~\ref{lemma:sphroot} yield
\[
\begin{array}{lclcl}
\chi_{D_2^+} & = & \chi_{D_4} & = & 0,\\
& & \chi_{D_1} & = & -\varpi_2.
\end{array}
\]

For $q=2$ and $p\geq 4$ we have
\[
\Sigma(G/P_{\ref{sym7c}})= \{ \alpha_1, \alpha_2, \alpha_3, \sigma=\alpha_{4,p}, \alpha_{p+1}, \alpha_{p+2}\}
\]
and
\[
\Delta(G/P_{\ref{sym7c}}) =  \{ D_1^+ = D_{q+1}^+, D_1^- = D_3^- = D_{p+2}^-, D_2^+ = D_{p+2}^+, D_2^-, D_3^+=D_{p+1}^-, D_4, E_{p}\}
\]
with Cartan pairing
\begin{scriptsize}
\[
\begin{array}{c|cccccc}
& \alpha_1 & \alpha_2 & \alpha_3 & \sigma & \alpha_{p+1} & \alpha_{p+2} \\
\hline
D_1^+ & 1 & 0 & -1 & 0 & 1 & -1 \\
D_1^- & 1 & -1 & 1 & 0 & -1 & 1 \\
D_2^+ & -1 & 1 & -1 & 0 & 0 & 1 \\
D_2^- & 0 & 1 & 0 & 0 & 0 & -1 \\
D_3^+ & -1 & 0 & 1 & -1 & 1 & -1 \\
D_4   & 0 & 0 & -1 & 1 & 0 & 0 \\
E_p   & 0 & 0 & 0 & 1 & -1 & 0
\end{array}
\]
\end{scriptsize}

There is only one color of $G/P_{\ref{sym7c}}$ mapping not dominantly to $G/Q_{\ref{sym7c}}$, namely $E_p$. If $p=4$ we would have the other possibility $D_4$, excluded by the fact that the other colors do not form a parabolic subset. This yields
\[
\chi_{E_p} = -\omega_p^{P_{\ref{sym7c}}}=-\varpi_p.
\]
The equalities~(\ref{eq:sphroot}) of Lemma~\ref{lemma:sphroot} yield
\[
\begin{array}{lclclclcl}
\chi_{D_1^+} & = & \chi_{D_1^-} &=& \chi_{D_2^+} &=& \chi_{D_3^+} &=& -\varpi_p,\\
      & &              & &          \chi_{D_2^-} & = & \chi_{D_4} & = & 0.     
\end{array}
\]

For $q=1$ and $p\geq 2$, the subgroup $P_{\ref{sym7c}}$ is found in~\cite{MR1424449} as case 4' of Table C. After loc.cit.\ we have
\[
\Sigma(G/P_{\ref{sym7c}}) = \{ \alpha_1+\alpha_{p+1}, \alpha_{2,p} \}
\]
and
\[
\Delta(G/P_{\ref{sym7c}}) = \{ D_1=D_{p+1}, D_2,E_p \}
\]
with Cartan pairing
\[
\begin{array}{c|cc}
& \alpha_1+\alpha_{p+1} & \alpha_{2,p} \\
\hline
D_1 & 2 & -1 \\
D_2 & -1 & 1 \\ 
E_p & -2 & 1 
\end{array}
\]
There is only one color of $G/P_{\ref{sym7c}}$ mapping not dominantly to $G/Q_{\ref{sym7c}}$, namely $E_p$. For $p=2$ we would have the other possibility $D_2$, excluded by the fact that the other colors do not form a parabolic subset. This yields
\[
\chi_{E_p} = -\omega_p^{P_{\ref{sym7c}}} = -\varpi_p.
\]
The equalities~(\ref{eq:sphroot}) of Lemma~\ref{lemma:sphroot} yield
\[
\begin{array}{lcl}
\chi_{D_1} & = & -\varpi_p,\\
\chi_{D_2} & = & 0.     
\end{array}
\]

\subsubsection{$I=\{\beta_{i}\}$ with $p\geq 3$, $q=1$, and $i\in\{2,\ldots,p-1\}$}\label{sym7eS}
%\textbf{Case \ref{sym7e}:} $I=\{\beta_{i}\}$ with $p\geq 3$, $q=1$, and $i\in\{2,\ldots,p-1\}$.
Using the same argument of the previous two cases, one shows that
\[
\Sigma(G/P_{\ref{sym7e}}) = \{ \sigma=\alpha_1+\alpha_{i+1},\alpha_{2,i}, \gamma=\alpha_{i+1}+2\alpha_{i+2,p+1}+\alpha_{p+1} \}
\]
and
\[
\Delta(G/P_{\ref{sym7e}}) = \{ D_1=D_{i+1},D_2, E_i, D_{i+2} \}
\]
with Cartan pairing
\[
\begin{array}{c|ccc}
& \sigma & \alpha_{2,i} & \gamma \\
\hline
D_1 & 2 & -1 & 0 \\
D_2 & -1 & 1 & 0 \\
E_i & -1 & 1 & -1 \\
D_{i+2} & -1 & 0 & 1
\end{array}
\]

The only color of $G/P_{\ref{sym7e}}$ not mapped dominantly to $G/Q_{\ref{sym7e}}$ is $E_i$. For $i=2$ we would have the other possibility of $D_2$, excluded by the fact that the other colors do not form a parabolic subset. Therefore we have
\[
\chi_{E_i} = -\omega_i^{P_{\ref{sym7e}}}=-\varpi_i.
\]
The equalities~(\ref{eq:sphroot}) of Lemma~\ref{lemma:sphroot} yield
\[
\begin{array}{lclcl}
& & \chi_{D_2} & = & 0, \\
\chi_{D_{i+2}} & = & \chi_{D_1} & = & -\varpi_i.
\end{array}
\]

\subsubsection{$I=\{\beta_{1},\beta_{2}\}$ with $p=2$ and $q\geq 2$}\label{sym7fS}
%\textbf{Case \ref{sym7f}:} $I=\{\beta_{1},\beta_{2}\}$ with $p=2$ and $q\geq 2$.
We claim that
\[
\Sigma(G/P_{\ref{sym7f}}) = \{ \alpha_1,\alpha_2,\alpha_3,\sigma=\alpha_3+2\alpha_{4,q+1}+\alpha_{q+2} \}
\]
and
\[
\Delta(G/P_{\ref{sym7f}}) = \{ D_1^+=D_3^+, D_1^-, D_2^+, D_2^-, D_3^-, D_4 \}
\]
with Cartan pairing
\[
\begin{array}{c|cccc}
& \alpha_1 & \alpha_2 & \alpha_3 & \sigma \\
\hline
D_1^+ & 1 & -1 & 1 & 0  \\
D_1^- & 1 & 0 & -1 & 0 \\
D_2^+ & -1 & 1 & -1 & 0 \\
D_2^- & 0 & 1 & 0 & -1 \\
D_3^- & -1 & 0 & 1 & 0 \\
D_4  & 0 & 0 & -1 & 1 
\end{array}
\]

To show the claim, we define as usual $K_{\ref{sym7f}}$ to be a subgroup of $G$ having the above spherical roots and colors. With the same techniques used before, one shows that we can take $K_{\ref{sym7f}}$ inside $H$, in which case it is a non-maximal proper parabolic subgroup of $H$. Moreover, using the same arguments, one shows that $G/K_{\ref{sym7f}}$ admits $G$-equivariant morphisms to $G/P_{\ref{sym7a}}$ and to $G/P_{\ref{sym7c}}$, but not to $G/P_{\ref{sym7e}}(i)$ for any $i$. This proves the claim.

The colors $D_1^-$, $D_2^+$ are not mapped dominantly to $G/Q$ (the other possibility give by $D_1^-$, $D_2^-$ is excluded because the other colors do not form a parabolic subset).

This gives
\[
\begin{array}{lclcl}
\chi_{D_1^-} & = & -\omega_1^{P} &=&-\varpi_1,\\
\chi_{D_2^+} & = & -\omega_2^{P} &=&-\varpi_2.
\end{array}
\]
The equalities~(\ref{eq:sphroot}) of Lemma~\ref{lemma:sphroot} yield
\[
\begin{array}{lclcl}
& & \chi_{D_1^+} & = & -\varpi_2, \\
& & \chi_{D_3^-} & = & -\varpi_1, \\
\chi_{D_2^-} & = & \chi_{D_4} & = & 0.
\end{array}
\]

\subsubsection{$I=\{\beta_{i},\beta_{j}\}$ with $q=1$ and $1\le i<j\le p$}\label{sym7hS}
%\textbf{Case \ref{sym7h}:} $I=\{\beta_{i},\beta_{j}\}$ with $q=1$ and $1\le i<j\le p$.
%We discuss case~(\ref{sym7h}).
%1
We first assume $1<i<j-1<p-1$, we will discuss separately the other possibilities.

We claim that, with this condition on $i$ and $j$, we have
\[
\Sigma(G/P_{\ref{sym7h}}) = \{ \alpha_1,\sigma = \alpha_{2,i}, \alpha_{i+1}, \gamma=\alpha_{i+2,j}, \alpha_{j+1},
\eta = \alpha_{j+1}+2\alpha_{j+2,p}+\alpha_{p+1} \}
\]
and
\[
\Delta(G/P_{\ref{sym7h}}) = \{ D_1^+=D_{i+1}^+, D_1^-=D_{j+1}^+, D_2, E_i, D_{i+1}^-=D_{j+1}^-, D_{i+2},E_{j}, D_{j+2} \}
\]
with Cartan pairing
\begin{scriptsize}
\[
\begin{array}{c|cccccc}
& \alpha_1 & \sigma &\alpha_{i+1} & \gamma & \alpha_{j+1} & \eta \\
\hline
D_1^+ & 1 & -1 & 1 & 0 & -1 & 0 \\
D_1^- & 1 & 0 & -1 & 0 & 1 & 0 \\
D_2 & -1 & 1 & 0 & 0 & 0 & 0 \\
E_i & 0 & 1 & -1 & 0 & 0 & 0 \\
D_{i+1}^- & -1 & 0 & 1 & -1 & 1 & 0 \\
D_{i+2}  & 0 & 0 & -1 & 1 & 0 & 0 \\
E_{j}  & 0 & 0 & 0 & 1 & -1 & -1 \\
D_{j+2}  & 0 & 0 & 0 & 0 & -1 & 1
\end{array}
\]
\end{scriptsize}

The claim is shown as usual: we denote by $K_{\ref{sym7h}}(i,j)$ a subgroup of $G$ corresponding to the above spherical roots and colors. The standard technique shows that $K_{\ref{sym7h}}(i,j)$ can be taken to be a non-maximal proper parabolic subgroup of $H$, and that it is conjugated in $G$ to a subgroup of $P_{\ref{sym7e}}(k)$ if and only if $k$ is equal to $i$ or $j$. This shows that $K_{\ref{sym7h}}(i,j)$ is conjugated to $P_{\ref{sym7h}}=P_{\ref{sym7h}}(i,j)$.

As usual one checks that $E_i$ and $E_j$ are the colors of $G/P_{\ref{sym7h}}$ not mapped dominantly to $G/Q_{\ref{sym7h}}$, yielding
\[
\begin{array}{lclcl}
\chi_{E_i} & = & -\omega_i^{P_{\ref{sym7h}}} &=& -\varpi_i,\\
\chi_{E_j} & = & -\omega_j^{P_{\ref{sym7h}}} &=& -\varpi_j.
\end{array}
\]
The equalities~(\ref{eq:sphroot}) of Lemma~\ref{lemma:sphroot} yield
\[
\begin{array}{lclcl}
\chi_{D_1^+} & = & \chi_{D_{i+2}} & = & -\varpi_i, \\
\chi_{D_1^-} & = & \chi_{D_{j+2}} & = & -\varpi_j, \\
& & \chi_{D_2} & = & 0, \\
& & \chi_{D_{i+1}^-} & = & -\varpi_i-\varpi_j.
\end{array}
\]
%2
We continue with case~(\ref{sym7h}) and make now the assumption $1=i<j-1<p-1$. We claim that then
\[
\Sigma(G/P_{\ref{sym7h}}) = \{ \alpha_1, \alpha_{2}, \gamma=\alpha_{3,j}, \alpha_{j+1},
\eta = \alpha_{j+1}+2\alpha_{j+2,p}+\alpha_{p+1} \}
\]
and
\[
\Delta(G/P_{\ref{sym7h}}) = \{ D_1^+=D_{j+1}^+, D_1^-, D_2^+, D_{2}^-=D_{j+1}^-, D_{3},E_{j}, D_{j+2} \}
\]
with Cartan pairing
\begin{scriptsize}
\[
\begin{array}{c|ccccc}
& \alpha_1 & \alpha_{2} & \gamma & \alpha_{j+1} & \eta \\
\hline
D_1^+ & 1 & -1 & 0 & 1 & 0 \\
D_1^- & 1 & 0 & 0 & -1 & 0  \\
D_{2}^+ & 0 & 1 & 0 & -1 & 0 \\
D_{2}^- & -1 & 1 & -1 & 1 & 0 \\
D_{3}  & 0 & -1 & 1 & 0 & 0 \\
E_{j}  & 0 & 0 & 1 & -1 & -1 \\
D_{j+2}  & 0 & 0 & 0 & -1 & 1 
\end{array}
\]
\end{scriptsize}

The claim is shown exactly as above, where we have discussed the assumption $1<i<j-1<p-1$. Here $D_1^-$ and $E_j$ are the colors of $G/P_{\ref{sym7h}}$ not mapped dominantly to $G/Q_{\ref{sym7h}}$.

For $j=3$ there is another possibility for these two colors, namely $D_1^-$ and $D_3$. But $P_{\ref{sym7h}}(1,3)$ can be chosen in such a way that $P_{\ref{sym7h}}(1,3)\subseteq P_{\ref{sym7e}}(3)\subseteq Q_{\ref{sym7e}}(3)$, and with $P_{\ref{sym7h}}(1,3)$ containing $P_{\ref{sym7e}}(3)\cap B_-$. In this case we have the commutative diagram
\begin{equation}\label{eq:sp-diag}
\begin{array}{ccc}
G/P_{\ref{sym7h}}(1,3) & \to & G/P_{\ref{sym7e}}(3) \\
\downarrow & & \downarrow \\
G/Q_{\ref{sym7h}}(1,3) &\to & G/Q_{\ref{sym7e}}(3)
\end{array}
\end{equation}
where the morphisms are induced by the above inclusions of subgroups, and it shows that $E_3$ cannot be mapped dominantly to $G/Q_{\ref{sym7h}}(1,3)$ because it is not mapped dominantly to $G/Q_{\ref{sym7e}}(3)$.

So we have
\[
\begin{array}{lclcl}
\chi_{D_1^-} & = & -\omega_1^{P_{\ref{sym7h}}} &=& -\varpi_1,\\
\chi_{E_j} & = & -\omega_j^{P_{\ref{sym7h}}}&=& -\varpi_j.
\end{array}
\]
The equalities~(\ref{eq:sphroot}) of Lemma~\ref{lemma:sphroot} yield
\[
\begin{array}{lclcl}
\chi_{D_1^+} & = & \chi_{D_{j+2}} & = & -\varpi_j, \\
& &  \chi_{D_{3}} & = & -\varpi_1, \\
& & \chi_{D_2^+} & = & 0, \\
& & \chi_{D_{2}^-} & = & -\varpi_1-\varpi_j.
\end{array}
\]
%3
We continue with case~(\ref{sym7h}), with the assumption $1<i=j-1<p-1$. We claim that then
\[
\Sigma(G/P_{\ref{sym7h}}) = \{ \alpha_1, \gamma=\alpha_{2,i}, \alpha_{j},\alpha_{j+1},
\eta = \alpha_{j+1}+2\alpha_{j+2,p}+\alpha_{p+1} \}
\]
and
\[
\Delta(G/P_{\ref{sym7h}}) = \{ D_1^+=D_{j+1}^+, D_1^-=D_j^-, D_2,E_{i}, D_{j}^+,D_{j+1}^-,D_{j+2} \}
\]
with Cartan pairing
\begin{scriptsize}
\[
\begin{array}{c|ccccc}
& \alpha_1 & \gamma & \alpha_{j} & \alpha_{j+1} & \eta \\
\hline
D_1^+ & 1 & 0 & -1 & 1 & 0 \\
D_1^- & 1 & -1 & 1 & -1 & 0  \\
D_{2} & -1 & 1 & 0 & 0 & 0 \\
E_i & 0 & 1 & -1 & 0 & 0 \\
D_{j}^+  & -1 & 0 & 1 & 0 & -1 \\
D_{j+1}^-  & -1 & 0 & 0 & 1 & 0 \\
D_{j+2}  & 0 & 0 & 0 & -1 & 1 
\end{array}
\]
\end{scriptsize}

The claim is shown as above. Here $E_i$ and $D_j^+$ are the colors of $G/P_{\ref{sym7h}}$ not mapped dominantly to $G/Q_{\ref{sym7h}}$.

For $i=2$ there is another possibility for these two colors, namely $D_2$ and $D_j^+$, excluded by the fact that the other colors do not form a parabolic subset.

So we have
\[
\begin{array}{lclcl}
\chi_{E_i} & = & -\omega_i^{P_{\ref{sym7h}}} &=& -\varpi_i,\\
\chi_{D_j^+} & = & -\omega_j^{P_{\ref{sym7h}}}&=& -\varpi_j.
\end{array}
\]
The equalities~(\ref{eq:sphroot}) of Lemma~\ref{lemma:sphroot} yield
\[
\begin{array}{lclcl}
\chi_{D_1^+} & = & \chi_{D_{j+2}} & = & -\varpi_j, \\
\chi_{D_{1}^-} & = & \chi_{D_{j+1}^-} & = & -\varpi_i, \\
& & \chi_{D_2^+} & = & 0.
\end{array}
\]
%4
We continue with case~(\ref{sym7h}), with the assumption $1<i<j-1=p-1$. We claim that then
\[
\Sigma(G/P_{\ref{sym7h}}) = \{ \alpha_1, \sigma=\alpha_{2,i}, \alpha_{i+1},\gamma=\alpha_{i+2,j},\alpha_{p+1} \}
\]
and
\[
\Delta(G/P_{\ref{sym7h}}) = \{ D_1^+=D_{i+1}^+, D_1^-=D_{p+1}^+, D_2,E_{i}, D_{i+1}^-=D_{p+1}^-,D_{i+2},E_p \}
\]
with Cartan pairing
\begin{scriptsize}
\[
\begin{array}{c|ccccc}
& \alpha_1 & \sigma & \alpha_{i+1} & \gamma & \alpha_{p+1} \\
\hline
D_1^+ & 1 & -1 & 1 & 0 & -1 \\
D_1^- & 1 & 0 & -1 & 0 & 1 \\
D_{2} & -1 & 1 & 0 & 0 & 0 \\
E_i & 0 & 1 & -1 & 0 & 0\\
D_{i+1}^-  & -1 & 0 & 1 & -1 & 1 \\
D_{i+2}  & 0 & 0 & -1 & 1 & 0 \\
E_p  & 0 & 0 & 0 & 1 & -2 
\end{array}
\]
\end{scriptsize}

The claim is shown as above. With the same techniques as above, one shows that $E_i$ and $E_p$ are the colors of $G/P_{\ref{sym7h}}$ not mapped dominantly to $G/Q_{\ref{sym7h}}$. 

So we have
\[
\begin{array}{lclcl}
\chi_{E_i} & = & -\omega_i^{P_{\ref{sym7h}}} &=& -\varpi_i,\\
\chi_{E_p} & = & -\omega_p^{P_{\ref{sym7h}}} &=& -\varpi_p.
\end{array}
\]
The equalities~(\ref{eq:sphroot}) of Lemma~\ref{lemma:sphroot} yield
\[
\begin{array}{lclcl}
\chi_{D_1^+} & = & \chi_{D_{i+2}} & = & -\varpi_i, \\
& & \chi_{D_{1}^-} & = & -\varpi_p, \\
& & \chi_{D_2} & = & 0,\\
& & \chi_{D_{i+1}^-} & = & -\varpi_i-\varpi_p.
\end{array}
\]
%5
We continue with case~(\ref{sym7h}), with the assumption $1=i=j-1<p-1$. We claim that then
\[
\Sigma(G/P) = \{ \alpha_1,\alpha_2,\alpha_3,\sigma=\alpha_3+2\alpha_{4,p+q-1}+\alpha_{p+q} \}
\]
and
\[
\Delta(G/P) = \{ D_1^+=D_3^+, D_1^-, D_2^+, D_2^-, D_3^-, D_4 \}
\]
with Cartan pairing
\[
\begin{array}{c|cccc}
& \alpha_1 & \alpha_2 & \alpha_3 & \sigma \\
\hline
D_1^+ & 1 & -1 & 1 & 0  \\
D_1^- & 1 & 0 & -1 & 0  \\
D_2^+ & 0 & 1 & -1 & 0  \\
D_2^- & -1 & 1 & 0 & -1 \\
D_3^- & -1 & 0 & 1 & 0  \\
D_4  & 0 & 0 & -1 & 1 
\end{array}
\]
The claim is shown as above. With the same techniques as above, one shows that $D_1^-$ and $D_2^-$ are the colors of $G/P_{\ref{sym7h}}$ not mapped dominantly to $G/Q_{\ref{sym7h}}$.

So we have
\[
\begin{array}{lclcl}
\chi_{D_1^-} & = & -\omega_1^{P_{\ref{sym7h}}} &=& -\varpi_1,\\
\chi_{D_2^-} & = & -\omega_2^{P_{\ref{sym7h}}} &=& -\varpi_2.
\end{array}
\]
The equalities~(\ref{eq:sphroot}) of Lemma~\ref{lemma:sphroot} yield
\[
\begin{array}{lclcl}
\chi_{D_1^+} & = & \chi_{D_4} & = & -\varpi_2, \\
& & \chi_{D_2^+} & = & 0,\\
& & \chi_{D_{3}^-} & = & -\varpi_1.
\end{array}
\]
%6
We continue with case~(\ref{sym7h}), with the assumption $1<i=j-1=p-1$. We claim that then
\[
\Sigma(G/P) = \{ \alpha_1,\sigma=\alpha_{2,p-1},\alpha_p,\alpha_{p+1} \}
\]
and
\[
\Delta(G/P) = \{ D_1^+=D_p^+, D_1^-=D_{p+1}^-, D_2, E_{p-1}, D_p^-, D_{p+1}^+ \}
\]
with Cartan pairing
\[
\begin{array}{c|cccc}
& \alpha_1 & \sigma & \alpha_p & \alpha_{p+1} \\
\hline
D_1^+ & 1 & -1 & 1 & -1  \\
D_1^- & 1 & 0 & -1 & 1  \\
D_2 & -1 & 1 & 0 & 0  \\
E_{p-1} & 0 & 1 & -1 & 0 \\
D_p^- & -1 & 0 & 1 & -1  \\
D_{p+1}^+ & -1 & 0 & 0 & 1 
\end{array}
\]
The claim is shown as above. With the same techniques as above, one shows that $E_{p-1}$ and $D_p^-$ are the colors of $G/P_{\ref{sym7h}}$ not mapped dominantly to $G/Q_{\ref{sym7h}}$.

So we have
\[
\begin{array}{lclcl}
\chi_{E_{p-1}} & = & -\omega_{p-1}^{P_{\ref{sym7h}}} &=& -\varpi_{p-1},\\[5pt]
\chi_{D_p^-} & = & -\omega_p^{P_{\ref{sym7h}}}&=& -\varpi_{p}.
\end{array}
\]
The equalities~(\ref{eq:sphroot}) of Lemma~\ref{lemma:sphroot} yield
\[
\begin{array}{lclcl}
\chi_{D_1^+} & = & \chi_{D_{p+1}^+} & = & -\varpi_{p-1}, \\
& & \chi_{D_2} & = & 0,\\
& & \chi_{D_{1}^-} & = & -\varpi_p.
\end{array}
\]
%7
We finish case~(\ref{sym7h}), discussing the last remaining assumption $1=i=j-1=p-1$. We claim that then
\[
\Sigma(G/P) = S
\]
and
\[
\Delta(G/P) = \{ D_1^+=D_3^+, D_1^-, D_2^+, D_2^-, D_3^- \}
\]
with Cartan pairing
\[
\begin{array}{c|ccc}
& \alpha_1 & \alpha_2 & \alpha_{3} \\
\hline
D_1^+ & 1 & -1 & 1  \\
D_1^- & 1 & 0 & -1  \\
D_2^+ & -1 & 1 & -1  \\
D_2^- & 0 & 1 & -1 \\
D_3^- & -1 & 0 & 1
\end{array}
\]

The claim is shown as above. With the same techniques as above, one shows that $D_1^-$ and $D_2^+$ are the colors of $G/P_{\ref{sym7h}}$ not mapped dominantly to $G/Q_{\ref{sym7h}}$.

So we have
\[
\begin{array}{lclcl}
\chi_{D_1^-} & = & -\omega_1^{P_{\ref{sym7h}}} &=& -\varpi_1,\\
\chi_{D_2^+} & = & -\omega_2^{P_{\ref{sym7h}}} &=& -\varpi_2.
\end{array}
\]
The equalities~(\ref{eq:sphroot}) of Lemma~\ref{lemma:sphroot} yield
\[
\begin{array}{lcl}
\chi_{D_1^+} & = & -\varpi_2, \\
\chi_{D_2^-} & = & 0,\\
\chi_{D_3^-} & = & -\varpi_1.
\end{array}
\]

\subsubsection{$I=\{\beta_{1},\beta'_{i}\}$ for any $i\in \{1,\ldots,q\}$, with $p=1$}\label{sym7iS}
%\textbf{Case \ref{sym7i}:} $I=\{\beta_{1},\beta'_{i}\}$ for any $i\in \{1,\ldots,q\}$, with $p=1$.
%1
Let us first suppose $1<i<q$, we will deal later with the other possibilities.
We claim that, with this condition on $i$, we have
\[
\Sigma(G/P_{\ref{sym7i}}) = \{ \alpha_1,\sigma = \alpha_{2,i}, \alpha_{i+1}, \gamma=\alpha_{i+1}+2\alpha_{i+2,q}+\alpha_{q+1} \}
\]
and
\[
\Delta(G/P_{\ref{sym7i}}) = \{ D_1^+=D_{i+1}^+, D_1^-, D_2, E_i, D_{i+1}^-, D_{i+2} \}
\]
with Cartan pairing
\[
\begin{array}{c|cccc}
& \alpha_1 & \sigma &\alpha_{i+1} & \gamma \\
\hline
D_1^+ & 1 & -1 & 1 & 0 \\
D_1^- & 1 & 0 & -1 & 0 \\
D_2   & -1 & 1 & 0 & 0 \\
E_i   & 0 & 1 & -1 & -1 \\
D_{i+1}^- & -1 & 0 & 1 & 0 \\
D_{i+2}   & 0 & 0 & -1 & 1 \\
\end{array}
\]

The claim is shown as for the previous cases. With the same techniques as above, one shows that $D_1^-$ and $D_{i+1}^-$ are the colors of $G/P_{\ref{sym7i}}$ not mapped dominantly to $G/Q_{\ref{sym7i}}$.

So we have
\[
\begin{array}{lclcl}
\chi_{D_1^-} & = & -\omega_1^{P_{\ref{sym7i}}} &=& -\varpi_1,\\
\chi_{D_{i+1}^-} & = & -\omega_{i+1}^{P_{\ref{sym7i}}} &=& -\varpi'_{i}-\varpi_1.
\end{array}
\]
The equalities~(\ref{eq:sphroot}) of Lemma~\ref{lemma:sphroot} yield
\[
\begin{array}{lclclcl}
\chi_{D_1^+} & = & \chi_{E_i} & = & \chi_{D_{i+2}} & = & -\varpi'_{i}, \\
& & & & \chi_{D_2} & = & 0.
\end{array}
\]
%2
We continue with case~(\ref{sym7i}), with the assumption $1=i<q$. We claim that then
\[
\Sigma(G/P_{\ref{sym7i}}) = \{ \alpha_1, \alpha_{2}, \gamma=\alpha_{2}+2\alpha_{3,q}+\alpha_{q+1} \}
\]
and
\[
\Delta(G/P_{\ref{sym7i}}) = \{ D_1^+, D_1^-, D_2^+, D_2^- D_3 \}
\]
with Cartan pairing
\[
\begin{array}{c|ccc}
& \alpha_1 & \alpha_2 & \gamma \\
\hline
D_1^+ & 1 & -1 & 0 \\
D_1^- & 1 & 0 & -1 \\
D_2^+  & 0 & 1 & 0 \\
D_2^-  & -1 & 1 & 0 \\
D_3 & 0 & -1 & 1
\end{array}
\]
The claim is shown as for the previous cases. We can take $Q_{\ref{sym7i}}$ so that $\alpha_1$ and $\alpha_2$ are the only simple roots that are not simple roots of $L_{Q_{\ref{sym7i}}}$. With the same techniques as above, one shows that $D_1^+$ and $D_{2}^-$ are the colors of $G/P_{\ref{sym7i}}$ not mapped dominantly to $G/Q_{\ref{sym7i}}$. To exclude other possibilities, one uses the inclusions $P_{\ref{sym7i}}\subseteq P_{\ref{sym7a}}\subseteq Q_{\ref{sym7a}}$ and the corresponding diagram similar to~(\ref{eq:sp-diag}).

So we have
\[
\begin{array}{lclcl}
\chi_{D_1^+} & = & -\omega_1^{P_{\ref{sym7i}}} &=& -\varpi_1,\\
\chi_{D_{2}^-} & = & -\omega_{2}^{P_{\ref{sym7i}}}&=& -\varpi'_1-\varpi_1.
\end{array}
\]
The equalities~(\ref{eq:sphroot}) of Lemma~\ref{lemma:sphroot} yield
\[
\begin{array}{lclcl}
\chi_{D_1^-} & = & \chi_{D_3} & = & -\varpi'_1, \\
 & & \chi_{D_2^+} & = & 0.
\end{array}
\]
%3
We continue with case~(\ref{sym7i}), with the assumption $1<i=q$. We claim that then
\[
\Sigma(G/P_{\ref{sym7i}}) = \{ \alpha_1, \sigma=\alpha_{2,q}, \alpha_{q+1} \}
\]
and
\[
\Delta(G/P_{\ref{sym7i}}) = \{ D_1^+=D_{p+1}^+, D_1^-, D_2, E_p, D_{q+1}^- \}
\]
with Cartan pairing
\[
\begin{array}{c|ccc}
& \alpha_1 & \sigma & \alpha_{q+1} \\
\hline
D_1^+ & 1 & -1 & 1 \\
D_1^- & 1 & 0 & -1 \\
D_2  & -1 & 1 & 0 \\
E_p  & 0 & 1 & -2 \\
D_{p+1}^- & -1 & 0 & 1
\end{array}
\]
The claim is shown as for the previous cases. With the same techniques as above, one shows that $D_1^-$ and $D_{q+1}^-$ are the colors of $G/P_{\ref{sym7i}}$ not mapped dominantly to $G/Q_{\ref{sym7i}}$.

So we have
\[
\begin{array}{lclcl}
\chi_{D_1^-} & = & -\omega_1^{P_{\ref{sym7i}}} &=& -\varpi_1,\\
\chi_{D_{q+1}^-} & = & -\omega_{q+1}^{P_{\ref{sym7i}}} &=& -\varpi'_q-\varpi_1.
\end{array}
\]
The equalities~(\ref{eq:sphroot}) of Lemma~\ref{lemma:sphroot} yield
\[
\begin{array}{lclcl}
\chi_{D_1^+} & = & \chi_{E_p} & = & -\varpi'_{q}, \\
 & & \chi_{D_2} & = & 0.
\end{array}
\]
%4
We finish case~(\ref{sym7i}), discussing the assumption $1=i=q$. The subgroup $P_{\ref{sym7i}}$ of $\Sp(4)$ appears then as the first occurrence of case~8 of \cite[Table~C]{MR1424449}, which gives
\[
\Sigma(G/P_{\ref{sym7i}}) = S
\]
and
\[
\Delta(G/P_{\ref{sym7i}}) = \{ D_1^+, D_1^-, D_2^+, D_2^- \}
\]
with Cartan pairing
\[
\begin{array}{c|cc}
& \alpha_1 & \alpha_2 \\
\hline
D_1^+ & 1 & -1 \\
D_1^- & 1 & -1 \\
D_2^+ & 0 & 1 \\
D_2^- & -1 & 1
\end{array}
\]
Here there are two possible choices of colors of $G/P_{\ref{sym7i}}$ not mapped dominantly to $G/Q_{\ref{sym7i}}$, namely $\{D_1^+, D_{2}^-\}$ and $\{D_1^-,D_{2}^-\}$. They are exchanged by the non-trivial $G$-equivariant automorphism of $G/P_{\ref{sym7i}}$, so (as one can easily check directly) they produce the same computation for the extended weight monoid, except for the labelling of the generators by the colors. 

Let us choose $\{D_1^+, D_{2}^-\}$, which yields
\[
\begin{array}{lclcl}
\chi_{D_1^+} & = & -\omega_1^{P_{\ref{sym7i}}} &=& -\varpi_1,\\
\chi_{D_{2}^-} & = & -\omega_{2}^{P_{\ref{sym7i}}} &=& -\varpi'_1-\varpi_1.
\end{array}
\]
The equalities~(\ref{eq:sphroot}) of Lemma~\ref{lemma:sphroot} yield
\[
\begin{array}{lcl}
\chi_{D_1^-} & = & -\varpi'_1, \\
\chi_{D_2^+} & = & 0.
\end{array}
\]

% 8.8

\subsection{$\mathsf F_4/\Spin(9)$}
Let $G$ denote the connected semisimple group of Dynkin type $\mathsf{F}_{4}$. The subgroup $H\subseteq\mathsf{F}_{4}$ has maximal torus $T$ and simple roots $\beta_{1}=\eps_{1}-\eps_{2}=2\alpha_{4}+\alpha_{2}+2\alpha_{3}, \beta_{2}=\alpha_{1},\beta_{3}=\alpha_{2}$ and $\beta_{4}=\alpha_{3}$ and is isomorphic to $\Spin(9)$. With this choice, we have $\omega_1=\varpi_2$, $\omega_2=\varpi_1+\varpi_3$, $\omega_3=\varpi_1+\varpi_4$, $\omega_4=\varpi_1$.

Assuming $I$ is minimal, we have the following possibilities:
\begin{enumerate}
\item\label{sym8a} $I=\{\beta_{1},\beta_{2}\}$,
\item\label{sym8b} $I=\{\beta_{3}\}$,
\item\label{sym8c} $I=\{\beta_{4}\}$.
\end{enumerate}

Note that case (\ref{sym8a}) gives rise to the two subcases $I=\{\beta_{1}\}$ and $I=\{\beta_{2}\}$. Let us denote the corresponding parabolic subgroups by $P_{12},P_{1},P_{2},P_{3}$ and $P_{4}$. These subgroups turn out to be wonderful and their Luna diagrams can be found in \cite[\S3.2]{MR2769314}. Additionally, one can check that the spherical roots and the colors given below are correct as usual: applying Corollary~\ref{cor:morphisms} and Proposition~\ref{prop:PinH} to check that they correspond to minimal parabolic subgroups of $H$ that are spherical in $G$, and using dimensions to check that the given data correspond to the correct parabolic subgroup.

\subsubsection{$I=\{\beta_{1},\beta_{2}\}$}\label{sym8aS}
%\textbf{Case \ref{sym8a}:} $I=\{\beta_{1},\beta_{2}\}$. 
We have
$$\Sigma(G/P_{12})=\{\sigma_{1}=\alpha_{1}+\alpha_{2}, \sigma_{2}=\alpha_{2}+\alpha_{3},\sigma_{3}=\alpha_{3},\sigma_{4}=\alpha_{4}\}$$
and
$$\Delta(G/P_{12})=\{D_{1},D_{2},D_{3}^{\pm},D_{4}^{\pm}\}$$
with Cartan pairing
$$\begin{array}{r|rrrr}&\sigma_1&\sigma_2&\sigma_{3}&\sigma_{4}\\
\hline
D_{1}&1&-1&0&0\\
D_{2}&1&1&-1&0\\
D_{3}^{+}&0&0&1&-1\\
D_{3}^{-}&-2&0&1&0\\
D_{4}^{+}&0&0&-1&1\\
D_{4}^{-}&0&-1&0&1
\end{array}
$$
Let $Q_{12}$ be the parabolic subgroup so that $\alpha_{2}$ and $\alpha_{3}$ are the simple roots of $L_{Q_{12}}$. Then $Q_{12}$ is minimal for containing $P_{12}$. The pull-backs of the two colors of $G/Q_{12}$ to $G/P_{12}$ are $D_{1}$ and one of $\{D_{4}^{\pm}\}$. Since $\{D_{2},D_{3}^{\pm},D_{4}^{-}\}$ is parabolic and $\{D_{2},D_{3}^{\pm},D_{4}^{+}\}$ is not, it follows that $D_{4}^{+}$ is pulled back from $G/Q_{12}$. Hence
\[
\begin{array}{lclcl}
\chi_{D_1} & = & -\omega_1^{P_{12}} &=& -\varpi_2,\\
\chi_{D_{4}^-} & = & -\omega_{4}^{P_{12}} &=& -\varpi_1.
\end{array}
\]
We recover as usual the other weights:
\[
\begin{array}{lcl}
\chi_{D_{2}}&=&-\varpi_{2},\\
\chi_{D_{3}^{+}}&=&-\varpi_{1},\\
\chi_{D_{3}^{-}}&=&-\varpi_{2},\\
\chi_{D_{4}^{-}}&=&0.
\end{array}
\]

\textbf{Subcase:} $I=\{\beta_{1}\}$. Here we have $\Sigma(G/P_{1})=\{\sigma_{1}=\alpha_{1}+2\alpha_{2}+3\alpha_{3},\sigma_{2}=\alpha_{4}\}$ and $\Delta(G/P_{1})=\{D_{3},D_{4}^{\pm}\}$. The inverse images of the colors $D_{3},D_{4}^{+},D_{4}^{-}$ of $G/P_{1}$ under the natural projection $G/P_{12}\to G/P_{1}$ are the colors $D_{3}^{+},D_{4}^{-}$ and $D_{4}^{+}$ of $G/P_{12}$ respectively.

\textbf{Subcase:} $I=\{\beta_{2}\}$. Here we have $\Sigma(G/P_{1})=\{\sigma_{1}=\alpha_{1}+\alpha_{2}, \sigma_{2}=\alpha_{2}+\alpha_{3},\sigma_{3}=\alpha_{3}+\alpha_{4}\}$ and $\Delta(G/P_{1})=\{D_{1},D_{2},D_{3},D_{4}\}$. The inverse images of the colors $D_{1},D_{2},D_{3},D_{4}$ of $G/P_{2}$ under the natural projection $G/P_{12}\to G/P_{2}$ are the colors $D_{1},D_{2}, D_{3}^{-}$ and $D_{4}^{-}$ of $G/P_{12}$ respectively.

\subsubsection{$I=\{\beta_{3}\}$}\label{sym8bS}
%\textbf{Case \ref{sym8b}:} $I=\{\beta_{3}\}$. 
We have
$$\Sigma(G/P_{3})=\{\sigma_{1}=\alpha_{1},\sigma_{2}=\alpha_{2}+\alpha_{3},\sigma_{3}=\alpha_{3}, \sigma_{4}=\alpha_{4}\}$$
and
$$\Delta(G/P_{3})=\{D_{1}^{+}=D_{4}^{+}, D_{1}^{-}=D_{3}^{-},D_{2}, D_{3}^{+}, D_{4}^{-}\}$$
with Cartan pairing
$$\begin{array}{r|rrrr}&\sigma_1&\sigma_2&\sigma_{3}&\sigma_{4}\\
\hline
D_{1}^{+}&1&-1&-1&1\\
D_{1}^{-}&1&0&1&-1\\
D_{2}&-1&1&-1&0\\
D_{3}^{+}&-1&0&1&0\\
D_{4}^{-}&-1&0&0&1
\end{array}
$$

A parabolic subset must contain an element of each of the sets $\{D_{1}^{+},D_{1}^{-}\},\{D_{2}\},\{D_{1}^{-},D_{3}\}$ and $\{D_{1}^{+},D_{4}^{-}\}$. Moreover, the colors that are moved by more than one root have to be in the parabolic subset, since colors on a (partial) flag variety are moved by at most one root. It follows that a parabolic subset  must contain one of the sets $\{D_{1}^{+},D_{1}^{-},D_{2}\}$ as a subset. Suppose that $a_{1}^{+}D_{1}^{+}+a_{1}^{-}D_{1}^{-}+a_{2}D_{2}$ is strictly positive on every spherical root. Then $a_{1}^{+}+a_{2}<a_{1}^{-}$ and $a_{1}^{-}<a_{1}^{+}$, which implies $a_{2}<0$, and we conclude that $\{D_{1}^{+},D_{1}^{-},D_{2}\}$ itself is not parabolic.

Suppose that $a_{1}^{+}D_{1}^{+}+a_{1}^{-}D_{1}^{-}+a_{2}D_{2}+a_{3}^{+}D_{3}^{+}$ is strictly positive on all spherical roots. Then $a_{1}^{+}+a_{1}^{-}-a_{2}-a_{3}^{+}>0$, $a_{1}^{+}<a_{2}$, $-a_{1}^{+}+a_{1}^{-}-a_{2}+a_{3}^{+}>0$ and $a_{1}^{+}>a_{1}^{-}$, from which $a_{1}^{-}>a_{2}$ and $a_{1}^{-}<a_{2}$, a contradiction.

Note that $4D_{1}^{+}+14D_{1}^{-}+6D_{2}+11D_{4}^{-}$ is strictly positive on each spherical root. We conclude that
$\{D_{1}^{+},D_{1}^{-},D_{2},D_{4}^{-}\}$ is the minimal parabolic subset of $\Delta(G/P_{3})$. 

First we choose another system of positive roots for $G$: $\alpha_{1}'=\eps_{2}-\eps_{3},\alpha_{2}'=\eps_{1}-\eps_{2},\alpha_{3}'=-\frac{1}{2}(\eps_{1}-\eps_{2}-\eps_{3}+\eps_{4}), \alpha_{4}'=\eps_{4}$. Note that  $\eps_{3}-\eps_{4}=2\alpha_{3}'+\alpha_{2}'$, which implies that this choice of positive roots is compatible with the standard choice of positive roots of $\laso_{9}$. Also note that $s_{\eps_{4}}\circ s_{\frac{1}{2}(-\eps_{1}+\eps_{2}+\eps_{3}-\eps_{4})}(\alpha_{i})=\alpha_{i}'$ for $i=1,2,3,4$.

It is clear that $P_{3}\subseteq Q_{\{\alpha_{1}',\alpha_{2}',\alpha_{4}'\}}$ and that $Q_{\{\alpha_{1}',\alpha_{2}',\alpha_{4}'\}}$ is minimal for this inclusion. Hence we found the conjugate $\widetilde{P}_{3}$ that is contained in $Q_{3}$. It follows that the $P$-weight of $D_{3}^{+}$ is $(-\omega_{3}')^{P}=-\varpi_{3}$.
Hence
\[
\chi_{D_{3}^{+}}=-\varpi_{3}
\]
and the equations
\begin{eqnarray*}
\chi_{D_{1}^{+}}+\chi_{D_{1}^{-}}-\chi_{D_{2}}+\chi_{D_{3}^{+}}-\chi_{D_{4}^{-}}=0\nonumber\\
-\chi_{D_{1}^{+}}+\chi_{D_{2}}=0\nonumber\\
-\chi_{D_{1}^{+}}+\chi_{D_{1}^{-}}-\chi_{D_{2}}+\chi_{D_{3}^{+}}=0\nonumber\\
\chi_{D_{1}^{+}}-\chi_{D_{1}^{-}}+\chi_{D_{4}^{-}}=0
\end{eqnarray*}
yield
\[
\begin{array}{lclclcl}
\chi_{D_{1}^{+}}&=&\chi_{D_{1}^{-}}&=&\chi_{D_{2}}&=&-\varpi_{3},\\
 & & & & \chi_{D_{4}^{-}} &=& 0.
\end{array}
\]

% Case 3
\subsubsection{$I=\{\beta_{4}\}$}\label{sym8cS}
%\textbf{Case \ref{sym8c}:} $I=\{\beta_{4}\}$.
We have
$$\Sigma(G/P_{4})=\{\sigma_{1}=\alpha_{1}+\alpha_{2}+\alpha_{3}, \sigma_{2}=\alpha_{2}+2\alpha_{3}+\alpha_{4},\sigma_{3}=\alpha_{4}\}$$
and
$$\Delta(G/P_{3})=\{D_{1},D_{3},D_{4}^{+},D_{4}^{-}\}$$
with Cartan pairing
$$\begin{array}{r|rrr}&\sigma_1&\sigma_2&\sigma_{3}\\
\hline
D_{1}&1&-1&0\\
D_{3}&0&1&-1\\
D_{4}^{+}&0&0&1\\
D_{4}^{-}&-1&0&1
\end{array}
$$

A parabolic subset must contain $D_{1},D_{3}$ and an element of $\{D_{4}^{+},D_{4}^{-}\}$. Suppose that $a_{1}D_{1}+a_{3}D_{3}+a_{4}^{-}D_{4}^{-}$ is strictly positive on the spherical roots. This implies $a_{1}>a_{4}^{-}$, $a_{1}<a_{3}$ and $a_{3}<a_{4}^{-}$, from which $a_{4}^{-}<a_{4}^{-}$, a contradiction.

Since $D_{1}+2D_{3}+3D_{4}^{+}$ is strictly positive on each spherical root, the minimal parabolic subset is given by $\{D_{1},D_{3},D_{4}^{+}\}$. Let $Q_{4}$ denote the parabolic subgroup such that only $\alpha_{4}$ is not a simple root of $L_{Q_{4}}$. Then $Q_{4}$ is minimal for containing a $G$-conjugate $\widetilde{P}_{4}$ of $P_{4}$. To find this conjugate we argue as follows.

First we look at a conjugate of $B_{4}$ which has the following system of positive roots: $\beta_{1}'=\eps_{3}+\eps_{4},\beta_{2}=\eps_{2}-\eps_{3},\beta_{3}'=\eps_{3}-\eps_{4},\beta_{4}'=\frac{1}{2}(\eps_{1}-\eps_{2}-\eps_{3}+\eps_{4})$. Note that this system is compatible with our choice of positive system of $\mathfrak{f}_{4}$. We claim that $P_{\{\beta_{1}',\beta_{2}',\beta_{3}'\}}\subseteq Q_{4}$ regularly. The Levi of $P$ has Lie algebra $\laso_{6}+\bbC$ which is embedded in the Lie algebra of the Levi of $Q$, which is $\laso_{7}+\bbC$. A weight of the nilpotent radical of $P_{\{\beta_{1}',\beta_{2}',\beta_{3}'\}}$ has the short root $\beta_{4}'$ in its support. Hence it has $\alpha_{4}$ in its support, which implies that it is a weight of the nilpotent radical of $Q$. Since $(\beta_{4}')^{\vee}=\alpha_{4}^{\vee}+\alpha_{3}^{\vee}$, we see that $\omega_{4}((\beta_{i}')^{\vee})=\delta_{4i}$. Hence $\omega_{4}^{P}=\varpi_{4}'$. It follows that 
\[
\chi_{D_{4}^{-}}=-\varpi_{4}.
\]

The equations
$$\chi_{D_{1}}-\chi_{D_{4}^{-}}=0,\quad -\chi_{D_{1}}+\chi_{D_{2}}=0,\quad -\chi_{D_{3}}+\chi_{D_{4}^{+}}+\chi_{D_{4}^{-}}=0$$
imply
\[
\begin{array}{lclcl}
\chi_{D_{1}}&=&\chi_{D_{3}}&=&-\varpi_{4},\\
& & \chi_{D_{4}^{+}}& =& 0.
\end{array}
\]

% 8.9

\subsection{$\mathsf E_6/(\Spin(10)\times \bbC^{\times})$}\label{sym8S} Let $G$ denote the connected and simply connected semisimple group of Dynkin type $\mathsf{E}_{6}$. Let $H\subseteq G$ be the Levi component with simple roots
$$\beta_{1}=\alpha_{6},\beta_{2}=\alpha_{5},\beta_{3}=\alpha_{4},\beta_{4}=\alpha_{3},\beta_{5}=\alpha_{2}.$$
Then $H$ is isogenous to $\bbC^{\times}\times\Spin(10)$ and we have
\[
\begin{array}{lcl}
\varpi_1&=& -\frac12\omega_1+\omega_6,\\[5pt]
\varpi_2&=& -\omega_1+\omega_5,\\[5pt]
\varpi_3&=& -\frac32\omega_1+\omega_4,\\[5pt]
\varpi_4&=& -\frac54\omega_1+\omega_3,\\[5pt]
\varpi_5&=& -\frac34\omega_1+\omega_2.\\[5pt]
\end{array}
\]
We denote by $\epsilon$ the extension of a generator of $\Chi(\CC^\times)$ to a character of the universal cover of $H$, chosen in such a way that $\omega_1=4\epsilon$, and we have
\[
\begin{array}{lcl}
\omega_1&=& 4\epsilon,\\
\omega_2&=&\varpi_5+3\epsilon,\\
\omega_3&=&\varpi_4+5\epsilon,\\
\omega_4&=&\varpi_3+6\epsilon,\\
\omega_5&=&\varpi_2+4\epsilon,\\
\omega_6&=&\varpi_1+2\epsilon.
\end{array}
\]

There is one possibility for $I$, namely $I=\{\beta_{1}\}$. Let $P\subseteq H$ denote the corresponding parabolic subgroup, i.e.~$P=P_{\{\beta_{2},\beta_{3},\beta_{4},\beta_{5}\}}$. Then $L_{P}$ is isogenous to $\bbC^{\times}\times\Spin(8)\times\bbC^{\times}$. Consider the Luna diagram for that appears as number 42 in \cite{MR2346359} and denote its corresponding subgroup with $P'\subseteq G$. We have
$$\Delta(G/P')=\{D_{1}^{+},D_{1}^{-},D_{3},D_{2},D_{5},D_{6}^{+},D_{6}^{-}\}$$
and%
\footnote{Here we denote by $\alpha_{ijk}$ the sum $\alpha_i+\alpha_j+\alpha_k$. Similar notations will be used throughout this case.}
$$\Sigma(G/P')=\{\alpha_{1},\alpha_{234},\alpha_{345},\alpha_{245},\alpha_{6}\}$$
with Cartan pairing
$$\begin{array}{r|rrrrr}&\sigma_1&\sigma_2&\sigma_{3}&\sigma_{4}&\sigma_{5}\\
\hline
D_{1}^{+}&1&0&-1&0&0\\
D_{1}^{-}&1&-1&0&0&0\\
D_{3}&-1&1&1&-1&0\\
D_{2}&0&1&-1&1&0\\
D_{5}&0&-1&1&1&-1\\
D_{6}^{+}&0&0&-1&0&1\\
D_{6}^{-}&0&0&0&-1&1
\end{array}
$$
Note that $\Sigma(G/H)=\{\alpha_{13456},2\alpha_{2}+2\alpha_{4}+\alpha_{3}+\alpha_{5}\}$, see e.g.~\cite[Case 18]{MR3345829}. The subset $\{D_{1}^{+},D_{3},D_{5},D_{6}^{+}\}\subseteq\Sigma(G/P')$ is distinguished and it gives a map to $G/H$.
It follows that $P'$ is conjugate to a subgroup of $H$. An application of Proposition \ref{prop:PinH} readily implies that $P'$ is conjugate to a parabolic subgroup of $H$. It follows from the classification that $P$ and $P'$ are conjugates.

Let $Q\subseteq G$ be the parabolic subgroup for which $L_{Q}$ has simple roots $\{\alpha_{1},\alpha_{6}\}$. Then $\{D_{1}^{+},D_{3},D_{2},D_{5},D_{6}^{-}\}\subseteq\Delta(G/P)$ is a distinguished subset that gives a map $G/P\to G/Q$.
This means that $D_{1}^{-}$ and $D_{6}^{+}$ come from $G/Q$. Hence
\[
\begin{array}{lclcl}
\chi_{D_{1}^{-}}&=&-\omega_{1}^{P}&=&-4\epsilon,\\
\chi_{D_{6}^{+}}&=&-\omega_{6}^{P}&=&-\varpi_1-2\epsilon.
\end{array}
\]

We find
\[
\begin{array}{lclcl}
\chi_{D_{1}^{+}}&=&\chi_{D_{5}}&=&-\varpi_{1}+2\epsilon,\\
& & \chi_{D_{2}}&=&0,\\
& & \chi_{D_{3}}&=&-\varpi_{1}-2\epsilon,\\
& & \chi_{D_{6}^{-}}&=&4\epsilon
\end{array}
\]

%    8.10

\subsection{$\mathsf E_6/\mathsf F_4$}\label{sym10S}
Let $G$ denote the connected semisimple group of Dynkin type $\mathsf{E}_{6}$. Let $H\subseteq G$ be the connected subgroup with simple roots$$\beta_{1}=\alpha_{2},\beta_{2}=\alpha_{4},\beta_{3}=\frac{1}{2}(\alpha_{3}+\alpha_{5}),\beta_{4}=\frac{1}{2}(\alpha_{1}+\alpha_{6}).$$
Then $H\subseteq G$ is a subgroup of Dynkin type $\mathsf{F}_{4}$, which is unique up to conjugacy \cite[Thm.~11.1]{dynkin}. Moreover, we have $\omega_{2}=\varpi_{1}$ which follows from the expression of the fundamental weights in terms of the simple roots.

There is one possibility for $I$, $I=\{\beta_{1}\}$. Let $P\subseteq H$ denote the corresponding parabolic subgroup, i.e.~$P=P_{\{\beta_{2},\beta_{3},\beta_{4}\}}$. Then $L_{P}=\Sp(6)\times\bbC^{\times}$ and $\lap^{u}=\bbC^{15}$. Let $Q\subseteq G$ be the parabolic subgroup for which $L_{Q}$ has simple roots $S\smallsetminus \{\alpha_{2}\}$. It is clear that $L_{P}\subseteq L_{Q}$ and $P^{u}\subseteq Q^{u}$.
These observations suggest that the Luna diagram of $G/P$ is given by \cite[Case 7]{MR2346359}. One can also verify this by applying as usual Corollary~\ref{cor:morphisms} and Proposition~\ref{prop:PinH}.
We have
$$\Sigma(G/P)=\{\sigma_{1}=\alpha_{1}+\alpha_{3}, \sigma_{2}=\alpha_{2}+\alpha_{4},\sigma_{3}=\alpha_{3}+\alpha_{4},\sigma_{4}=\alpha_{4}+\alpha_{5},\sigma_{5}=\alpha_{5}+\alpha_{6}\}$$
and
$$\Delta(G/P)=\{D_{1},D_{2},D_{3},D_{4},D_{5},D_{6}\}$$
with Cartan matrix 
$$\begin{array}{r|rrrrr}&\sigma_1&\sigma_2&\sigma_{3}&\sigma_{4}&\sigma_{5}\\
\hline
D_{1}&1&0&-1&0&0\\
D_{2}&0&1&-1&-1&0\\
D_{3}&1&-1&1&-1&0\\
D_{4}&-1&1&1&1&-1\\
D_{5}&0&-1&-1&1&1\\
D_{6}&0&0&0&-1&1
\end{array}
$$
Our discussion above implies that
\[
\chi_{D_{2}}=-\omega_{2}^{P}=-\varpi_1.
\]
We find
\[
\begin{array}{lclclclcl}
\chi_{D_{2}}&=&\chi_{D_{3}}&=&\chi_{D_{4}}&=&\chi_{D_{5}}&=&-\varpi_{1},\\
& & & & \chi_{D_{1}}&=&\chi_{D_{6}}&=&0.
\end{array}
\]

\subsection{$\SL(n)\times\SL(n)/\diag(\SL(n))$}\label{sym11S} There are two possibilities: $I=\{\beta_{1}\}$ or $I=\{\beta_{n-1}\}$. The generators for the case $I=\{\beta_{1}\}$ have been calculated in \cite[Lemma 5.2]{MR4053617}. The second case is related to the first by Remark \ref{remark:subgroups}.

%%%%%%%%%%%%%%%%%%%%%%%%%%%%%%%%
%%%%%                                            %%%%%%%%%%%%%%%
%%%%%          NEW SECTION          %%%%%%%%%%%%%%%
%%%%%                                           %%%%%%%%%%%%%%%
%%%%%%%%%%%%%%%%%%%%%%%%%%%%%%%%  

\section{Non-symmetric cases}\label{s:sph}\label{s:nonsym}

\subsection{$\SL(p+q+2)/(\SL(p+1)\times \SL(q+1))$ with $p>q\geq 0$}\label{sph1S} There are four different cases, listed below. In each case one can derive the generators immediately from from the indicated sub-cases of \ref{sym2} by Remark~\ref{remark:H'}.
\begin{enumerate}
\item $2\le p<q$, $I=\{\beta_{1}\}$ or $I=\{\beta_{p}\}$, use \ref{sym2-2S}, 
\item $p=1<q$, $I=\{\beta\}$, use \ref{sym2-1S},
\item $p=1,q\ge4$, $I=\{\beta_{j}'\}$, $2\leq j \leq q-2$, use \ref{sym2-6S},
\item $p=0,q\ge1$, $I=S_{H}\smallsetminus\{\beta_{j}\}$, $1\leq j\leq q$, use \ref{sym2-7S}.
\end{enumerate}

\subsection{$\SO(4n+2)/\SL(2n+1)$ with $n\geq 1$}\label{sph2S}
We have the following possibilities for $I$: $I=\{\beta_{1}\}$ or $I=\{\beta_{2n}\}$. This case derives immediately from case~(\ref{sym6S}), by Remark~\ref{remark:H'}.

\subsection{$\Spin(9)/\Spin(7)$}\label{sph3S}
Let $G=\Spin(9)$ and let $H$ be the connected semisimple subgroup whose simple roots are $\beta_{1}=\eps_{3}+\eps_{4},\beta_{2}=\alpha_{2}$ and $\beta_{3}=\frac{1}{2}(\alpha_{1}+\alpha_{3})$. 
Then $H$ is isomorphic to $\Spin(7)$ and the isomorphism is given by the representation $\varpi_{3}+1$.

There is only one possibility for $I$, namely $I=\{\beta_{1}\}$. Using as usual Corollary~\ref{cor:morphisms} and Proposition~\ref{prop:PinH}, one checks that
$$\Sigma(G/P)=\{\sigma_{1}=\alpha_{1}+\alpha_{2},\sigma_{2}=\alpha_{2}+\alpha_{3},\sigma_{3}=\alpha_{3}+\alpha_{4},\sigma_{4}=\alpha_{4}\}$$
and
$$\Delta(G/P)=\{D_{1},D_{2},D_{3},D_{4}^{+},D_{4}^{-}\}$$
with Cartan pairing
$$\begin{array}{r|rrrr}&\sigma_1&\sigma_2&\sigma_{3}&\sigma_{4}\\
\hline
D_{1}&1&-1&0&0\\
D_{2}&1&1&-1&0\\
D_{3}&-1&1&1&-1\\
D_{4}^{+}&0&0&0&1\\
D_{4}^{-}&0&-2&0&1
\end{array}
$$

Consider the parabolic subgroup $Q=Q_{\{\alpha_{1},\alpha_{2},\alpha_{3}\}}$ of $G$. We have $P\subseteq Q$ and $Q$ is minimal for this inclusion. Indeed, the corresponding set of colors corresponding to the map $G/P\to G/Q$ is either $\{D_{1},D_{2},D_{3},D_{4}^{+}\}$ or $\{D_{1},D_{2},D_{3},D_{4}^{-}\}$. We note that the latter is not a parabolic subset of $\Delta(G/P)$, that the former is, and in fact that it is minimal.

Using that $\omega_{4}=\beta_{1}+\beta_{2}+\beta_{3}=\varpi_{1}$, we get
\[
\chi_{D_{4}^{-}}=-\omega^{P}_{4}=-\varpi_1.
\]
and we obtain
\[
\begin{array}{lclcl}
\chi_{D_{1}}&=&\chi_{D_{4}^{+}}&=&0,\\
\chi_{D_{2}}&=&\chi_{D_{3}}&=&-\varpi_{1}.
\end{array}
\]

%   9.4

\subsection{$\Spin(7)/\mathsf G_2$}
Let $G=\Spin(7)$ and let $H\subseteq G$ be the connected subgroup with simple roots $\beta_{1}=\frac{1}{3}(\alpha_{1}+2\alpha_{3})$ and $\beta_{2}=\alpha_{2}$. Then $H$ is simply connected of Dynkin type $\mathsf{G}_{2}$. In fact, $H$ has an irreducible $7$-dimensional representation $H\to\SO(7)$ of highest weight $\varpi_{1}$ and its lift $H\to\Spin(7)$ is the embedding we described above.

We have the following possibilities for $I$:
\begin{enumerate}
\item\label{sph4a} $I=\{\beta_{1}\}$,
\item\label{sph4b} $I=\{\beta_{2}\}$.
\end{enumerate}

The Luna diagrams for the respective homogeneous spaces, denoted here by $G/P_{1}$ and $G/P_{2}$, appear in \cite[\S3.4]{MR2769314}. We will prove in a moment which diagram corresponds to which case, and we will use the observation that $Q_{\{\alpha_1,\alpha_3 \}}$ has Levi subgroup of semisimple type $\mathsf A_1$, and the unipotent radical $Q^u_{\{\alpha_1,\alpha_3 \}}$ does not contain the irreducible $\SL(2)$-module of dimension $4$.

%  Case 1
\subsubsection{$I=\{\beta_{1}\}$}\label{sph4aS}
%\textbf{Case \ref{sph4a}:} $I=\{\beta_{1}\}$.
We have 
$$\Sigma(G/P_{1})=\{\sigma_{1}=\alpha_{1}+\alpha_{2},\sigma_{2}=\alpha_{2}+\alpha_{3},\sigma_{3}=\alpha_{3}\}$$
and 
$$\Delta(G/P_{1})=\{D_{1},D_{2},D_{3}^{+},D_{3}^{-}\}$$
with Cartan pairing
$$\begin{array}{r|rrr}&\sigma_1&\sigma_2&\sigma_{3}\\
\hline
D_{1}&1&-1&0\\
D_{2}&1&1&-1\\
D_{3}^{+}&0&0&1\\
D_{3}^{-}&-2&0&1
\end{array}
$$

We observe that $\{D_{2},D_{3}^{+}\}$ is the minimal parabolic subset, as there are no other parabolic subsets with two or less elements. This implies that the colors $D_{1},D_{3}^{-}$ are pull-backs of colors on $G/Q$ via $G/P_1\to G/Q$ which is given by the inclusion  $P_{1}=P_{\{\beta_{1}\}}\subseteq Q_{\{\alpha_1,\alpha_3\}}=Q$.

This also shows that these spherical roots and colors correspond to $P_1$ and not $P_2$. Indeed, the unipotent radical of $P_2$ contains the $4$-dimensional irreducible $\SL(2)$-module, whereas $Q_{\{\alpha_1,\alpha_2\}}$ doesn't.

We deduce
\[
\begin{array}{lclcl}
\chi_{D_{1}}&=&-\omega_{1}^{P}&=&-\varpi_{1},\\[5pt]
\chi_{D_{3}^{-}}&=&-\omega_{3}^{P}&=&-\varpi_{1},
\end{array}
\]
and this yields
\[
\begin{array}{lcl}
\chi_{D_{2}}&=&-\varpi_{1},\\
\chi_{D_{3}^{+}}&=&0.
\end{array}
\]

% Case 2
\subsubsection{$I=\{\beta_{2}\}$}\label{sph4bS}
%\textbf{Case \ref{sph4b}:} $I=\{\beta_{2}\}$.
We have 
$$\Sigma(G/P_{2})=\{\sigma_{1}=\alpha_{1},\sigma_{2}=\alpha_{2},\sigma_{3}=\alpha_{3}\}$$
and
$$\Delta(G/P_{2})=\{D_{1}^{+}=D_{2}^{+},D_{1}^{-}=D_{3}^{-},D_{2}^{-},D_{3}^{+}\}$$
with Cartan pairing
$$\begin{array}{r|rrr}&\sigma_1&\sigma_2&\sigma_{3}\\
\hline
D_{1}^{+}&1&1&-1\\
D_{1}^{-}&1&-2&1\\
D_{2}^{-}&-2&1&0\\
D_{3}^{+}&-1&0&1
\end{array}
$$

Since $D_{1}^{+}$ and $D_{1}^{-}$ are moved by two roots, they must be in the parabolic subset. Suppose that $\{D_{1}^{+},D_{1}^{-},D_{2}^{-}\}$ is parabolic. Then there exist non-negative coefficients $a_{1}^{+},a_{1}^{-},a_{2}^{-}$ such that $a_{1}^{+}D_{1}^{+}+a_{1}^{-}D_{1}^{-}+a_{2}^{-}D_{2}^{-}$ is strictly positive on each spherical root. In particular, $a_{1}^{-}>a_{1}^{+}$. The first two spherical roots yield $a_{1}^{+}+a_{1}^{-}-2a_{2}^{-}>0$ and $a_{1}^{+}-2a_{1}^{-}+a_{2}^{-}>0$, which implies $a_{1}^{+}<a_{1}^{-}$, a contradiction.

Since $3D_{1}^{+}+D_{1}^{-}+3D_{3}^{+}$ is strictly positive on each spherical root, we see that $\{D_{1}^{+},D_{1}^{-},D_{3}^{+}\}$ is a minimal parabolic subset of colors. This means that $Q=Q_{\{\alpha_2\}}$ contains $P_2$ up to conjugation, and that $D_{2}^{-}$ comes from $G/Q$.

Hence
\[
\begin{array}{lclcl}
\chi_{D_{2}^{-}}&=&-\omega_{2}^{P}&=&-\varpi_{2},
\end{array}
\]
and we obtain
\[
\begin{array}{lclcl}
\chi_{D_{1}^{+}}&=&\chi_{D_{1}^{-}}&=&-\varpi_{2},\\
& & \chi_{D_{3}^{+}}&=&0.
\end{array}
\]

% 9.5

\subsection{$\mathsf G_2/\SL(3)$}\label{sph5S}
Let $G$ be the connected group of Dynkin type $\mathsf{G}_{2}$ and let $H\subseteq G$ be the connected subgroup whose roots are the long roots of $G$. Set $\beta_{2}=\alpha_{2}$ and $\beta_{1}=s_{\omega_{2}}(\beta_{1})$. Then $H=\SL(3)$. There are two possibilities for $I$,
\begin{enumerate}
\item\label{sph8a} $I=\{\beta_{1}\}$,
\item\label{sph8b} $I=\{\beta_{2}\}$.
\end{enumerate}
The corresponding subgroups $P_{1}$ and $P_{2}$ are conjugated by a simple reflection of $W_{G}$. The Luna diagram of $G/P_{1}$ occurs as case G3 in \cite{MR1424449}. Denoting by $\alpha_1$ the short simple root, and by $\alpha_2$ the long one, one checks with the usual argument that $$\Sigma(G/P_{1})=\{\sigma_{1}=\alpha_{1},\sigma_{2}=\alpha_{1}+\alpha_{2}\}$$
and
$$\Delta(G/P_{1})=\{D_{1}^{+}, D_{1}^{-},D_{2}\}$$
with Cartan pairing
$$\begin{array}{r|rr}&\sigma_1&\sigma_2\\
\hline
D_{1}^{+}&1&0\\
D_{1}^{-}&1&-1\\
D_{2}&-1&1
\end{array}
$$
The minimal parabolic subset is $\{D_{1}^{+},D_{2}\}$ which implies that $D_{1}^{-}$ is a pull-back from the color of $G/Q_{\{\alpha_{2}\}}$, where $P_{1}=P_{\{\beta_{2}\}}\subseteq Q_{\{\alpha_{2}\}}$ is regular. It follows that
\[
\chi_{D_{1}^{-}}=-\omega_{1}^{P_{1}}=-\varpi_{1}
\]
and the equations yield
\[
\begin{array}{lcl}
\chi_{D_{2}}&=&-\varpi_{1}, \\
\chi_{D_{1}^{+}}&=&0.
\end{array}
\]

% 9.6

\subsection{$(\Sp(2m)\times \Sp(2n))/(\Sp(2m-2)\times \SL(2)\times \Sp(2n-2))$ with $m,n\geq1$ and $\min\{m,n\}>1$ } We have the following possibilities for $I$:
\begin{enumerate}
\item\label{sph6a} $I=\{\beta_i\}$ with $m\geq 2$ and $i\in\{1,\ldots,m-1 \}$,
\item\label{sph6b} $I=\{\beta_1'\}$.
\end{enumerate}
There is also a third possibility $I=\{\beta''_j \}$ with $n>1$ and $j\in\{1,\ldots,n-1 \}$, but it can be skipped here because it falls under the analysis of the first case by just swapping $m$ with $n$.

We recall the following data for $G/H$, taken from~\cite[Case 42]{MR3345829}. If $n,m>1$ we have
\[
\Sigma(G/H)=\{ \alpha_1+\alpha'_1, \alpha_1+2\alpha_{2,m-1}+\alpha_m, \alpha'_1+2\alpha'_{2,n-1}+\alpha'_n \}
\]
and
\[
\Delta(G/H)=\{ D_1=D_{1'}, D_2, D_{2'}\}
\]
with Cartan pairing given by the restriction to $\Xi(G/H)$ of respectively $\alpha_1^\vee$, $\alpha_2^\vee$, $(\alpha'_2)^\vee$. If $m=1$ or $n=1$, then the second (resp.\ third) spherical root and the second (resp.\ third) color does not appear.

\subsubsection{$I=\{\beta_i\}$ with $m>1$ and $i\in\{1,\ldots,m-1 \}$}\label{sph6aS}
%\textbf{Case \ref{sph6a}:} $I=\{\beta_i\}$ with $m>1$ and $i\in\{1,\ldots,m-1 \}$.
Let $G$ act naturally on $\CC^{2m}\oplus\CC^{2n}$, and let $V\cong\CC^2\oplus\CC^2$ be the subspace generated by the $m$-th and the $(m+1)$-th elements of the canonical basis of $\CC^{2m}$, and by the $n$-th and the $(n+1)$-th elements of the canonical basis of $\CC^{2n}$. We choose $H$ to be inside the stabilizer of $V$ in $G$, so that $H$ contains the diagonal subgroup of $\SL(2)\times\SL(2)$ (embedded in $\GL(V)$ in the obvious way).

We recall that, with this choice (and up to restricting weights from $T_G$ to $T_H$) we have $\omega_j=\varpi_j$ for all $j\in\{1,\ldots,m-1\}$, and $\omega_m=\varpi_{m-1}+\varpi'_1$.

As usual, we choose $P$ so that it contains the intersection $H\cap B_-$. We use notations similar to section~(\ref{sym7}): the parabolic subgroup $P$ will be denoted also by $P_{\ref{sph4a}}$ or $P_{\ref{sph4a}}(i)$, whereas that of case~(\ref{sph4b}) will be denoted also by $P_{\ref{sph4b}}$. With similar notation we consider the subgroup $Q_{\ref{sph4a}}(i)$, which is such that its Levi subgroup has all simple roots except for $\alpha_i$.

Let us discuss how one checks that $G/P_{\ref{sph4a}}(i)$ has the spherical roots and colors indicated below. The usual procedure assures that the given data correspond to parabolic subgroups of $H$. The inclusion of $P_{\ref{sph4a}}(i)$ in $Q_{\ref{sph4a}}(i)$ is then enough to conclude that the data we give for $P_{\ref{sph4a}}(i)$ does not correspond to $P_{\ref{sph4a}}(j)$ for any $j\neq i$.

It remains to exclude the possibility that the data given here corresponds to $P_{\ref{sph4b}}$. This is done by noticing that $P_{\ref{sph4b}}$ is contained in a proper parabolic subgroup of $G$ containing the first factor $\Sp(2m)$ of $G$, something that is not compatible with the data given here.

We analyze now all possible values of $i$, $m$, and $n$.

Suppose $i=m-1$, $m> 2$ and $n>1$. Then
\[
\Sigma(G/P_{\ref{sph4a}}(m-1))=\{\alpha_1, \sigma_1=\alpha_{2,m-1}, \alpha_m, \alpha_1', \sigma_2=\alpha'_{1}+2\alpha_{2,n-1}+\alpha_n\}
\]
and
\[
\Delta(G/P_{\ref{sph4a}}(m-1))=\{D_{1}^{+}=D_m^+, D_1^-=D_{1'}^-, D_{2},E_{m-1}, D_m^-=D_{1'}^+, D_{2'}\}
\]
with Cartan pairing
\[
\begin{array}{r|rrrrr|r}
&\alpha_1 & \sigma_1 & \alpha_m &\alpha_1' &\sigma_2\\
\hline
D_{1}^{+}& 1 & -1 & 1  & -1 & 0 \\
D_{1}^{-}& 1 & 0  & -1 & 1 &  0 \\
D_{2}    & -1 & 1 & 0 & 0 & 0 \\
E_{m-1}  & 0 & 1 & -2 & 0 & 0 \\
D_m^-    & -1 & 0 & 1 & 1 & 0 \\
D_{2'}   & 0 & 0 & 0 & -1 & 1 
\end{array}
\]
There is only one color not mapped dominantly to $G/Q_{\ref{sph4a}}(m-1)$, namely $E_{m-1}$. Indeed, if $m>3$, it's the only color moved by $\alpha_{m-1}$. If $m=3$ the other possibility given by $D_2$ does not correspond to a parabolic subset of colors. Thus
\[
\chi_{E_{m-1}} = -\omega_{m-1}^{P_{\ref{sph4a}}(m-1)} = -\varpi_{m-1},
\]
and this yields
\[
\begin{array}{lclclcl}
& & \chi_{D_{1}^+}&=&\chi_{D_{m}^-}&=&-\varpi_{m-1}, \\
\chi_{D_1^-}&=&\chi_{D_2}&=&\chi_{D_{2'}}&=&0.
\end{array}
\]

Suppose $i=m-1$, $m>2$, and $n=1$. Then
\[
\Sigma(G/P_{\ref{sph4a}}(m-1))=\{\alpha_1, \sigma=\alpha_{2,m-1}, \alpha_m, \alpha_1'\}
\]
and
\[
\Delta(G/P_{\ref{sph4a}}(m-1))=\{D_{1}^{+}=D_m^+, D_1^-=D_{1'}^+, D_{2},E_{m-1}, D_m^-=D_{1'}^-\}
\]
with Cartan pairing
\[
\begin{array}{r|rrrr}
&\alpha_1 & \sigma & \alpha_m &\alpha_1'\\
\hline
D_{1}^{+}& 1 & -1 & 1  & -1 \\
D_{1}^{-}& 1 & 0  & -1 & 1\\
D_{2}    & -1 & 1 & 0 & 0  \\
E_{m-1}  & 0 & 1 & -2 & 0  \\
D_m^-    & -1 & 0 & 1 & 1 
\end{array}
\]
We proceed exactly as in the previous case, obtaining
\[
\chi_{E_{m-1}} = -\omega_{m-1}^{P_{\ref{sph4a}}(m-1)} = -\varpi_{m-1},
\]
and
\[
\begin{array}{lclclcl}
& & \chi_{D_{1}^+}&=&\chi_{D_{m}^-}&=&-\varpi_{m-1}, \\
& & \chi_{D_1^-}&=&\chi_{D_2}&=&0.
\end{array}
\]

Suppose $i=1$, $m=2$ and $n\geq 2$. Then
\[
\Sigma(G/P_{\ref{sph4a}}(1))=\{\alpha_1,\alpha_2, \alpha_1', \sigma=\alpha'_{1}+2\alpha'_{2,n-1}+\alpha'_n\}
\]
and
\[
\Delta(G/P_{\ref{sph4a}}(1))=\{D_{1}^{+}, D_1^-=D_{1'}^-, D_{2}^+, D_2^-=D_{1'}^+, D_{2'}\}
\]
with Cartan pairing
\[
\begin{array}{r|rrrr}
&\alpha_1 & \alpha_2 &\alpha_1' &\sigma\\
\hline
D_{1}^{+}& 1 & -1  & -1 & 0 \\
D_{1}^{-}& 1 & -1 & 1 &  0 \\
D_{2}^+    & 0 & 1 & -1 & 0 \\
D_2^-    & -1 & 1 & 1 & 0 \\
D_{2'}   & 0 & 0 & -1 & 1 
\end{array}
\]
The only color not mapped dominantly to $G/Q_{\ref{sph4a}}(1)$ is $D_1^+$, which gives
\[
\chi_{D_{1}^+} = -\omega_{1}^{P_{\ref{sph4a}}(1)} = -\varpi_{1},
\]
and
\[
\begin{array}{lclclcl}
& & & & \chi_{D_2^-}&=&-\varpi_{1}, \\
\chi_{D_{1}^-}&=&\chi_{D_{2}^+}&=&\chi_{D_{2'}}&=&0.
\end{array}
\]

Suppose $i=1$, $m=2$, and $n=1$. Then
\[
\Sigma(G/P_{\ref{sph4a}}(1))=\{\alpha_1,\alpha_2, \alpha_1'\}
\]
and
\[
\Delta(G/P_{\ref{sph4a}}(1))=\{D_{1}^{+}, D_1^-=D_{1'}^-, D_{2}^+, D_2^-=D_{1'}^+\}
\]
with Cartan pairing
\[
\begin{array}{r|rrr}
&\alpha_1 & \alpha_2 &\alpha_1'\\
\hline
D_{1}^{+}& 1 & -1  & -1\\
D_{1}^{-}& 1 & -1 & 1\\
D_{2}^+    & 0 & 1 & -1\\
D_2^-    & -1 & 1 & 1
\end{array}
\]
We proceed exactly as in the previous case, obtaining
\[
\chi_{D_{1}^+} = -\omega_{1}^{P_{\ref{sph4a}}(1)} = -\varpi_{1},
\]
and
\[
\begin{array}{lclcl}
 & & \chi_{D_2^-}&=&-\varpi_{1}, \\
\chi_{D_{1}^-}&=&\chi_{D_{2}^+}&=&0.
\end{array}
\]

Suppose $1<i< m-1$ and $n>1$. Then
\[
\Sigma(G/P_{\ref{sph4a}}(i))=\{\alpha_1,\sigma_1=\alpha_{2,i}, \alpha_{i+1}, \sigma_2 = \alpha_{i+1}+2\alpha_{i+2,m-1}+ \alpha_m, \alpha_1',\sigma_3= \alpha'_{1}+2\alpha'_{2,n-1}+ \alpha'_n\}
\]
and
\[
\Delta(G/P_{\ref{sph4a}}(i))=\{D_{1}^{+}=D_{i+1}^+, D_1^-=D_{1'}^-, D_{2}, E_{i}, D_{i+1}^-=D_{1'}^+, D_{i+2}, D_{2'}\}
\]
with Cartan pairing
\[
\begin{array}{r|rrrrrr}
&\alpha_1 & \sigma_1 & \alpha_{i+1} &\sigma_2 &\alpha_1' & \sigma_3 \\
\hline
D_{1}^{+}& 1 & -1  & 1 & 0 & -1 & 0\\
D_{1}^{-}& 1 & 0 & -1 &  0 & 1 & 0\\
D_{2}    & -1 & 1 & 0 & 0 & 0 & 0\\
E_{i}  & 0 & 1 & -1 & -1 & 0 & 0\\
D_{i+1}^- & -1 & 0 & 1 & 0 & 1 & 0\\
D_{i+2}  & 0 & 0 & -1 & 1 & 0& 0 \\
D_{2'}  & 0 & 0 & 0 & 0 & -1& 1
\end{array}
\]

The only color not mapped dominantly to $G/Q_{\ref{sph4a}}(i)$ is $E_i$. This is obvious if $i>2$; if $i=2$ then the other possibility $D_2$ is excluded because the other colors do not form a parabolic subset.

We deduce
\[
\chi_{E_i} = -\omega_{i}^{P_{\ref{sph4a}}(i)} = -\varpi_{i},
\]
and
\[
\begin{array}{lclclclcl}
\chi_{D_1^+} &=& \chi_{D_{i+1}^-}&=&\chi_{D_{i+2}}&=&-\varpi_{i}, \\
\chi_{D_1^-} &=& \chi_{D_{2}}&=&\chi_{D_{2'}}&=&0.
\end{array}
\]

Suppose $1<i<m-1$ and $n=1$. Then
\[
\Sigma(G/P_{\ref{sph4a}}(i))=\{\alpha_1,\sigma_1=\alpha_{2,i}, \alpha_{i+1}, \sigma_2 = \alpha_{i+1}+2\alpha_{i+2,m-1}+ \alpha_m, \alpha_1' \}
\]
and
\[
\Delta(G/P_{\ref{sph4a}}(i))=\{D_{1}^{+}=D_{i+1}^+, D_1^-=D_{1'}^-, D_{2}, E_{i}, D_{i+1}^-=D_{1'}^+, D_{i+2}\}
\]
with Cartan pairing
\[
\begin{array}{r|rrrrr}
&\alpha_1 & \sigma_1 & \alpha_{i+1} &\sigma_2 &\alpha_1' \\
\hline
D_{1}^{+}& 1 & -1  & 1 & 0 & -1 \\
D_{1}^{-}& 1 & 0 & -1 &  0 & 1 \\
D_{2}    & -1 & 1 & 0 & 0 & 0 \\
E_{i}  & 0 & 1 & -1 & -1 & 0 \\
D_{i+1}^- & -1 & 0 & 1 & 0 & 1 \\
D_{i+2}  & 0 & 0 & -1 & 1 & 0
\end{array}
\]

We proceed exactly as in the previous case, obtaining
\[
\chi_{E_i} = -\omega_{i}^{P_{\ref{sph4a}}(i)} = -\varpi_{i},
\]
and
\[
\begin{array}{lclclclcl}
\chi_{D_1^+} &=& \chi_{D_{i+1}^-}&=&\chi_{D_{i+2}}&=&-\varpi_{i}, \\
& & \chi_{D_1^-} &=& \chi_{D_{2}}&=&0.
\end{array}
\]

Suppose $i=1$, $m\geq 3$, and $n>1$. Then
\[
\Sigma(G/P_{\ref{sph4a}}(i))=\{\alpha_1,\alpha_{2}, \sigma_1 = \alpha_{2}+2\alpha_{3,m-1}+ \alpha_m, \alpha_1', \sigma_2 = \alpha'_{1}+2\alpha'_{2,n-1}+ \alpha'_n\}
\]
and
\[
\Delta(G/P_{\ref{sph4a}}(i))=\{D_{1}^{+}=D_{1'}^+, D_1^-, D_{2}^+, D_{2}^-=D_{1'}^-,D_3, D_{2'}\}
\]
with Cartan pairing
\[
\begin{array}{r|rrrrr}
&\alpha_1 & \alpha_{2} & \sigma_1  &\alpha_1' &\sigma_2 \\
\hline
D_{1}^{+}& 1 & -1  & 0 & 1 & 0 \\
D_{1}^{-}& 1 & 0 & -1 &  -1 & 0\\
D_{2}^{+}& 0 & 1  & 0 & -1 & 0 \\
D_{2}^{-}& -1 & 1 & 0 & 1& 0 \\
D_{3} & 0 & -1 & 1 & 0& 0 \\
D_{2'} & 0 & 0 & 0 & -1& 1
\end{array}
\]

The only color not mapped dominantly to $G/Q_{\ref{sph4a}}(i)$ is $D_1^-$. This yields
\[
\chi_{D_1^-} = -\omega_{1}^{P_{\ref{sph4a}}(i)} = -\varpi_{1},
\]
and
\[
\begin{array}{lclclclcl}
& &  \chi_{D_{2}^-}&=&\chi_{D_{3}}&=&-\varpi_{1}, \\
\chi_{D_1^+} &=& \chi_{D_{2}^+}&=&\chi_{D_{2'}}&=&0.
\end{array}
\]

It remains the case $i=1$, $m\geq 3$, and $n=1$. Then
\[
\Sigma(G/P_{\ref{sph4a}}(i))=\{\alpha_1,\alpha_{2}, \sigma = \alpha_{2}+2\alpha_{3,m-1}+ \alpha_m, \alpha_1'\}
\]
and
\[
\Delta(G/P_{\ref{sph4a}}(i))=\{D_{1}^{+}=D_{1'}^+, D_1^-, D_{2}^+, D_{2}^-=D_{1'}^-,D_3\}
\]
with Cartan pairing
\[
\begin{array}{r|rrrr}
&\alpha_1 & \alpha_{2} & \sigma  &\alpha_1' \\
\hline
D_{1}^{+}& 1 & -1  & 0 & 1 \\
D_{1}^{-}& 1 & 0 & -1 &  -1\\
D_{2}^{+}& 0 & 1  & 0 & -1 \\
D_{2}^{-}& -1 & 1 & 0 & 1\\
D_{3} & 0 & -1 & 1 & 0
\end{array}
\]
We proceed exactly as in the previous case, obtaining
\[
\chi_{D_1^-} = -\omega_{1}^{P_{\ref{sph4a}}(i)} = -\varpi_{1},
\]
and
\[
\begin{array}{lclclclcl}
& &  \chi_{D_{2}^-}&=&\chi_{D_{3}}&=&-\varpi_{1}, \\
& & \chi_{D_1^+} &=& \chi_{D_{2}^+}&=&0.
\end{array}
\]
\subsubsection{$I=\{\beta_1'\}$}\label{sph6bS}
%\textbf{Case \ref{sph6b}: $I=\{\beta_1'\}$}.
We choose now a different embedding of $H$ into $G$. Let $V'\cong\CC^2\oplus\CC^2$ be the subspace generated by the first and last elements of the canonical basis of $\CC^{2m}$, and by the first and last elements of the canonical basis of $\CC^{2n}$. We choose $H$ to be inside the stabilizer of $V'$ in $G$, so that $H$ contains the diagonal subgroup of $\SL(2)\times\SL(2)$ (embedded in $\GL(V')$ in the obvious way).

We recall that, with this choice and up to restricting weights from $T_G$ to $T_H$, we have $\omega_1=\omega'_1=\varpi'_1$, $\omega_j=\varpi'_1+\varpi_{j-1}$ for all $j\in\{2,\ldots,m\}$, and $\omega'_j=\varpi'_1+\varpi''_{j-1}$ for all $j\in\{2,\ldots,n\}$.

In this case the subgroup $Q_{\ref{sph4b}}$ is such that its Levi subgroup has all simple roots except for $\alpha_1$ and $\alpha'_1$. To check that $G/P_{\ref{sph4b}}$ has the spherical roots and colors indicated below, one proceeds as for $G/P_{\ref{sph4a}}(i)$.

Suppose that $m,n>1$. Then we have
\[
\Sigma(G/P)=\{\alpha_1,\sigma=\alpha_1+2\alpha_{2,m-1}+\alpha_m,\alpha'_1, \sigma'=\alpha'_1+2\alpha'_{2,n-1}+\alpha'_n \}
\]
and
\[
\Delta(G/P)=\{D^+_1=D^+_{1'}, D_1^-,D_{1'}^-, D_2,D_{2'}\}
\]
with Cartan pairing
\[
\begin{array}{r|rrrr}& \alpha_1 & \sigma & \alpha_1'& \sigma'\\
\hline
D_{1}^{+}&1&0 & 1 & 0\\
D_{1}^{-}&1&0& -1 & 0\\
D_{1'}^{-}&-1&0& 1 & 0\\
D_{2}&-1&1& 0 & 0\\
D_{2'}&0&0& -1 & 1
\end{array}
\]
The colors not mapped dominantly to $G/Q_{\ref{sph4b}}$ are $D_1^-$ and $D_{1'}^-$. Therefore
\[
\begin{array}{lclcl}
\chi_{D_1^-} &=& -\omega_{1}^{P_{\ref{sph4b}}} &=& -\varpi'_{1},\\
\chi_{D_{1'}^-} &=& -(\omega'_{1})^{P_{\ref{sph4b}}} &=& -\varpi'_{1},\\
\end{array}
\]
and
\[
\chi_{D_{1}^+}=\chi_{D_{2}}=\chi_{D_{2'}}=0.
\]

Suppose now $m>n=1$. Then we have
\[
\Sigma(G/P)=\{\alpha_1,\sigma=\alpha_1+2\alpha_{2,m-1}+\alpha_m,\alpha'_1 \}
\]
and
\[
\Delta(G/P)=\{D^+_1=D^+_{1'}, D_1^-,D_{1'}^-, D_2\}
\]
with Cartan pairing
\[
\begin{array}{r|rrr}& \alpha_1 & \sigma & \alpha_1'\\
\hline
D_{1}^{+}&1&0 & 1 \\
D_{1}^{-}&1&0& -1 \\
D_{1'}^{-}&-1&0& 1 \\
D_{2}&-1&1& 0
\end{array}
\]
We proceed exactly as in the previous case, obtaining
\[
\begin{array}{lclcl}
\chi_{D_1^-} &=& -\omega_{1}^{P_{\ref{sph4b}}} &=& -\varpi'_{1},\\
\chi_{D_{1'}^-} &=& -(\omega'_{1})^{P_{\ref{sph4b}}} &=& -\varpi'_{1},\\
\end{array}
\]
and
\[
\chi_{D_{1}^+}=\chi_{D_{2}}=0.
\]

If $n>m=1$, then the computation is the same as the previous one (with $n$ and $m$ swapped), and we get
\[
\Delta(G/P)=\{D^+_1=D^+_{1'}, D_1^-,D_{1'}^-, D_{2'}\}
\]
with
\[
\begin{array}{lclcl}
\chi_{D_1^-} &=& -\omega_{1}^{P_{\ref{sph4b}}} &=& -\varpi'_{1},\\
\chi_{D_{1'}^-} &=& -(\omega'_{1})^{P_{\ref{sph4b}}} &=& -\varpi'_{1},\\
\end{array}
\]
and
\[
\chi_{D_{1}^+}=\chi_{D_{2'}}=0.
\]

\bibliographystyle{plain}
\bibliography{bibfileMVP}{}

\end{document}